\documentclass[12pt]{amsart}
\usepackage[T1]{fontenc}
\usepackage{amsmath, amssymb}
\usepackage{bm}
\usepackage[colorlinks,linktocpage, colorlinks=true,
            linkcolor=red,
            urlcolor=blue,
            citecolor=blue]{hyperref}
\usepackage{backref}

\usepackage[english]{babel}
\usepackage{mathrsfs}
\usepackage[all]{xy}
\usepackage{enumitem}
\usepackage{url}
\usepackage{todonotes}

\setlist[enumerate,1]{
  label=\textup{(\arabic*)},
  ref=\arabic*
}

\newtheorem{thm}{Theorem}[section]
\newtheorem*{thm*}{Theorem}
\newtheorem{lem}[thm]{Lemma}
\newtheorem{fact}[thm]{Fact}
\newtheorem{facts}[thm]{Facts}
\newtheorem{prop}[thm]{Proposition}
\newtheorem*{prop*}{Proposition}

\newtheorem{cor}[thm]{Corollary}

\theoremstyle{definition}
\newtheorem{defn}[thm]{Definition}
\newtheorem{notation}[thm]{Notation}
\newtheorem{remark}[thm]{Remark}
\newtheorem{remarks}[thm]{Remarks}

\newcommand{\norm}[1]{\left\lVert #1\right\rVert}
\newcommand{\dminus}{ 
\buildrel\textstyle\ .\over{\hbox{ 
\vrule height3pt depth0pt width0pt}{\smash-} 
}}
\numberwithin{equation}{section}

\def\bb{\mathbb}

\def\om{\omega}

\def\vp{\varphi}

\def\bb{\mathbb}

\def\sg{\sigma}

\def\cal{\mathcal}

\newcommand{\cstar}{$\mathrm{C}^*$}
\newcommand{\wstar}{$\mathrm{W}^*$}
\def\bb R{\mathbf R}

\def \md{\operatorname{mod}}

\newcommand{\ri}{\mathrm{i}}

\newcommand\ip[2]{\left\langle #1\, ,\!\ #2 \right\rangle}
\newcommand\cU{{\cal U}}
\newcommand\cL{{\cal L}}

\newcommand{\cM}{\mathcal{M}}
\newcommand{\cI}{\mathcal{I}}

\def \sa{\operatorname{sa}}
\def \rb{\operatorname{rb}}
\def \tb{\operatorname{tb}}

\def \rt{\operatorname{right}}

\def \vP{\mathbf{P}}
\def \vQ{\mathbf{Q}}
\def \vR{\mathbf{R}}

\newcounter{axiomlist}
\renewcommand{\theaxiomlist}{(A\arabic{axiomlist})}

\newcommand{\axiomtag}{%
  \refstepcounter{axiomlist}%
  \tag*{\theaxiomlist}%
}


\begin{document}


\title{Totally bounded elements in \wstar-probability spaces}

\author{H. Ando, J. Arulseelan, I. Goldbring, B. Hart, and T. Sinclair}
\thanks{I. Goldbring was partially supported by NSF grant DMS-2054477.
T. Sinclair was partially supported by NSF grant DMS-2055155.
B. Hart was supported by an NSERC DG}

\address{Department of Mathematics and Informatics, Chiba University, 1-33 Yayoi-cho,
Inage, Chiba 263-8522, Japan}
\email{hiroando@math.s.chiba-u.ac.jp}
\address{Department of Mathematics, Iowa State University, 396 Carver Hall, 411 Morrill Road, Ames, IA 50011, USA}
\email{jananan@iastate.edu}
\urladdr{https://sites.google.com/view/jananan-arulseelan}

\address{Department of Mathematics\\University of California, Irvine, 340 Rowland Hall (Bldg.\# 400),
Irvine, CA 92697-3875}
\email{isaac@math.uci.edu}
\urladdr{http://www.math.uci.edu/~isaac}

\address{Department of Mathematics and Statistics, McMaster University, 1280 Main St., Hamilton ON, Canada L8S 4K1}
\email{hartb@mcmaster.ca}
\urladdr{http://ms.mcmaster.ca/~bradd/}

\address{Mathematics Department, Purdue University, 150 N. University Street, West Lafayette, IN 47907-2067}
\email{tsincla@purdue.edu}
\urladdr{http://www.math.purdue.edu/~tsincla/}


\begin{abstract}
We introduce the notion of a totally ($K$-) bounded element of a \wstar-probability space $(M, \vp)$ and, borrowing ideas of Kadison, give an intrinsic characterization of the $^*$-subalgebra $M_{\tb}$ of totally bounded elements.  Namely, we show that $M_{\tb}$ is the unique strongly dense $^*$-subalgebra $M_0$ of totally bounded elements of $M$ for which the collection of totally $1$-bounded elements of $M_0$ is complete with respect to the $\|\cdot\|_\varphi^\#$-norm and for which $M_0$ is closed under all operators $h_a(\log(\Delta))$ for $a \in \mathbb{N}$, where $\Delta$ is the modular operator and $h_a(t):=1/\cosh(t-a)$ (see Theorem 4.3).  

We also prove that totally $K$-bounded elements of an Ocneanu ultraproduct admit representatives with the same total bound using a careful effective estimate of the distance of a given totally bounded element to the totally 1-bounded elements. An alternative proof in the appendix uses an isometric $H^\infty$-lifting theorem for the Ocneanu multiplier quotient,
derived from a metric $H^\infty$-lifting theorem for
$C^*$-quotients and a $C^*$-algebraic Schur parametrization.
We combine these results with Rieffel and Van Daele's bounded operator approach to modular theory to arrive at a new language and axiomatization of \wstar-probability spaces as metric structures.  Previous work of Dabrowski had axiomatized \wstar-probability spaces using a smeared version of multiplication, but the subalgebra $M_{\tb}$ allows us to give an axiomatization in terms of the original algebra operations.  Finally, we prove the (non-)axiomatizability of several classes of \wstar-probability spaces.
\end{abstract}

\maketitle

\tableofcontents

\section{Introduction}\label{SectionIntro}
First introduced in earnest in the papers \cite{FHS2} and \cite{FHS3}, the model theory of finite von Neumann algebras has been a very active area of research; see \cite{GH} for a survey of many of the results that have been obtained to date.  In order to be treated as a metric structure, one must equip a finite von Neumann algebra $M$ with a faithful, normal trace $\tau$; such a tracial von Neumann algebra (or tracial \wstar-probability space) is then treated as a structure in a many-sorted language, where the sorts are operator norm bounded balls of positive integer radii, equipped with the metric coming from the trace.  Each sort is then equipped with the restriction of the $*$-algebra structure; it is important to note that the proof that the algebra multiplication is uniformly continuous makes use of the tracial property of the state.   

It is natural to try to extend the model-theoretic treatment of von Neumann algebras beyond the finite case.  The next obvious class of von Neumann algebras to consider is the $\sigma$-finite ones.  Since such algebras can be characterized as those which admit a faithful, normal state, one can try to mimic the case of finite von Neumann algebras by treating \wstar-probability spaces $(M,\vp)$ as metric structures in the same way as in the case of tracial \wstar-probability spaces.  This time, one runs into the immediate issue that the algebra multiplication restricted to operator norm bounded balls may no longer be uniformly continuous with respect to the metric coming from the state, whence one must find a different approach.  In \cite{Dab}, Dabrowski axiomatized the class of \wstar-probability spaces in a language that still used operator norm bounded balls as sorts, but which had to give up ordinary algebra multiplication in favor of a ``smeared multiplication'' obtained through an averaging procedure involving the modular automorphism group.

In this paper, we offer an alternative axiomatization of the class of \wstar-probability spaces motivated by the desire to have ordinary multiplication as part of the language.  As a consequence, a compromise must be made in terms of the sorts.  Since the issue in adapting the language for tracial \wstar-probability spaces to the non-tracial case stemmed from the fact that multiplication is not uniformly continuous with respect to the metric coming from the state, it becomes natural to change the sorts in such a way that multiplication does become uniformly continuous.  Consequently, we take as our sorts the set of \textbf{totally $K$-bounded} elements of the \wstar-probability space (as $K$ ranges over the natural numbers); here an element $x\in M$ is totally $K$-bounded if both left and right multiplication by both $x$ and $x^*$, viewed as operators on the dense subspace $M\omega_\vp$ of the GNS-Hilbert space $H_\varphi$, extend to bounded operators on $H_\vp$ of operator norm at most $K$.  (Here, $\om_\vp$ is the vector in $H_\vp$ corresponding to the identity in $M$.)  It is readily verified that all of the $*$-algebra operations, restricted to the set of totally $K$-bounded elements, are uniformly continuous with respect to the norm $\| \cdot \|^\#_\vp$ arising from the state (see Section~\ref{SectionPrelim} for the definition), which topologizes the strong* topology on operator norm bounded sets.  
Finally, since the set of totally bounded vectors can be shown to be strongly dense in the algebra, one can feel confident that one has not lost any information by restricting to such elements.

The proof that the category of models of our theory is equivalent to the category of \wstar-probability spaces is, in itself, mathematically interesting.  Indeed, to establish this equivalence of categories, one needs to show, in particular, that if one begins with a model $A$ of our theory, then forms the corresponding \wstar-probability space $(M, \vp)$ obtained as the strong closure of the direct limit of the sorts, and then ``dissects'' again to obtain a model $A' $, then naturally $A$ is isomorphic to $A'$.  In the model-theoretic treatments of \cstar-algebras and of tracial von Neumann algebras, the analogous step is quickly dispensed with by standard functional calculus techniques.  The aforementioned functional calculus arguments become problematic in the non-tracial setting, however, due to the necessity of handling left and right norms simultaneously.  More to the point, it is a priori possible that after completing $A$, $(M,\vp)$ might have new totally bounded elements compared to those of $A$, whence $A'$ would be strictly larger than $A$.  In order to show that this does not, in fact, happen, in Section \ref{tb} we use the ideas of Kadison \cite{Kadison} to obtain a novel characterization of the totally bounded elements of a \wstar-probability space, a result which we believe is of independent interest.  In addition, to be able to axiomatize the fact that the elements of the sort for $K$-bounded elements that are in fact $1$-bounded come from elements of the $1$-bounded sort, we prove (Theorem~\ref{thm:SK-liftingnew}) that a totally $K$-bounded element in an Ocneanu ultraproduct can be represented by a sequence of totally $K$-bounded elements from the individual algebras. The proof of this result is surprisingly delicate, and it relies on Theorem~\ref{thm:effective-lower-sort-test}, which will also be one of the key ingredients to the effective axiomatization of W$^*$-probability spaces in Section~\ref{SectionLanguage}. In the appendix, we present an alternative, completely different proof of Theorem~\ref{thm:SK-liftingnew}. This time, we use an isometric $H^\infty$-lifting theorem for the Ocneanu multiplier quotient, which in turn relies on an operator-valued Schur parameterization and a metric $H^\infty$-lifting theorem for \cstar-quotients; we believe that these intermediate results are of independent interest and that they will clarify the significance of the role played by the analytic continuation of the modular orbit. 

When axiomatizing \cstar-algebras and tracial von Neumann algebras as model-theoretic structures using operator norm balls as sorts, it is an important step to be able to express the fact that the elements in the operator norm ball of radius $n$ whose operator norm is actually at most $1$ must be in the image of the canonical inclusion mapping from the ball of radius $1$ to the ball of radius $n$.  Such an axiomatization can be given using a standard functional calculus argument.  We would like to express the analogous fact in our current setting.  Unfortunately, such a simple functional calculus argument is not available to us in the non-tracial setting.  

In Section~\ref{sec:effectivealphaandSKdefinability}, we prove an effective distance estimate that serves as a substitute for this functional calculus argument, and use it to prove Theorem~\ref{thm:SK-liftingnew}.
For reasons connected with recent applications of model-theoretic ideas to computability concerns in operator algebras \cite{AGH}, having such an effective axiomatization is crucial. 

In Section~\ref{SectionLanguage} we then describe our axioms and carry out the proof of the aforementioned equivalence of categories; in fact, the equivalence is strengthened to be an equivalence of ultracategories in the sense that, under the above equivalence of categories, the model-theoretic ultraproduct corresponds to the Ocneanu ultraproduct.

A crucial tool in the study of non-tracial \wstar-probability spaces is the \textbf{modular automorphism group} $\sigma_t^\vp$ of the state $\vp$.  In Dabrowski's approach, the modular automorphism group is simply part of the language. 
In Section~\ref{SectionExpanding}, we prove that the modular
automorphisms are definable by constructing sort-uniform polynomial
approximations in the Rieffel--Van Daele operator.  We also give
concrete effective axioms for the expansion by rational-time
modular automorphisms.

In the final section of the paper, we mention some (non-)axiomatizability results for various classes of factors viewed as structures in our language.  Most of these results follow from results of one of the authors and Haagerup \cite{AH} viewed through a model-theoretic lens.

We assume familiarity with basic von Neumann algebra theory, although we review some basic facts we need about \wstar-probability spaces in  Section~\ref{SectionPrelim}.
 We also assume familiarity with basic continuous model theory; the reader may consult \cite{munster} or \cite{FHS2} for introductions to continuous model theory as it pertains to operator algebras.  Although not essential, familiarity with the axiomatization of tracial von Neumann algebras given in \cite{FHS2} (see also \cite{GH}) will help the reader appreciate the nuances encountered when moving to the $\sigma$-finite setting. For Tomita-Takesaki theory, \cite{TakesakiII} is a standard reference.

 \section*{AI Usage Statement}
The authors used ChatGPT 5.5 Pro/5.6 Pro and Claude Opus 5 at a late stage in the research in order to find strategies to fix incomplete earlier proofs. In particular, the following ideas were suggested by ChatGPT/Opus: the revised proof of Theorem~\ref{thm:SK-liftingnew} via Theorem~\ref{thm:effective-lower-sort-test}, the idea of Schur's algorithm to complete the proof of Theorem \ref{thm:almost-isometric-Cstar-Hinfty-lift}, a cut-off argument that helped to complete the proof of Theorem \ref{thm:recovermodular}. The final statements and proofs are reorganized by the authors, who are responsible for the content.  
\section*{Acknowledgements}
The third author would like to thank the Institut Henri Poincar\' e for their hospitality during the 2018 trimester on Model theory, valued fields, and combinatorics, where some of this work took place. 
The authors thank the anonymous referee for pointing out a number of errors (both small and large) in a previous version of the paper.

\section{Preliminaries on \wstar-probability spaces}\label{SectionPrelim}

\subsection{Standard forms}

As mentioned in the introduction, the class of structures we wish to capture model-theoretically are the $\sigma$-finite von Neumann algebras, by which we mean those von Neumann algebras $M$ for which any set $(p_i)_{i\in I}$ of pairwise orthogonal projections from $M$ contains only countably many nonzero elements.  It is well-known that these algebras can be alternatively characterized as those which admit a faithful normal state; we will use this characterization to view a $\sigma$-finite von Neumann algebra as a metric structure.  

However, it will behoove us to consider yet another characterization of this class.  Fix a faithful, normal representation $\pi\colon M\to \cal B(H)$ of $M$.  We recall that $\om\in H$ is called \textbf{cyclic} for $\pi$ if $\pi(M)\om$ is dense in $H$ and \textbf{separating} for $\pi$ if $\om$ is cyclic for $\pi(M)'$, that is, if $\pi(a)\om\not=0$ for all $a\in M\setminus \{0\}$.  $\pi$ is then called a \textbf{standard representation} if it admits a cyclic and separating unit vector $\om$, in which case one says that $(M,\pi,\om)$ is a \textbf{standard form} of $M$.  We remark that this is an abbreviated notation; the standard form in the sense of \cite{Haa} consists of a quadruple $(M,H,J,P)$, where $H$ is a Hilbert space on which $M$ acts, $J$ is an isometric antilinear involution, and $P$ is a self-dual cone in $H$ satisfying axioms of the standard form. Since we do not need the precise axioms or the uniqueness of the standard form, we do not mention $J$ or $P$. They can be constructed from $(M,\pi,\omega)$.   It is convenient to abuse notation and speak simply of $(M,\om)$ as a standard form of $M$, conflating the distinction between $M$ and $\pi(M)$.  One then has that $M$ is $\sigma$-finite if and only if it admits a standard representation; moreover, any two such standard representations of $M$ are unitarily conjugate.

If $(M,\om)$ is a standard form of $M$, then the state $\varphi_\om(a):=\langle a\om,\om\rangle$ is a faithful, normal state on $M$.  Conversely, given a faithful, normal state $\vp$ on $M$, one obtains a standard form of $M$ using the GNS construction, which we briefly recall.  First, one introduces the inner product $\ip{\cdot}{\cdot}_\vp$ on $M$ given by
\[
\ip{a}{b}_\vp = \vp(b^*a)
\]
with corresponding norm
\[
\|a\|_\vp = \sqrt{\ip{a}{a}_\vp}.
\] 
The Hilbert space completion of $(M,\ip{\cdot}{\cdot}_\vp)$ is denoted by $H_\vp$ and elements of $M$, when viewed as elements of $H_\vp$, are often decorated with a hat, so $a\in M$ would be denoted $\hat a$.  Given $a\in M$, one gets a bounded operator $\pi_\varphi(a)\in \cal B(H_\vp)$ defined on the dense subset $\widehat{M}$ of $H_\vp$ by setting $\pi_\varphi(a)(\hat{b}):=\widehat{ab}$.  It is then easy to see that $\pi_\vp:M\to B(H_\vp)$ is a standard representation of $M$ with cyclic and separating vector $\om_\vp:=\hat{1}$.  

We note also that the topology induced by $\|\cdot\|_\vp$ agrees with the strong operator topology on operator-norm-bounded subsets of $M$; closed operator-norm balls are complete for the corresponding metric.

\subsection{Totally bounded elements}
\label{subsec:tb}
Until further notice, fix a standard form $(M,\om)$ of $M$.  In general, for $a\in M$, the right multiplication map $b\om\mapsto ba\om:M\om\to M\om$ does not extend to a bounded operator on $H$.  Thus, we say that $a \in M$ is \textbf{right $K$-bounded} if, for all $b \in M$, we have $\|ba\omega\|\leq K\|b\omega\|$.  Note that some sources in the literature would call such an element right $\sqrt{K}$-bounded.  Note also that the notion of bounded element depends on the choice of standard form of $M$.  If $a$ is right $K$-bounded, then the above map extends to a bounded operator on $H$, which we denote by $\pi'(a)$; in this case, we set $\|a\|_{\rt}:=\|\pi'(a)\|$ and note that $\|a\|_{\rt}\leq K$. In a slight abuse of notation, for $\xi\in H$ we will sometimes write $\|\pi'(\xi)\|\leq K$ if $M\om\ni a\om\mapsto a\xi$ extends to a $K$-bounded operator on $H$. In this way we have $\pi'(a) = \pi'(a\om)$. We say that $a\in M$ is \textbf{right bounded} if it is right $K$-bounded for some $K$; we let $M_{\rb}$ denote the set of right bounded elements of $M$.  

Note that if $a\in M_{\rb}$, then $\pi'(a)\in M'$.  More generally, we have:

\begin{lem}\label{boundedlemma}
For $v\in H$, we have that $v\in M'\om$ if and only if the map $$b\om\mapsto bv:M\om\to H$$ extends to a bounded operator on $H$.  In particular, $a\in M$ is right-bounded if and only if $a\om\in M'\om$, in which case $\pi'(a)$ is the unique $z\in M'$ such that $a\om=z\om$.
\end{lem}

\begin{proof}
First suppose that $v\in M'\om$.  Then $v=x\om$ for some $x\in M'$.  For $b\in M$, it follows that $bv=bx\om=xb\om$, whence the map $b\om\mapsto bv$ is the restriction of the bounded operator $x$ to $M\om$.

Conversely, suppose that the map $b\om\mapsto bv:M\om\to H$ extends to the bounded operator $x$ on $H$.  Then $v=1\cdot v=x(1\om)=x\om$.  It suffices to show that $x\in M'$.  To see this, take $a,b\in M$ and note $(xa)(b\om)=x(ab\om)=abv=ax(b\om)$; since $M\om$ is dense in $H$, it follows that $ax=xa$.  Since $a\in M$ is arbitrary, it follows that $x\in M'$.

The final statement follows immediately from the first statement and its proof.
\end{proof}

Continuing with the assumption that $(M,\om)$ is a standard form of $M$, we say $x\in M$ is \textbf{totally $K$-bounded} if $\|x\| \leq K$ and both $x$ and $x^*$ are right $K$-bounded.  We say that $x\in M$ is \textbf{totally bounded} if $x$ is totally $K$-bounded for some $K$. In this case, we define the totally bounded norm of $x$ by 
\[\|x\|_{\mathrm{tb}}=\max\{\|x\|,\|R_x\|,\|R_{x^*}\|\},\]
where we use the notation $R_x:=\pi'(x)$.
 We will denote the set of totally bounded elements by $M_{\tb}$.  

\begin{notation}
For any $X \subseteq M$, we let $S_K(X)$ denote the set of totally $K$-bounded elements of $X$.
\end{notation}
\begin{remark}
    The notation above implicitly identifies $M$ with a choice of standard form.  More precisely, one should write $S_K(X,\om)$.  However, since we always fix a choice of cyclic separating vector, it is safe to suppress it.
\end{remark}

Now suppose that $(M,\vp)$ is a \wstar-probability space, by which we mean a von Neumann algebra equipped with a faithful normal state $\vp$.  We always associate to such a space the standard form $(M,\om_\vp)$ stemming from the GNS construction as discussed above, and any mention of (totally) bounded elements is relative to this standard form on $M$.  In this case, we often write $S_K(M,\varphi)$ to denote the set of totally $K$-bounded elements of $(M,\varphi)$.

We will also use the following characterization of totally bounded elements that is essentially due to Connes~\cite[Lemma 4]{Connes2}; see also \cite[Lemma 2.11]{AG2}. Recall
that the domain $D(\sigma^{\varphi}_{\ri/2})$ of $\sigma^{\varphi}_{\ri/2}$ is the set of all
$a\in M$ such that the map $F\colon\mathbb{R}\ni t\mapsto\sigma^{\varphi}_t(a)\in M$ extends
to a bounded, $\sigma$-weakly continuous map
$F\colon\{z\in\mathbb{C}\mid 0\le\mathrm{Im}\,z\le\tfrac{1}{2}\}\to M$ which is
norm-holomorphic in the interior. (Here, $\sigma_t^\varphi$ is the modular automorphism group of $(M,\varphi)$; see Section~\ref{SectionRvD} for the definition.)  Note that $F$ need not be norm-continuous on the boundary.
The domain $D(\sigma^{\varphi}_{-\ri/2})$ is defined analogously. By~\cite{CioZsi},
$\sigma^{\varphi}_{\pm \ri/2}$ are closed operators. For details on the analytic generators,
we refer the reader to~\cite{CioZsi, Haa79}. 

\begin{lem}\label{lem:totally-bounded-char}
Let $(M,\varphi)$ be a $W^*$-probability space. For $a\in M$, the following are equivalent:
\begin{enumerate}
  \item $a$ is totally bounded;
  \item $a\in D(\sigma^{\varphi}_{\ri/2})\cap D(\sigma^{\varphi}_{-\ri/2})$.
\end{enumerate}
Moreover, whenever these conditions hold,
\[
  \|a\|_{\mathrm{right}} = \|\sigma^{\varphi}_{-\ri/2}(a^*)\| = \|\sigma^{\varphi}_{\ri/2}(a)\|,
  \quad
  \|a^*\|_{\mathrm{right}} = \|\sigma^{\varphi}_{-\ri/2}(a)\|.
\]
Consequently,
\[
  S_K(M,\varphi) = \Bigl\{a\in D(\sigma^{\varphi}_{\ri/2})\cap D(\sigma^{\varphi}_{-\ri/2})
  \ :\  \|a\|,\;\|\sigma^{\varphi}_{\ri/2}(a)\|,\;\|\sigma^{\varphi}_{-\ri/2}(a)\|\le K\Bigr\}.
\]
\end{lem}

A consequence of the preceding fact is the following:

\begin{prop}\label{phidense}
Suppose that $(M,\varphi)$ is a \wstar-probability space.  Then $M_{\tb}$ is strongly dense in $M$.
\end{prop}  

\begin{proof}
If $x$ is entire $\sigma^{\varphi}$-analytic, then it is totally bounded by Lemma \ref{lem:totally-bounded-char}.  It remains to note that, by Gaussian smoothing, analytic elements are $\sigma$-strongly dense in $M$ (see \cite[Lemma 2.2]{CioZsi} where density in the $\sigma$-weak topology was proved; since analytic elements form a convex set, $\sigma$-strong density also follows).
\end{proof}

If $(M,\varphi)$ is a \wstar-probability space, we also consider the norm $\|\cdot\|_\vp^\#$ on $M$ given by

\[
\|a\|_\vp^\#:=\sqrt{\frac{\vp(a^*a+aa^*)}{2}}=\sqrt{\frac{\|a\|_\vp^2+\|a^*\|_\vp^2}{2}}.
\]

Note that when $\vp$ is a trace, then $\|\cdot\|_\vp=\|\cdot\|_\vp^\#$.  It is clear that $\|a^*\|_\vp^\#=\|a\|_\vp^\#$ and $\|a\|_\vp^\#\leq \|a\|$ for all $a\in M$.

Since the norm $\|\cdot\|_\vp$ induces the strong topology on operator-norm-bounded subsets of $M$, it follows that $\|\cdot\|_\vp^\#$ induces the strong-* topology on such subsets.
In particular, closed operator-norm balls are
$\|\cdot\|_\vp^\#$-complete.

\begin{lem}\label{statecont}
For all $a\in M$, we have $|\vp(a)|\leq \sqrt{2}\|a\|_\vp^\#$.
\end{lem}

\begin{proof}
The statement follows immediately from Cauchy-Schwarz.
\end{proof}

\begin{lem}\label{hashtagfacts}
Suppose that $(M,\vp)$ is a \wstar-probability space.  Then for any $K$, the strong and strong* topologies agree on totally $K$-bounded sets.
\end{lem}
\begin{proof}
Let $(a_i)_{i\in I}$ be a net in $S_K(M,\varphi)$ and $a\in S_K(M,\varphi)$. 
Then $d_i:=a_i-a\in S_{2K}(M,\varphi)$ and  
$$ \|d_i^*\|_\varphi^2
 =\varphi(d_id_i^*)=\langle R_{d_i^*}(d_i\omega_\vp),\omega_\vp\rangle
 \le 2K\|d_i\|_\varphi.$$
Therefore, 
\[
 \|a-a_i\|_{\varphi}^{\#\,2}
 \le \frac{\|a-a_i\|_\varphi^2
              +2K\|a-a_i\|_\varphi}{2}.
\]
The claim then follows from the above inequality. 
\end{proof}

As stated in the introduction, when treating a \wstar-probability space as a many-sorted structure, our intended sorts are the $S_K(M)$'s.  We will equip these sorts with the metric stemming from the norm $\|\cdot\|_\vp^\#$.  Consequently, we will need to know that our sorts are complete with respect to this metric, a fact we now prove.  

First, we need to introduce some important operators associated to any standard form.  Suppose that $(M,\om)$ is a standard form of $M$.  One then defines the unbounded operators $S_0:M\om\to H$ and $F_0:M'\om\to H$ by $S_0(a\om):=a^*\om$ and $F_0(b\om):=b^*\om$ for $a\in M$ and $b\in M'$.  An important fact about these unbounded operators are that they are closable \cite[VI, Lemma 1.5]{TakesakiII}; we set $S:=\overline{S_0}$ and $F:=\overline{F_0}$.  One also has that $S^* = F$, $S^2\xi=\xi\ \text{ for }\xi\in D(S)$ and $ F^2\eta=\eta \text{ for }\eta\in D(F)$. Given a W$^*$-probability space, we will denote by $S_\vp$ and $F_\vp$ the associated operators on the GNS construction.

\begin{lem}
For each $K$, the set $S_K(M,\varphi)$ is $\|\cdot\|_\vp^\#$-complete.
\end{lem}
\begin{proof}
Suppose that $(a_n)_{n\in \mathbb N}$ is a $\|\cdot\|^\#_\vp$-Cauchy sequence from $S_K(M,\varphi)$.  Since closed operator norm bounded balls are $\|\cdot\|_\varphi$-complete, there is $a\in M$ such that $a_n$ strongly converges to $a$.  We claim that $a$ is also right $K$-bounded.  To see this, take $b\in M$ and note that $ba_n$ converges strongly to $ba$ and hence $\|ba_n\|_\vp$ converges to $\|ba\|_\vp$.  Since $\|a_n\|_{\rt}\leq K$, we have that $\|ba\|_\vp=\lim_{n\to \infty}\|ba_n\|_\vp\leq K\|b\|_\vp$, whence $\|a\|_{\rt}\leq K$. 

Arguing as in the previous paragraph, there is $a' \in M$ such that $a_n^*$ strongly converges to $a'$; moreover, $\|a'\|,\|a'\|_{\rt}\leq K$. 

It remains to show that $a^* = a'$. Indeed, we have that $(a_n-a)\om\to 0$ and $S_{\vp}(a_n - a)\om = (a_n^* - a^*)\om \to (a' - a^*)\om$.  Since $S_{\vp}$ is closed, it follows that $(a' - a^*)\om = S_{\vp}0 = 0$ and thus $a'=a^*$.
\end{proof}

\subsection{The category of W$^*$-probability spaces}
We now describe embeddings between probability spaces.  An \textbf{embedding} between \wstar-probability spaces $(M,\vp)$ and $(N,\psi)$ is a unital, normal, state-preserving $^*$-homomorphism $f:M\to N$ for which there is a normal, state-preserving conditional expectation map $E:N\to f(M)$.

The reason for this choice of embeddings comes from modular theory (see the next lemma).  Notice that in the tracial setting, every trace-preserving $^*$-homomorphism automatically induces a conditional expectation onto its image.  This corresponds to the fact that, in this setting, the modular automorphism group is trivial.

In order to show that we axiomatized the category of \wstar-probability spaces, we will need the following alternate characterization of embeddings between \wstar-probability spaces.  Note that this characterization refers to the modular automorphism group of a \wstar-probability space, defined in detail in Section \ref{SectionRvD}.  Throughout this paper, we let $M_{\sa}$ denote the set of self-adjoint elements of $M$.

\begin{lem}\label{expectation}
    Suppose that $(M,\vp)$ and $(N, \psi)$ are \wstar-probability spaces and $f:M\to N$ is a state-preserving $^*$-homomorphism. Then the following are equivalent:
    \begin{enumerate}
    \item $f$ is an embedding of \wstar-probability spaces.
    \item $f$ preserves distances to the sets of self-adjoint elements, that is, for all $a\in M$, we have $d_\vp(a,M_{\sa}) = d_\psi(f(a),N_{\sa})$.
    \item $f$ commutes with the modular automorphism group, that is, for every $t \in \mathbb R$ and $a\in M$, we have $f(\sigma_t^\vp(a)) = \sigma_t^\psi(f(a))$. 
    \end{enumerate}
\end{lem}

\begin{proof}
For the implication (1) implies (2), note that if $E:N\to M$ is the state-preserving conditional expectation, then
$E$ is the orthogonal projection $L^2(N,\psi)\to L^2(M,\vp)$ and hence the distance
to $M_{\sa}\om$ computed in $M$ equals the distance computed in $N$.

The direction (2) implies (3) follows from Rieffel and Van Daele's bounded operator approach to Tomita-Takesaki theory described in Section \ref{SectionRvD}.  Indeed, using the notation developed there, if distances to the real self-adjoint subspaces agree, then the real
orthogonal projection $P_N$ restricts to $P_M$.  The relation
$iP(\xi)=Q(i\xi)$ gives the same conclusion for $Q$, and hence for $R=P+Q$.  Item (3) follows immediately from this fact.

The direction (3) $\implies$ (1) is a theorem of Takesaki \cite{Tak72expectation} (proof can also be found in \cite[Chapter IX, Theorem 4.2]{TakesakiII}). 
\end{proof}

In connection with the previous lemma, we note that if $a \in S_K(M)$, then $d_\vp(a,M_{\sa})=d_\vp(a,b)$ for some $b\in S_{3K}(M)$; this fact will be proven in Section \ref{SectionRvD} (see Corollary \ref{PpreservesRb}).

\subsection{Ocneanu and Groh--Raynaud ultraproducts}\label{Ocneanupreliminaries}

Throughout this section, we fix an index set $I$,
$(\bm{M},\bm{\varphi})=(M_i,\varphi_i)_{i\in I}$ a family of $W^*$-probability spaces, and a
non-principal ultrafilter $\mathcal{U}$ on $I$. $\ell^{\infty}(\bm{M})$ is the
\cstar-algebra of all uniformly bounded $I$-indexed families
$x=(x_i)_{i\in I}\in\prod_{i\in I}M_i$ equipped with the norm
\[
  \|x\| = \sup_{i\in I}\|x_i\|,\quad x=(x_i)_{i\in I}\in\ell^{\infty}(\bm{M}).
\]
Let
\[
  \mathcal{L}_{\mathcal{U}}
  = \mathcal{L}_{\mathcal{U}}(\bm{M},\boldsymbol{\varphi})
  = \Bigl\{x=(x_i)_{i\in I}\in\ell^{\infty}(\bm{M})
    \Bigm| \lim_{i\to\mathcal{U}}\|x_i\|_{\varphi_i}=0\Bigr\},
  \quad
  \mathcal{I}_{\mathcal{U}} = \mathcal{L}_{\mathcal{U}}\cap\mathcal{L}^*_{\mathcal{U}},
\]
and
\[
  \mathcal{M}^{\mathcal{U}}
  = \mathcal{M}^{\mathcal{U}}(\bm{M},\boldsymbol{\varphi})
  = \{x\in\ell^{\infty}(\bm{M})\mid
    x\mathcal{I}_{\mathcal{U}}\subset\mathcal{I}_{\mathcal{U}},\;
    \mathcal{I}_{\mathcal{U}}x\subset\mathcal{I}_{\mathcal{U}}\}.
\]

$\mathcal{M}^{\mathcal{U}}$ is a $C^*$-algebra (with pointwise operations and supremum norm)
in which $\mathcal{I}_{\mathcal{U}}$ is a closed two-sided ideal. We then define the
\textbf{Ocneanu ultraproduct} \cite{Ocneanu,AH}
\[
  \prod_{\mathcal{U}}(M_i,\varphi_i) = \mathcal{M}^{\mathcal{U}}/\mathcal{I}_{\mathcal{U}}
\]
(the quotient \cstar-algebra). The quotient map is denoted by
\[
  q\colon\mathcal{M}^{\mathcal{U}}\to\prod_{\mathcal{U}}(M_i,\varphi_i),
  \quad q((x_i)_i) =: (x_i)_{\mathcal{U}}.
\]
The formula
\[
  \varphi((x_i)_{\mathcal{U}}) = \lim_{i\to\mathcal{U}}\varphi_i(x_i)
\]
defines a faithful normal state on $M=\prod_{\mathcal{U}}(M_i,\varphi_i)$. Thus,
$(M,\varphi)$ is a $W^*$-probability space.

We will frequently use the following fact (see \cite[Theorem 4.1]{AH}): for every $t\in\mathbb{R}$, if
$(x_i)_i\in\mathcal{M}^{\mathcal{U}}$, then $(\sigma^{\varphi_i}_t(x_i))_i\in\mathcal{M}^{\mathcal{U}}$
and
\[
  (\sigma^{\varphi_i}_t(x_i))_{\mathcal{U}} = \sigma^{\varphi}_t((x_i)_{\mathcal{U}}).
\]
We will also briefly use the Groh--Raynaud ultraproduct.
Let $(M_i)_{i\in I}$ be von Neumann algebras represented in standard
form on Hilbert spaces $H_i$, and let $H_{\mathcal U}:=\prod_{\mathcal U} H_i$ be the Hilbert space ultraproduct.  Every bounded family
$(x_i)_{i\in I}\in\ell^\infty(M_i)$ defines an operator on
$H_{\mathcal U}$ by 
\[(x_i)_{\mathcal U}(\xi_i)_{\mathcal U}
:=
(x_i\xi_i)_{\mathcal U}\]
The von Neumann algebra generated by these operators is called the
\textbf{Groh--Raynaud ultraproduct} and is denoted by
$\prod_{\mathcal U}^{\mathrm{GR}} M_i.$
If $\varphi_i$ is a faithful normal state on $M_i$ for every $i\in I$ and
$N:=\prod_{\mathcal U}^{\mathrm{GR}} M_i$, 
then the family $(\varphi_i)_i$ defines a normal state
$\varphi_{\mathcal U}$ on $N$.  If $p:=\operatorname{supp}(\varphi_{\mathcal U})$, then, by \cite[Proposition~3.15]{AH}, we have
\[
\prod_{\mathcal U}(M_i,\varphi_i)
\cong
pNp.
\]
Thus the Ocneanu ultraproduct is a $\sigma$-finite corner of
the Groh--Raynaud ultraproduct. 
For a von Neumann algebra $M$, let $S_{\rm n}(M)$ denote its normal
state space and define the state-space diameter $d(M)$ by 
\[
d(M)
:=
\sup_{\varphi,\psi\in S_{\rm n}(M)}
\inf_{u\in\mathcal U(M)}
\|\varphi-u\psi u^*\|.
\]
The diameter for factors can be computed as follows (due to Connes-St\o rmer \cite{CS78}, Connes-Haagerup-St\o rmer \cite{CHS85}, and Haagerup-St\o rmer \cite{HS90}):
\begin{fact}\label{fact:state space diameter}
Let $M$ be a factor. Then $d(M)$ is 
\begin{itemize}
\item[{\rm{(1)}}] $2(1-\frac{1}{n})$ if $M$ is of type {\rm{I}}$_n,\ (n\in \mathbb{N}\cup \{\infty\})$.
\item[{\rm{(2)}}] $2$ if $M$ is of type {\rm{II}}.
\item[{\rm{(3)}}] $2\frac{1-\lambda^{\frac{1}{2}}}{1+\lambda^{\frac{1}{2}}}$ if $M$ is of type {\rm{III}}$_{\lambda}\ (0\le \lambda\le 1)$. 
\end{itemize}
\end{fact}
We will also use the fact that $d(M)=2$ if $M$ is not a factor.
We will need the following slight generalization of the diameter formula of \cite[Lemma~6.10]{AH}. 
\begin{lem}\label{lem:diameter-ultraproduct}
Let $\mathcal U$ be an ultrafilter on an arbitrary index set $I$ and
let $(M_i)_{i\in I}$ be a family of von Neumann algebras. Then
\[
d\left(\prod_{\mathcal U}^{\mathrm{GR}}M_i\right)
=
\lim_{i\to\mathcal U}d(M_i).
\]
\end{lem}
The proof is essentially the same as that of \cite[Lemma~6.10]{AH}.  
We refer to \cite{AH} for more details on ultraproducts of W$^*$-probability spaces. Although the results in \cite{AH} are stated for countably incomplete ultrafilters on $\mathbb{N}$, all the results we cited here work for arbitrary ultrafilters on arbitrary index sets with essentially the same arguments.

\subsection{Arveson spectral theory and kernel smoothing}\label{subsec:Arveson}
Here, we briefly recall Arveson's spectral theory for the modular flow. 
See \cite[Chapter XI]{TakesakiII} for details. 
We identify the dual group $\widehat{\mathbb{R}}$ of the additive group $\mathbb{R}$ with itself. For $f\in L^1(\mathbb{R})$, we define the Fourier transform $\hat{f}$  by 
\begin{equation}\hat{f}(\lambda):=\int_{\mathbb{R}}e^{\ri t\lambda}f(t)dt,\qquad \lambda \in \widehat{\mathbb{R}}=\mathbb{R}.\label{eq:FT}
\end{equation}

For $f\in L^1(\mathbb{R})$, we also consider the function $\sigma_f^\varphi:M\to M$ given by \[\sigma_f^{\varphi}(x):=\int_{\mathbb{R}}f(t)\sigma_t^{\varphi}(x)dt,\qquad x\in M,\] where the integral is taken in the $\sigma$-weak sense.  We note that $\|\sigma^\varphi_f(x)\|\leq \|f\|_1\|x\|$.

Below, we always let $M$ act on the Hilbert space $L^2(M,\varphi)$ associated with $\varphi$ with $\omega_{\varphi}$ the GNS vector, and $\Delta_{\varphi}$ is defined as the positive self-adjoint operator on $L^2(M,\varphi)$ as usual. 
For $f\in L^1(\mathbb R)$ and $x\in M$, we have 
\begin{equation}\sigma_f^{\varphi}(x)\omega_{\varphi}=\int_{\mathbb{R}}f(t)e^{it\log \Delta_{\varphi}}x\omega_{\varphi}dt=\hat{f}(\log \Delta_{\varphi})x\omega_{\varphi}.\label{eq modular op}
\end{equation}

For $x\in M$, $\text{Sp}_{\sigma^{\varphi}}(x)$ is defined by
\[\left \{\lambda \in \widehat{\mathbb{R}} \,\middle|\, \ \hat{f}(\lambda)=0 \text{  for all } f\in L^1(\mathbb{R}) \text{  with  }\sigma^{\varphi}_f(x)=0\right \}.\] 

For a closed subset $E$ of $\widehat{\mathbb{R}}$, the \textbf{spectral subspace} of $\sigma^{\varphi}$ corresponding to $E$ is given by 
\[M(\sigma^{\varphi},E):=\{x\in M\ : \ \text{Sp}_{\sigma^{\varphi}}(x)\subseteq E\}.\]
From (\ref{eq modular op}) and spectral theory, it follows that  
\[x\in M(\sigma^{\varphi},E)\iff 1_E(\log \Delta_{\varphi})x\omega_{\varphi}=x\omega_{\varphi}.\]
Moreover, the inclusion 
\begin{equation}\operatorname{Sp}_{\sigma^{\varphi}}(\sigma_f^{\varphi}(x))\subseteq \operatorname{Sp}_{\sigma^{\varphi}}(x)\cap \operatorname{supp}(\widehat{f})\label{eq:inclusionspectrum}
\end{equation}
holds.
The next lemma shows smoothing preserves total boundedness.
\begin{lem}\label{lem:smoothingpreserveSK}
Let $(M,\varphi)$ be a W$^*$-probability space, let $K>0$, and let
$x\in S_K(M,\varphi)$.  Then for every $f\in L^1(\mathbb R)$, we have
\[
\sigma_f^\varphi(x)\in S_{K\|f\|_1}(M,\varphi).
\]
In particular, if $\|f\|_1\leq1$, then 
$\sigma_f^\varphi(x)\in S_K(M,\varphi)$.
\end{lem}

\begin{proof}
By Lemma~\ref{lem:totally-bounded-char},
$x\in
D(\sigma_{\ri/2}^\varphi)
\cap
D(\sigma_{-\ri/2}^\varphi)$ and
\[
\max\left(\|x\|,
\|\sigma_{\ri/2}^\varphi(x)\|,
\|\sigma_{-\ri/2}^\varphi(x)\|\right)
\leq K.
\]

We first show that $\sigma_f^{\varphi}$ commutes
with the analytic generators:
\begin{equation}\label{eq:smoothing-analytic-generator}
\sigma_{\pm\ri/2}^\varphi
\bigl(\sigma_f^\varphi(x)\bigr)
=
\sigma_f^\varphi
\bigl(\sigma_{\pm\ri/2}^\varphi(x)\bigr).
\end{equation}
We verify the identity for $\sigma_{\ri/2}^\varphi$; the other case is
analogous. Since $x\in D(\sigma_{\ri/2}^{\varphi})$, the map $t\mapsto \sigma_t^{\varphi}(x)$ extends to a bounded $\sigma$-weakly continuous map $$F_x\colon \{z\in \mathbb C\mid 0\le \operatorname{Im}z\le \frac12\}\ni z\mapsto \sigma_z^{\varphi}(x)\in M,$$ which is holomorphic in the interior.   
For $0\le \operatorname{Im}z\le \frac12$, set $G(z)
:=
\int_{\mathbb R}
f(t)\sigma_{t+z}^\varphi(x)dt.$  
Since $F_x$ is bounded on the closed strip and $f\in L^1(\mathbb R)$,
the integral is well-defined in the $\sigma$-weak sense and $G$ is bounded. Moreover, $G$ is $\sigma$-weakly
continuous on the closed strip and holomorphic in its interior.

For $s\in\mathbb R$, we have
\[
\begin{aligned}
G(s)
&=
\int_{\mathbb R}
f(t)\sigma_{t+s}^\varphi(x)\,dt\\
&=
\sigma_s^\varphi
\left(
\int_{\mathbb R}f(t)\sigma_t^\varphi(x)\,dt
\right)\\
&=
\sigma_s^\varphi(\sigma_f^\varphi(x)).
\end{aligned}
\]
This shows that $G$ is the analytic extension of the modular orbit of
$\sigma_f^\varphi(x)$.  Consequently, $
\sigma_f^\varphi(x)\in D(\sigma_{\ri/2}^\varphi)$ holds  
and \[
\begin{aligned}
\sigma_{\ri/2}^\varphi(\sigma_f^\varphi(x))
&=
G(\ri/2)\\
&=
\int_{\mathbb R}
f(t)\sigma_t^\varphi
\bigl(\sigma_{\ri/2}^\varphi(x)\bigr)\,dt\\
&=
\sigma_f^\varphi
\bigl(\sigma_{\ri/2}^\varphi(x)\bigr).
\end{aligned}
\]
The same argument on the lower strip proves
\[
\sigma_{-\ri/2}^\varphi(\sigma_f^\varphi(x))
=
\sigma_f^\varphi
\bigl(\sigma_{-\ri/2}^\varphi(x)\bigr).
\]

It follows that
\[
\|\sigma_f^\varphi(x)\|
\leq
\|f\|_1\|x\|
\leq K\|f\|_1,
\]
while \eqref{eq:smoothing-analytic-generator} gives
\[
\begin{aligned}
\left\|
\sigma_{\pm\ri/2}^\varphi(\sigma_f^\varphi(x))
\right\|
&=
\left\|
\sigma_f^\varphi
\bigl(\sigma_{\pm\ri/2}^\varphi(x)\bigr)
\right\|\\
&\leq
\|f\|_1
\left\|
\sigma_{\pm\ri/2}^\varphi(x)
\right\|\\
&\leq
K\|f\|_1.
\end{aligned}
\]
By Lemma~\ref{lem:totally-bounded-char} again, we obtain  $\sigma_f^\varphi(x)\in S_{K\|f\|_1}(M,\varphi).$
\end{proof}

For $\tau>0$, define the \textbf{Gaussian kernel}
\[
k_\tau(t)=(4\pi\tau)^{-1/2}e^{-t^2/(4\tau)}.
\]
Then
\[
\int_{\mathbb R} k_\tau(t)\,dt=1,
\qquad
\widehat{k_\tau}(s)=\int_{\mathbb R}k_{\tau}(t)e^{\ri ts}dt=e^{-\tau s^2}.
\]
The following Gaussian smoothing lemma is well-known; see, for example, \cite{CioZsi,TakesakiII}; it is also used in \cite{AG2}. 
\begin{lem}\label{lem:gauss-total-bounded}
    Let $(M,\varphi)$ be a W$^*$-probability space, $a\in M$, and $\tau>0$. Then the element 
    $b:=\sigma_{k_{\tau}}^{\varphi}(a)$ is $\sigma^{\varphi}$-analytic.  Moreover, for all $z\in \mathbb{C}$, we have 
    \[
\sigma_z^{\varphi}(b)
=
\int_{\mathbb R} k_\tau(t-z)\sigma_t^{\varphi}(a)dt.
\]
In particular, for $y\in \mathbb{R}$, we have
    \begin{equation}
        \|\sigma_{\ri y}^{\varphi}(b)\|\le e^{y^2/4\tau}\|a\|,\qquad y\in \mathbb{R}.\label{eq heat regularize}
\end{equation}
Therefore, by Lemma \ref{lem:totally-bounded-char},\, $b\in S_{e^{1/(16\tau)}\norm{a}}(M,\varphi)$ holds.\\
In particular, if $e\in M$ satisfies $0\le e\le 1$, then 
\[
 e_\tau:=\sigma_{k_\tau}^\varphi(e)
 =\int_{\mathbb{R}} k_\tau(t)\sigma_t^\varphi(e)\,dt,
\]
satisfies $0\le e_\tau\le1$, $\varphi(e_\tau)=\varphi(e)$, and $e_\tau\in S_{e^{1/(16\tau)}}(M,\varphi)$. 
\end{lem}
Combining Lemma \ref{lem:gauss-total-bounded} with \cite[Lemma 4.14]{AH}, we obtain:
\begin{cor}\label{cor:heat-descends}
Suppose that $M = \prod_\cU (M_i,\varphi_i)$ is an Ocneanu ultraproduct of W$^*$-probability spaces.
Let $a=(a_i)_\cU\in M$ and $\tau>0$. 
For each $i\in I$, set $b_i:=\sigma_{k_{\tau}}^{\varphi_i}(a_i)$. 
Then $(b_i)_\cU\in M$, is totally $e^{1/(16\tau)}\|a\|$-bounded and satisfies 
\[(b_i)_{\cU}=\sigma_{k_{\tau}}^{\varphi}(a).\]
\end{cor}

We will also need another kind of smoothing using the \textbf{Fej\'er kernel} $F_r: \mathbb{R} \to \mathbb{R}$ defined by
\[
    F_r(t) = \begin{cases}
        \frac{1 - \cos(rt)}{\pi r t^2} & \text{if } t \not= 0 \\
        \frac{r}{2\pi} & \text{if } t = 0
    \end{cases}
\]
whose Fourier transform is given by
\[
    \widehat{F_r}(\lambda) = \begin{cases}
        1 - \frac{|\lambda|}{r} & \text{if } |\lambda| \leq r \\
        0 & \text{if } |\lambda| > r
    \end{cases}.
\]

\section{Tomita-Takesaki theory according to Rieffel and Van Daele}\label{SectionRvD}

In this section, we present the basic facts from the bounded operator approach to Tomita-Takesaki due to Rieffel and Van Daele \cite{RvD} used throughout this paper.

Fix a standard form $(M,\om)$ of $M$.  We set $\cal K:=\overline{M_{\sa}\omega}$, a closed, \emph{real} subspace of $H$.  Viewing $H$ as a real Hilbert space by taking the real part of the original inner product, we let $P:H\to H$ denote the orthogonal projection onto $\cal K$.  We furthermore set $Q:H\to H$ to be the orthogonal projection onto the closed real subspace $i\cal K$ of $H$ and set $R:=P+Q$, a positive operator.  Moreover, one considers the polar decomposition of $\Theta:=P-Q=JT$ of $P-Q$, where $J$ is a partial isometry and $T$ is a positive operator.  We list here the most important facts about these operators and spaces, whose proofs can be found in \cite{RvDc,RvD}:

\begin{facts}\label{RvDfacts}

\

\begin{enumerate}
    \item $R$ and $T$ are complex-linear operators.
    \item $J$ is an anti-linear, self-adjoint, orthogonal operator (whence $J^2=1$) satisfying $J\omega=\omega$.
    \item $R$ and $2-R$ are injective.
    \item For all $v\in H$, $iP(v)=Q(iv)$.
    \item $\overline{M'_{\sa}\omega}=i\cal K^\perp$.
    \item If $x \in M$, then $\Theta x\omega = (2 - R)x^*\omega$.
    \item If $x' \in M'$, then $\Theta x'\omega = Rx^{'*}\omega$
\end{enumerate}
\end{facts}

Since the spectra of $R$ and $2-R$ are both contained in $[0,2]$ and $0$ is not the eigenvalue of $R$ or $2-R$, one may use Borel functional calculus to define the operators $R^{it}$ and $(2-R)^{it}$ for all $t\in \mathbb{R}$, yielding two commuting strongly continuous one-parameter groups of unitaries of $H$.
One then defines a new strongly continuous one-parameter group of unitaries $\Delta^{it}:=(2-R)^{it}R^{-it}$. 

The notation $\Delta^{it}$ may seem strange since one has not defined the operator $\Delta$ itself; in the usual (unbounded) approach to Tomita-Takesaki theory, one indeed defines the unbounded operator $\Delta$ first and then uses (unbounded) functional calculus to define the unitary operators $\Delta^{it}$. One can describe $\Delta$ using the approach discussed above, namely $\Delta:=R^{-1}(2-R)$; the presence of the unbounded operator $R^{-1}$ renders $\Delta$ unbounded.  

It will behoove us to recall the ``standard'' approach to defining $\Delta$.  Indeed, recalling the closed, unbounded operators $S$ and $F$ introduced in the previous section, the usual definition of $\Delta$ is as the square of the positive operator in the polar decomposition of $S$, that is $S=J\Delta^{1/2}$, where $J$ is as defined above.  Note then that $\Delta=S^*S=FS$.

We next observe how this picture looks when we switch focus from $M$ to $M'$.  Let us temporarily decorate the corresponding operators with a $'$, such as $S'$, $F'$, etc...  It is clear that $S'=F$, $F'=S$, $\Delta'=F'S'=SF=\Delta^{-1}$.  By item (5) of Facts \ref{RvDfacts}, we see that $P'=1-Q$.  Analogously, one has that $Q'=1-P$.  Consequently, $\Theta'=P'-Q'=(1-Q)-(1-P)=P-Q=\Theta$.  Note also that $R'=P'+Q'=(1-Q)+(1-P)=2-R$.  

The following is \cite[Corollary 4.4]{RvD}.

\begin{lem}\label{thetaAdjoint1}
    For every $x' \in M'$, there exists a unique $x \in M$ such that 
    \[
    \Theta x'\om = x\om \quad \text{ and } \quad \Theta (x')^*\om = x^*\om.
    \]

By interchanging $M$ and $M'$, we also have that, for every $y \in M$, there exists a unique $y' \in M'$ so that 
    \[
    \Theta y\om = y'\om \quad \text{ and } \quad \Theta y^*\om = (y')^*\om.
    \]
\end{lem}

For our purposes, we will need a sharpening of the previous lemma.  In what follows, for $v\in M\om$, we abuse notation by setting $\|\pi(v)\|:=\|a\|$, where $a\in M$ is the unique element of $M$ for which $v=a\om$.

\begin{lem}\label{Pbounds}
    If $x \in M'$ is such that $\|x\|\leq K$, then $P(x\om),Q(x\om)\in M\omega$ and $\|\pi(Px\om)\|,\|\pi(Qx\om)\|\leq 2K$.
\end{lem}

\begin{proof}
    First suppose that $x'\in M'$ is such that $0\leq x'\leq 1$.  In the proof of \cite[Lemma 4.3]{RvD}, the authors find $x\in M_{\sa}$ with $\|x\|\leq 1$ such that $P(x'\omega)=x\omega$.  Consequently, for general positive $x'\in M'$, there is $x\in M_{\sa}$ such that $\|x\|\leq \|x'\|$ satisfying $P(x'\omega)=x\omega$.

    If $x'\in M'_{\sa}$ is not necessarily positive, then we may write $x'=x_1'-x_2'$ with $x_1',x_2'\in M_{\sa}'$ positive and $\|x_i'\|\leq \|x'\|$; it follows that $$P(x'\omega)=P(x_1'\omega)-P(x_2'\omega)=x_1\omega-x_2\omega=(x_1-x_2)\omega,$$ where $x_1,x_2\in M_{\sa}$ are such that $\|x_1-x_2\|\leq \|x_1\|+\|x_2\|\leq \|x_1'\|+\|x_2'\|\leq 2\|x'\|$. The claim follows.

    If $x' \in M'$, then by the cartesian decomposition in $M'$, we can find $y', z' \in M'_{\sa}$ such that $x' = y' + iz'$ and $\|y'\|, \|z'\| \leq \|x'\|$.  Recall from \cite{RvD} that $iM'_{\sa}\om$ is orthogonal to $M_{\sa}\om$ in the real Hilbert space $H$ and hence $P(iz'\om) = 0$.  Thus $P(x'\om) = P(y'\om)$ and the claim follows from the previous paragraph.

    Noting that $P = -iQi$, the same proof gives the statement for $Q$.
\end{proof}

An immediate corollary of Lemma \ref{Pbounds} is the following:

\begin{cor}\label{Pleft}
If $a\in M_{\tb}$, then $P(a\omega)=b\omega$ for a unique $b\in M$ with $\|b\|\leq 2\|a\|_{\rt}$. 
\end{cor}

We now prove an upper bound for the right norm of $b$ as in the previous corollary:

\begin{lem}\label{Pright}
Fix $a\in M_{\tb}$ and write $P(a\omega)=b\omega$ for $b\in M_{\sa}$. Then $b\in M_{\tb}$ and $\|b\|_{\rt}\leq \|a\|_{\rt}+2\|a\|$.
\end{lem}

\begin{proof}
First, recalling that $\omega$ is also a cyclic and separating vector for $M'\subseteq \cal B(H_\vp)$, we set $\cal K':=\overline{M'_{\sa}\omega}$ and let $P'$ denote the projection onto $\cal K'$.  By Fact \ref{RvDfacts}(5), we have that $\cal K'=i\cal K^\perp$, whence $P'=1-Q$.  By Lemma \ref{Pbounds} applied to $P'$ and using the fact that $a\in M=(M')'$, we have that $P'(ia\omega)=c\omega$ with $c\in M'$ satisfying $\|c\|\leq 2\|a\|$.  On the other hand, $P'(ia\omega)=(1-Q)(ia\omega)=ia\omega -iP(a\omega)$, so $P(a\omega)=a\omega+ic\omega$.  Recalling that $P(a\omega)=b\omega$, we have, for any $d\in M$, that
$$\|db\|_\vp=\|da\omega+d(ic\omega)\|\leq \|a\|_{\rt}\|d\|_\vp+\|c\|\|d\|_\vp\leq (\|a\|_{\rt}+2\|a\|)\|d\|_\vp.$$  It follows that $b\in M_{\tb}$ and $\|b\|_{\rt}\leq \|a\|_{\rt}+2\|a\|$.
\end{proof}

\begin{cor}\label{PpreservesRb}
If $a\in S_K(M)$, then $P(a\omega)=b\omega$ for a unique $b\in S_{3K}(M)$.
\end{cor}

\begin{proof}
    By Corollary~\ref{Pleft}, there is a unique $b\in M$ such that
    $P(a\omega)=b\omega$    
    and
    $\|b\|
        \leq 2\|a\|_{\rt}
        \leq 2K$. Moreover, the proof of Lemma~\ref{Pbounds} shows that $b$ is
    self-adjoint. By Lemma~\ref{Pright},
    \[
        \|b\|_{\rt}
        \leq \|a\|_{\rt}+2\|a\|
        \leq 3K.
    \]
    Since $b=b^*$, we also have
    \[
        \|b^*\|_{\rt}=\|b\|_{\rt}\leq 3K.
    \]
    Moreover, $\|b\|\leq 2K\leq 3K$.
    Therefore, $b\in S_{3K}(M)$ holds.
\end{proof}

The main theorem of Tomita-Takesaki theory is the following:

\begin{thm}[Theorem VI.1.19 in \cite{TakesakiII}]\label{TT}
$JMJ=M'$ and $\Delta^{it}M\Delta^{-it}=M$ for all $t\in \mathbb{R}$.
\end{thm}

By the previous fact, one can define a strongly continuous one-parameter family $\sigma_t$ of automorphisms of $M$ by defining $\sigma_t(x):=\Delta^{it}x\Delta^{-it}$.  This family is called the \textbf{modular automorphism group} of the standard form $(M,\omega)$.

We now specialize to the case that $(M,\vp)$ is a \wstar-probability space with standard form stemming from the GNS representation.  As mentioned in Section \ref{SectionPrelim}, sometimes we identify $M$ with its image $\pi_{\vp}(M)$.  Consequently, $M'$ here really refers to $\pi_{\vp}(M)'$.  We occasionally might decorate the various operators just introduced with $\vp$'s, such as $P_\vp$, to indicate that we are working in this context.  In particular, we write $\sigma_t^\vp$ for the aforementioned one-parameter family of automorphisms of $M$ and refer to this family as the \textbf{modular automorphism group of $\vp$.}  

In Section \ref{SectionExpanding}, we explain how the Rieffel-Van Daele approach to Tomita-Takesaki theory can be used to expand the language given in Section \ref{SectionLanguage} by symbols for the modular automorphism group of $\vp$ and to extend the theory presented in Section \ref{SectionLanguage} to an effective theory in the expanded language axiomatizing the modular automorphism group.  In order to do so, we will need the following two facts, the first of which is a special case of \cite[Theorem 5.15]{RvD}:

\begin{fact}\label{515}
For each $a\in M$, there is a unique $b\in M$ such that $\Delta^{it}(a\omega)=b\omega$.  Moreover, we have $\pi(b)=\Delta^{it}\pi(a)\Delta^{-it}$.  Consequently, $\|a\|=\|b\|$.
\end{fact}

Note that by taking adjoints in the previous fact, we see also that $\Delta^{it}(a^*\omega)=b^*\omega$.  The second fact is a special case of \cite[Theorem 5.14]{RvD}:

\begin{fact}\label{514}
For all $a,b\in M$ and $t\in \mathbb{R}$, we have
$$(\Delta^{it}J\pi(a)J\Delta^{-it})(b\omega)=\pi(b)(\Delta^{it}(J a\omega)).$$
\end{fact}

We now return to the general case of a standard form $(M,\vp)$.  We need one further fact about the modular operator, to be used in Section~\ref{tb}. See \cite[VI, Lemma 1.17]{TakesakiII}.

\begin{fact}[Bridging Lemma/Takesaki's Resolvent Lemma]\label{bridge}
    For every $x \in M$ and $\lambda \in \mathbb{C}\setminus \mathbb{R}_{+}$, there exists a unique $y' \in M'$ such that $(\Delta^{-1} - \lambda I)^{-1}x\om = y'\om$. Furthermore, one has 
    \[
    \|y'\| \leq \frac{\|x\|}{\sqrt{2|\lambda|-2\operatorname{Re}(\lambda)}}.
    \]
    Similarly, for every $x' \in M'$ and $\lambda \in \mathbb{C}\setminus \mathbb{R}_{+}$, there exists a unique $y \in M$ such that $(\Delta - \lambda)^{-1}x'\om = y\om$.  Furthermore, one has
    \[
    \|y\| \leq \frac{\|x'\|}{\sqrt{2|\lambda|-2\operatorname{Re}(\lambda)}}.
    \]
\end{fact}

We will need a ``totally bounded'' version of the previous fact:
\begin{cor}\label{bridgetotallynew}
    For $\lambda \in \mathbb{C}\setminus \mathbb{R}_+$, set $c_\lambda
=
\sqrt{2|\lambda|-2\operatorname{Re}\lambda}$. If $x\in M$ (resp. $x'\in M'$) is totally bounded, then $y'\in M'$ (resp. $y\in M$) as in Fact \ref{bridge} is totally bounded with 
\[
\|y'\|
\leq
\frac{\|x\|}{c_\lambda},
\]
\[
\|y'\|_{\mathrm{right}}
\leq
\frac{\|x^*\|_{\mathrm{right}}}{c_\lambda},
\qquad
\|(y')^*\|_{\mathrm{right}}
\leq
\frac{\|x\|_{\mathrm{right}}}{c_\lambda}
\]
(and similarly for the bounds on $\|y\|_{\rt},\|y^*\|_{\rt}$).  
\end{cor}
\begin{proof}
Put $r:=\pi'(x)\in M', s:=\pi'(x^*)\in M'.$ Thus $r\omega=x\omega, 
s\omega=x^*\omega,$ and
\[
\|r\|=\|x\|_{\mathrm{right}},
\qquad
\|s\|=\|x^*\|_{\mathrm{right}}.
\]
Since $S(r\omega)=x^*\omega=s\omega$ 
and $F(s\omega)=s^*\omega$, we have 
$r\omega\in D(\Delta)$ and
$\Delta r\omega=s^*\omega.$

From the identity
\[
(t^{-1}-\lambda)^{-1}
=
-\frac1\lambda
(t-\lambda^{-1})^{-1}t
\qquad(t>0),
\]
we obtain
\[
\begin{aligned}
y'\omega
&=
(\Delta^{-1}-\lambda I)^{-1}r\omega\\
&=
-\frac1\lambda
(\Delta-\lambda^{-1}I)^{-1}\Delta r\omega\\
&=
-\frac1\lambda
(\Delta-\lambda^{-1}I)^{-1}s^*\omega.
\end{aligned}
\]
Since $\lambda^{-1}\notin\mathbb R_+$, 
Fact~\ref{bridge}, applied to $s^*\in M'$ with parameter $\lambda^{-1}$
gives an element $a\in M$ such that
\[
(\Delta-\lambda^{-1}I)^{-1}s^*\omega=a\omega,\qquad 
\|a\|
\leq
\frac{\|s^*\|}{c_{\lambda^{-1}}}.
\]
Consequently, 
$y'\omega=-\frac1\lambda a\omega$. 
By the right-bounded-vector characterization for the standard form
$(M',\omega)$, we have 
\[
\|y'\|_{\mathrm{right}}
\leq
\frac1{|\lambda|}\|a\|.
\]
Now $c_{\lambda^{-1}}
=
\frac{c_\lambda}{|\lambda|}$, 
and hence
\[
\|y'\|_{\mathrm{right}}
\leq
\frac{\|s\|}{c_\lambda}
=
\frac{\|x^*\|_{\mathrm{right}}}{c_\lambda}.
\]

We next estimate the right norm of $(y')^*$.  From 
$r\omega
=
(\Delta^{-1}-\lambda I)y'\omega$ 
and $\Delta^{-1}=SF$, applying the antilinear operator $F$ gives
\[
\begin{aligned}
r^*\omega
&=
F r\omega\\
&=
F(\Delta^{-1}y'\omega)-\overline{\lambda}F(y'\omega)\\
&=
\Delta F(y'\omega)-\overline{\lambda}F(y'\omega)\\
&=
(\Delta-\overline{\lambda}I)(y')^*\omega.
\end{aligned}
\]
Therefore
\[
(y')^*\omega
=
(\Delta-\overline{\lambda}I)^{-1}r^*\omega.
\]
Applying Fact~\ref{bridge} to $r^*\in M'$, with parameter
$\overline{\lambda}$, gives $b\in M$ such that $(y')^*\omega=b\omega$ and
\[
\|b\|
\leq
\frac{\|r^*\|}{c_{\overline{\lambda}}}
=
\frac{\|r\|}{c_\lambda}.
\]
Hence
\[
\|(y')^*\|_{\mathrm{right}}
\leq
\frac{\|x\|_{\mathrm{right}}}{c_\lambda}.
\]

The operator-norm estimate
$\|y'\|
\leq
\frac{\|x\|}{c_\lambda}$ 
 follows from the first assertion of Fact~\ref{bridge}.  Thus $y'$ is totally
bounded.
\end{proof}

\section{Characterization of $M_{\tb}$}\label{tb}

In this section, we prove a characterization of the set of totally bounded elements of a \wstar-probability space, to be used in our axiomatization of \wstar-probability spaces in Section~\ref{SectionLanguage}.  One of the items in the characterization involves the operator $h_a(\log(\Delta))$, where $h_a(t)$ is the following function, defined for all $a\in \mathbb{R}$: 

$$h_{a}(t) := 1/(\cosh(t-a)) = \frac{2}{e^{t-a} + e^{a-t}}.$$

Throughout this section, we fix a \wstar-probability space $(M,\vp)$ with associated standard form $(M,\om)$.

\begin{prop}\label{hatb}
For every $a \in \mathbb{R}$ and $x \in M_{\tb}$, there exists $y \in M_{\tb}$ satisfying $h_{a}(\log(\Delta))x\om = y\om$. Furthermore, 
\[
\|y\|\leq\|x\|,
\qquad
\|y\|_{\rt}\leq\|x\|_{\rt},
\qquad
\|y^*\|_{\rt}\leq\|x^*\|_{\rt}.
\]
\end{prop}

\begin{proof}
Note that
\[
h_a(\log(\Delta))
=
2(e^{-a}\Delta+e^a\Delta^{-1})^{-1}
=
2\ri (\Delta+\ri e^aI)^{-1}
   (\Delta^{-1}+\ri e^{-a}I)^{-1}.
\]
Set $\lambda_-=-\ri e^{-a}$ and $\lambda_+=-\ri e^a$.
Then $c_{\lambda_-}c_{\lambda_+}
=
\sqrt{2e^{-a}}\sqrt{2e^a}=2$. 

Two applications of Corollary~\ref{bridgetotallynew} therefore give
$y\in M_{\tb}$ with
$h_a(\log(\Delta))x\om=y\om$ and
\[
\|y\|
 \leq\frac{2}{c_{\lambda_-}c_{\lambda_+}}\|x\|
 =\|x\|,
\]
\[
\|y\|_{\rt}
 \leq\frac{2}{c_{\lambda_-}c_{\lambda_+}}\|x\|_{\rt}
 =\|x\|_{\rt},
\qquad
\|y^*\|_{\rt}
 \leq\frac{2}{c_{\lambda_-}c_{\lambda_+}}\|x^*\|_{\rt}
 =\|x^*\|_{\rt}.
\]
\end{proof}

The following easy calculation will prove useful in the axiomatization of W$^*$-probability space in Section~\ref{SectionLanguage}; see in particular axiom \ref{axiom:closedunderh_a(logD)}.

\begin{lem}\label{calculation}
For $x,y\in M$, we have $h_a(\log(\Delta))x\om = y\om$ if and only if $$2R(2-R)x\om = (e^{-a}(2-R)^2 + e^{a}R^2)y\om.$$ 
\end{lem}

The following is the main result of this section:

\begin{thm}\label{characterization}
Suppose that $M_0$ is a strongly dense $*$-subalgebra of $M$ such that:
\begin{enumerate}
    \item $M_0\subseteq M_{\tb}$.
    \item $S_1(M_0)$ is $\|\cdot\|_\varphi^\#$-complete.
    \item $h_a(\log \Delta)(M_0\omega)\subseteq M_0\omega $\, for every $a\in \mathbb{N}$.
\end{enumerate}
Then $M_0=M_{\tb}$.
\end{thm}
\begin{defn}\label{defn:good}
    Let $(M,\varphi)$ be a W$^*$-probability space. A strongly dense $*$-subalgebra $M_0$ of $M$ satisfying the hypothesis of Theorem \ref{characterization} is called a \textbf{good subalgebra} of $M$.  
\end{defn}

\begin{remark}
    At this point, the stipulation of closure of $M_0$ under $h_a(\log(\Delta))$ may appear to be lacking a clear motivation. However, given a W$^*$-probability space $(M,\vp)$, the restriction of the spectral subspaces of $\log(\Delta)$ to $M\om$ are related to the spectral subspaces $M(\sg^\vp, [a,b])$ of $(M,\vp)$ (as defined in section \ref{subsec:Arveson}) by 
    \[M(\sg^\vp,[a,b]) = \{x\in M : 1_{[a,b]}(\log(\Delta))x\om = x\om\}.\]
    This fact, together with the characterization in Lemma \ref{calculation}, allows us to express the compact spectral subspaces in terms of the bounded framework of Rieffel and Van Daele on which our axiomatic treatment of \wstar-probability spaces will rely. These compact spectral subspaces have regularity and density properties that will prove useful in recovering $M$ from its totally bounded elements.
\end{remark}

From now on, $j$ denotes an element of $\mathbb{R}^{>1}$.  We let $E(j^{-1}, j)$ denote the spectral projection of $\Delta$ (and $\Delta^{-1}$) corresponding to the interval $(j^{-1}, j)$ and set $E_j:=E(j^{-1},j)(H_\varphi)$.

The following is \cite[Lemma 4.12]{Kadison}:

\begin{fact}\label{kadison412}
    Suppose that $x \in M$ and $x\om \in E_j$.  Then:
    \begin{enumerate}
        \item For all $n\in \mathbb{Z}$, there is $x_n\in M$ such that $\Delta^n(x\om)=x_n\om$ and $\|x_n\|\leq j^{|n|}\|x\|$.
        \item $x$ is right bounded and $\|x\|_{\rt}\leq j^{1/2}\|x\|$.
    \end{enumerate}
The statement obtained by switching $M$ and $M'$ is also valid. 
\end{fact}

\begin{cor}\label{tot}
    Suppose that $x \in M$ and $x\om \in E_j$.  Then $x$ is totally $j^{3/2}\|x\|$-bounded.  
\end{cor}

\begin{proof}
    By Fact \ref{kadison412}, $\|x\|_{\rt} \leq j^{\frac{1}{2}}\|x\|$.  By Fact \ref{kadison412} again, we may write $\Delta(x\om)=x_1\om$ with $x_1\in M$ and $\|x_1\|\leq j\|x\|$.  Since $E_j$ is invariant under $\Delta$, we may apply the previous fact again to conclude that $\|x_1\|_{\rt}\leq j^{1/2}\|x_1\|\leq j^{3/2}\|x\|$.  Note that $\Delta x\om \in M'\om$ by the previous fact and thus $Sx\om=F\Delta x\om$. Since $\pi'(Fy\om) = \pi'(y)^*$ for all $y\in M'$, it follows that 
    \[
    \|x^*\|_{\rt}=\|\pi'(Sx\om)\| = \|\pi'(Sx\om)^*\| = \|\pi'(FSx\om)\| = \|\pi'(\Delta x\om)\| \leq j^{\frac{3}{2}}\|x\|.
    \]
This proves the corollary.
\end{proof}

Our next main goal is the following, which by Corollary \ref{tot}, is a special case of Theorem \ref{characterization}: 

\begin{thm}\label{Dcore}
Suppose that $M_0$ is a good subalgebra.  If $x\in M$ is such that $x\om \in E_j$ for some $j\in \mathbb{N}$, then $x \in M_0$.
\end{thm}

First we prove:
\begin{lem}\label{lem:ha-preservation}
Let $M_0\subseteq M$ be a good subalgebra.  Fix $r\in\mathbb N$. Then for every $x\in M$, there is a unique $T_r(x)\in M_0$ such that
\[
T_r(x)\omega
=
h_r(\log\Delta)x\omega.
\]
Moreover, there is a constant $C_r>0$, depending only on $r$, such
that
\[
\|T_r(x)\|\leq \|x\|,
\]
and
\[
\|T_r(x)\|_{\mathrm{right}},
\quad
\|T_r(x)^*\|_{\mathrm{right}}
\leq C_r\|x\|.
\]
\end{lem}

\begin{proof}
By the Fourier transform formula for the hyperbolic secant
(see \cite[Chapter~3, Example~3, pp.~81--83]{SteinShakarchiComplex}), the Fourier transform of $\kappa(t):=\frac{1}{2\cosh(\pi t/2)}$ (in our convention (\ref{eq:FT})) is $\widehat{\kappa}(s)=\frac{1}{\cosh s}$.

Therefore, the Fourier transform of  $\kappa_r(t):=e^{-\ri rt}\kappa(t)$ 
 satisfies  
\[
\widehat{\kappa_r}(s)
=
\frac1{\cosh(s-r)}
=
h_r(s),
\qquad
\|\kappa_r\|_1=1.
\]

For $x\in M$, define
\[
T_r(x)
:=
\int_{\mathbb R}
\kappa_r(t)\sigma_t^\varphi(x)\,dt.
\]
Then
$T_r(x)\in M$ and $
\|T_r(x)\|
\leq
\|\kappa_r\|_1\|x\|
=
\|x\|.$
Moreover,
\[
T_r(x)\omega
=
\int_{\mathbb R}
\kappa_r(t)\Delta^{\ri t}x\omega\,dt
=
h_r(\log\Delta)x\omega.
\]

The function $H(z)$ given by
\[
H(z)=
\int_{\mathbb R}
\kappa_r(t-z)\sigma_t^\varphi(x)\,dt
\]
is analytic on the strip $
|\operatorname{Im}z|<1$. 
For $s\in \mathbb{R}$, a change of variables shows $H(s)=\sigma_s^\varphi(T_r(x))$. 
Consequently $H$ is the analytic continuation of the orbit map 
$s\mapsto \sigma_s^{\varphi}(T_r(x))$.
Since the map $t\longmapsto\kappa_r(t\pm i/2)$
belongs to $L^1(\mathbb R)$, we obtain
\[
\left\|\sigma_{\pm \ri/2}^\varphi(T_r(x))\right\|
\leq
\|\kappa_r(\,\cdot\mp \ri/2)\|_1\,\|x\|.
\]
Set $C_r
:=
\max\left\{
1,\,
\|\kappa_r(\,\cdot+\ri/2)\|_1,\,
\|\kappa_r(\,\cdot-\ri/2)\|_1
\right\}.$ 
Then
\[
\|T_r(x)\|_{\mathrm{right}}
=
\|\sigma_{\ri/2}^\varphi(T_r(x))\|
\leq C_r\|x\|
\]
and
\[
\|T_r(x)^*\|_{\mathrm{right}}
=
\|\sigma_{-\ri/2}^\varphi(T_r(x))\|
\leq C_r\|x\|.
\]
Thus, $T_r(x)$ is totally $C_r\|x\|$-bounded.

Next, we show that $T_r(x)\in M_0$.  By Kaplansky's density theorem, there exists a net $(x_\lambda)_\lambda$ in $M_0$ converging 
$*$-strongly to $x$ with $\sup_\lambda\|x_\lambda\|\leq \|x\|$.
Since $M_0$ is closed under $h_r(\log\Delta)$, we have
$T_r(x_\lambda)\in M_0.$ 
By the preceding estimates, we have 
$T_r(x_\lambda)
\in
S_{C_r\|x\|}(M_0)$ 
for every $\lambda$.
Furthermore,
\[
\|T_r(x_\lambda)-T_{r}(x)\|_\varphi
\leq
\int_{\mathbb R}
|\kappa_r(t)|
\,
\|\sigma_t^\varphi(x_\lambda-x)\|_\varphi\,dt= \|\kappa_r\|_1\|x_{\lambda}-x\|_{\varphi}
\xrightarrow{\lambda\to \infty}0.
\]
The same argument applied to adjoints gives
convergence in $\|\cdot\|_\varphi^\#$. 
Choose an integer $N\geq C_r\norm x$.  Every $T_r(x_\lambda)$ lies in $S_N(M_0)$.  By scaling, the $\|\cdot\|_\varphi^\#$-completeness of $S_N(M_0)$ implies that this subset is $\|\cdot\|_\varphi^\#$-closed in
$S_N(M)$, and therefore $T_r(x)\in M_0$.
\end{proof}
\begin{proof}[Proof of Theorem \ref{Dcore}]
Suppose
$x\omega\in E_j
=
1_{(j^{-1},j)}(\Delta)H.$

Fix any $r\in\mathbb N$. By Fact~\ref{kadison412}, there are elements
$x_{1},x_{-1}\in M$ such that
\[
\Delta x\omega=x_1\omega,
\qquad
\Delta^{-1}x\omega=x_{-1}\omega.
\]
Define $z
:=
\frac12\left(e^{-r}x_1+e^r x_{-1}\right)\in M.$ 
Then
\[
\begin{aligned}
z\omega
&=
\frac12
\left(
e^{-r}\Delta
+
e^r\Delta^{-1}
\right)x\omega  \\
&=
\cosh(\log \Delta-r)x\omega.
\end{aligned}
\]
Since
$h_r(\log \Delta)=\cosh(\log \Delta-r)^{-1},$ 
we obtain
$h_r(\log \Delta)z\omega=x\omega.$ 

By Lemma~\ref{lem:ha-preservation}, we have 
$h_r(\log \Delta)z\omega\in M_0\omega$, that is, there is $y\in M_0$ such
that 
\[y\omega=h_r(\log \Delta)z\omega=x\omega.\] 
Since $\omega$ is separating for $M$, we conclude that $x=y\in M_0$. 
\end{proof}

In light of Theorem \ref{Dcore}, in order to finish the proof of Theorem \ref{characterization}, one must approximate arbitrary totally bounded elements by ones that belong to compact spectral subspaces.  
We use the method of Bochner integrals found, for instance, in \cite[II, Lemma VI.2.4]{TakesakiII}.  

For each $v \in H$, define $v_r = \widehat{F_r}(\log(\Delta))v$, where $F_r$ is the Fej\'er kernel defined in Subsection \ref{subsec:Arveson}.  By the Fourier transform computation done there, we have that $v_r \in E(e^{-r}, e^r)H$ for every $r > 0$.
 
For a proof of the following see, for instance \cite[Section 2]{Dab}.

\begin{fact}
    Suppose that $x\in M$ and set $v:=x\om$.  Then for each $r>0$, there is $x_r\in M$ such that $v_r=x_r\om$.  Moreover, $\|x_r\|\leq \|x\|$. Furthermore, we have $x_r$ converges $*$-strongly to $x$ as $r \to \infty$.
\end{fact}

\begin{proof}[Proof of Theorem \ref{characterization}]
Suppose that $x\in M_{\tb}$ is a totally $K$-bounded element and set $v=x\om$.  By Theorem \ref{Dcore}, each $x_r\in M_0$.  
Since $F_r$ is real-valued,
even, positive, and has $L^1$-norm one, by Lemma~\ref{lem:smoothingpreserveSK}, we have $x_r\in S_K(M_0)$.  Since $S_K(M_0)$ is $\|\cdot\|_\varphi^\#$-complete, we have that $x\in M_0$.   
\end{proof} 

\section{Effective distance estimate to the totally $1$-bounded elements and the proof of definability of $S_K$}\label{sec:effectivealphaandSKdefinability}

As stated in the introduction, when axiomatizing \cstar-algebras and tracial von Neumann algebras as model-theoretic structures using operator norm balls as sorts, it is an important step to be able to express the fact that the elements in the operator norm ball of radius $n$ whose operator norm is actually at most $1$ must be in the image of the canonical inclusion mapping from the ball of radius $1$ to the ball of radius $n$.  Such an axiomatization can be given using a standard functional calculus argument.  We would like to express the analogous fact in our current setting.  Unfortunately, such a simple functional calculus argument is not available to us in the non-tracial setting.  
Our goal in this section is (1) to prove a replacement for this functional calculus argument which will provide the relevant axioms and its effectiveness, and (2) to prove that the operation of taking totally bounded elements commutes with the Ocneanu ultraproduct.   

\subsection{Statements of the main results}\label{subsec:statements}

We begin by setting up some notation:

\begin{defn}
Fix $n,k \in \mathbb{N}$ and a W$^*$-probability space $(M,\varphi)$.  For $x \in S_n(M)$, define
\[
B^{L}_{k,n}(x)
:=
\sup_{y \in S_1(M)}
\frac{\|xy\|_{\vp}}
{\max\{\|y\|_{\vp},1/k\}},
\]
and similarly define
\[
B^{L}_{k,n}(x^*), \qquad B^{R}_{k,n}(x), \qquad B^{R}_{k,n}(x^*),
\]
using left and right multiplication respectively.  We then set
\[
\operatorname{TB}_{k,n}(x)
:=
\max\left\{
B^{L}_{k,n}(x),
B^{L}_{k,n}(x^*),
B^{R}_{k,n}(x),
B^{R}_{k,n}(x^*)
\right\}.
\]
\end{defn}
For $n\ge1$, set $\beta_{k,n}=\frac{2^{-k}}{n}$ and define, for $x\in S_n(M,\varphi)$,
\begin{equation}
 \Phi_n(x)
 =\sum_{k=1}^\infty
 \beta_{k,n}\bigl(\operatorname{TB}_{k,n}(x)\dminus 1\bigr).
 \label{eq:Phi-n}
\end{equation}
We record some elementary bounds on $\operatorname{TB}_{k,n}$ and
$\Phi_n$.  If $x\in S_n(M,\varphi)$ and $y\in S_1(M,\varphi)$, then
\[
 \|xy\|_\varphi\le \|x\|\|y\|_\varphi\le n\|y\|_\varphi
\]
and
\[
 \|yx\|_\varphi\le \|x\|_{\mathrm{right}}\|y\|_\varphi
 \le n\|y\|_\varphi,
\]
and the same estimates hold with $x$ replaced by $x^*$.  Hence
\begin{equation}\label{eq:TB-upper-bound}
 0\le \operatorname{TB}_{k,n}(x)\le n.
\end{equation}
Consequently,
\[
 0\le \Phi_n(x)
 \le
 \sum_{k=1}^{\infty}\frac{2^{-k}}{n}(n-1)
 =
 \frac{n-1}{n}<1
\]
for $n\ge2$ (and $\Phi_1=0$).

For $N\in\mathbb N$, set
\[
 \Phi_{n,N}(x)
 :=
 \sum_{k=1}^{N}
 \beta_{k,n}
 \bigl(\operatorname{TB}_{k,n}(x)\dminus1\bigr).
\]
Then \eqref{eq:TB-upper-bound} gives
\begin{equation}\label{eq:Phi-tail}
 0\le
 \Phi_n(x)-\Phi_{n,N}(x)
 \le
 \frac{n-1}{n}\sum_{k>N}2^{-k}
 =
 \frac{n-1}{n}2^{-N}
 \le 2^{-N}.
\end{equation}
In particular, $\Phi_n$ is the effective uniform limit of the
finite sums $\Phi_{n,N}$.

We will also use the following uniform continuity estimate.  If
$x,x'\in S_n(M,\varphi)$ and $y\in S_1(M,\varphi)$, then
\[
 \|(x-x')y\|_\varphi\le \|x-x'\|_\varphi,
 \qquad
 \|y(x-x')\|_\varphi\le \|x-x'\|_\varphi.
\]
Since
\[
 \max\{\|y\|_\varphi,1/k\}\ge 1/k,
\]
it follows that
\[
 \left|
 B^L_{k,n}(x)-B^L_{k,n}(x')
 \right|
 \le k\|x-x'\|_\varphi
 \le k\sqrt2\,\|x-x'\|_\varphi^\#,
\]
and the same estimate holds for $B^R_{k,n}$ and for the corresponding
expressions with $x$ replaced by $x^*$.  Therefore
\begin{equation}\label{eq:TB-Lipschitz}
 \left|
 \operatorname{TB}_{k,n}(x)
 -
 \operatorname{TB}_{k,n}(x')
 \right|
 \le
 k\sqrt2\,\|x-x'\|_\varphi^\#.
\end{equation}
Since $r\mapsto r\dminus1$ is $1$-Lipschitz and
$\sum_{k=1}^{\infty}k2^{-k}=2$, we conclude that
\begin{equation}\label{eq:Phi-Lipschitz}
 |\Phi_n(x)-\Phi_n(x')|
 \le
 \frac{2\sqrt2}{n}\,
 \|x-x'\|_\varphi^\#.
\end{equation}
Thus $\Phi_n$ is a uniformly continuous definable predicate, with a
modulus independent of the $W^*$-probability space.

In Section~\ref{subsec:effectivealpha}, we will show the following effective computability result. 
\begin{thm}
\label{thm:effective-lower-sort-test}
For every $n\in\mathbb{N}$ and every rational $d\in(0,1]$, one can effectively
compute an integer $k=k(n,d)$ and a positive rational
$\delta=\delta(n,d)$ such that, for every $W^*$-probability space
$(M,\varphi)$ and every $x\in S_n(M,\varphi)$, we have
\[
 \operatorname{TB}_{k,n}(x)\le 1+\delta
 \quad\Longrightarrow\quad
 \operatorname{dist}_{\|\cdot\|_\varphi^\#}
 \bigl(x,S_1(M,\varphi)\bigr)<d.
\]
\end{thm}
We then use Theorem~\ref{thm:effective-lower-sort-test} to prove the following result, which will be essential for the effective axiomatization of W$^*$-probability spaces: 
\begin{thm}
\label{thm:computable-alpha}
There is a uniformly computable family $(\alpha_n)_{n\geq1}$ of
continuous nondecreasing functions
$\alpha_n:[0,1]\longrightarrow[0,n]$ with $\alpha_n(0)=0$ such that,
for every $n\geq1$, every $W^*$-probability space $(M,\varphi)$,
and every $x\in S_n(M,\varphi)$, we have
\begin{equation}
 \operatorname{dist}_{\|\cdot\|_\varphi^\#}
 \bigl(x,S_1(M,\varphi)\bigr)
 \le\alpha_n\bigl(\Phi_n(x)\bigr).
 \label{eq:alpha-bound}
\end{equation}
\end{thm}
\begin{remark}
For notational convenience, we extend $\alpha_n$ to 
$[0,\infty)$ by setting
\[
\alpha_n(t):=\alpha_n(1)\qquad(t\geq1).
\]
We continue to denote this extension by $\alpha_n$.
\end{remark}
We fix an index set $I$,
$(\bm{M},\bm{\varphi})=(M_i,\varphi_i)_{i\in I}$ a family of $W^*$-probability spaces, and a
non-principal ultrafilter $\mathcal{U}$ on $I$.  
In Section~\ref{subsec:SKdefinabilitynew}, we will show that every totally $K$-bounded element in the Ocneanu ultraproduct $(M,\varphi)=\prod_{\cU}(M_i,\varphi_i)$ can be represented by a family of totally $K$-bounded elements in $\ell^{\infty}(\bm{M})$, and vice versa. More precisely, we show the following:
\begin{thm}
\label{thm:SK-liftingnew}
For every $K>0$, the following equality holds:
\begin{equation}
   S_K\!\left(\prod_{\cU}(M_i,\varphi_i)\right)
   =
   \prod_{\cU}S_K(M_i,\varphi_i).\label{eq:SKdefinablenew}
\end{equation}
\end{thm}

For model-theoretic reasons, we refer to this goal as aiming to show that the set of totally $K$-bounded elements forms a definable set. It is straightforward to show the inclusion $\supseteq$ in (\ref{eq:SKdefinablenew}) (Lemma~\ref{represents}). The converse inclusion $\subseteq$ is, however, much more delicate and it relies crucially on Theorem~\ref{thm:effective-lower-sort-test}.\\

Before proceeding further, we explain how a na\"ive attempt at proving the hard part $\subseteq$ by kernel smoothing might proceed, the pitfalls it faces, and how our proof bypasses these potential problems. 
Fix $x = (x_i)_{\mathcal{U}} \in S_K(M,\varphi)$.
By the density of $(M_i)_{\mathrm{tb}}$ in $M_i$, if $\cU$ is countably incomplete, we may
assume that each $x_i$ satisfies $\|x_i\|\le K$ and is totally bounded, or even 
$\sigma^{\varphi_i}$-analytic. However, a priori we may have no bound on $\|x_i\|_{\mathrm{tb}}$,
though $\|x\|_{\mathrm{tb}}\le K$ holds. 

It is tempting to apply some kernel smoothing.
For example, let $F_a$ ($a>0$) be the Fej\'er kernel, and consider the Fej\'er smoothing
of $x$:
$$
  x_a := \sigma^{\varphi}_{F_a}(x)
  = \int_{\mathbb{R}} F_a(t)\sigma^{\varphi}_t(x)\,dt.
$$
Then because the $F_a$-smoothing commutes with the Ocneanu ultraproduct, we have
$$
  x_a = (x_{i,a})_{\mathcal{U}},\quad x_{i,a} = \sigma^{\varphi_i}_{F_a}(x_i).
$$
Then
$$
  \lim_{a\to\infty}\|x_a - x\|^{\#}_{\varphi} = 0,
$$
and because the Fourier transform of $F_a$ has compact support,
$\mathrm{supp}(\widehat{F_a})\subset[-a,a]$, we have
$x_{i,a}\in M_i(\sigma^{\varphi_i},[-a,a])$, which by \cite[Proposition~2.14]{AG2} implies that
$$
  x_{i,a}\in S_{Ke^{a/2}}(M_i,\varphi_i).
$$
Thus, $x_{i,a}$ is almost totally $K$-bounded if $a\sim 0$, but we need $a\to\infty$ in
order to get $x_i\sim x_{i,a}$, and then we would lose the control of $\|x_{i,a}\|_{\mathrm{tb}}$.
This difficulty is not specific to the Fej\'er kernel. Kernel smoothing improves the modular
regularity of a chosen representative, or gives spectral control depending on the smoothing
parameter, but it does not reflect the original total $K$-boundedness of the element in the
Ocneanu ultraproduct. Our proof below overcomes this difficulty by showing that sufficiently good finite-level control of the left and right multiplication norms,
as measured by $\operatorname{TB}_{k,n}$, forces an element to lie uniformly close to $S_1$.

In the Appendix, we will give an alternative proof of this result by a completely different method, namely the $H^{\infty}$-lifting of the analytic extension of the modular orbit $t\mapsto \sigma_t^{\varphi}(x)$.  

\subsection{A few preliminary lemmas}

\begin{lem}\label{lem:tb-algebra}
Let $K,L>0$. Then the inclusion $$S_K(M,\varphi)S_L(M,\varphi)\subseteq S_{KL}(M,\varphi)$$ holds. Moreover, $S_K(M,\varphi)$ is convex, balanced, and invariant under the
adjoint operation and under the real modular automorphism group.
\end{lem}

\begin{proof} Let $x\in S_K(M,\varphi)$ and $y\in S_L(M,\varphi)$. Then $\|xy\|\le \|x\|\|y\|\le KL$. 
Since
$R_{xy}=R_yR_x$ and
$R_{(xy)^*}=R_{x^*}R_{y^*}$, we have $\|R_{xy}\|\le \|R_x\|\|R_y\|\le KL$ and $\|R_{(xy)^*}\|\le KL$ as well.
Therefore, $xy\in S_{KL}(M,\varphi)$ holds. Most of the last assertions are immediate. For the final assertion, we use the fact that $\sigma_t^{\varphi}$ leaves $D(\sigma_{\pm \ri/2}^{\varphi})$ invariant and commutes with $\sigma_{\pm \ri/2}^{\varphi}$ on these domains. 
\end{proof}

\begin{lem}
\label{lem:modular-continuity}
Let $a=a^*\in S_n(M,\varphi)$ and let $R>0$.  Then, for every $t\in\mathbb{R}$, we have
\begin{equation}
 \|\sigma_t^\varphi(a)-a\|_\varphi
 \le nR|t|+4ne^{-R/2}.
 \label{eq:modular-continuity}
\end{equation}
\end{lem}

\begin{proof}
By $a=a^*\in S_n(M,\varphi)$, we have 
\[
 \|a\om_{\varphi}\|,\quad \|\Delta_{\varphi}^{\frac12}a\om_{\varphi}\|=\|\sigma_{-\ri/2}^{\varphi}(a)\om_{\varphi}\|,\quad \|\Delta_{\varphi}^{-\frac12}a\om_{\varphi}\|\le n.
\]
Consequently, the spectral projections $E_+=1_{(R,\infty)}(\log \Delta_{\varphi})$ and $E_-=1_{(-\infty,-R)}(\log \Delta_{\varphi})$ satisfy 
\[
 \|E_+a\om_{\varphi}\|\le ne^{-R/2},
 \qquad
 \|E_-a\om_{\varphi}\|\le ne^{-R/2}.
\]
On $[-R,R]$ one has $|e^{\ri ts}-1|\le R|t|$. 
Therefore, for $E:=1-E_+-E_-$ and $t\in \mathbb{R}$, we have 
\[
\|(\Delta_{\varphi}^{\ri t}-1)Ea\om_{\varphi}\|\le R|t|\|Ea\om_{\varphi}\|\le nR|t|,
\]
and 
\[
\begin{aligned}
    \|(\Delta_{\varphi}^{\ri t}-1)E_\pm a\om_{\varphi}\|&\le 2\|E_{\pm}a\om_{\varphi}\|\\
    &\le 2ne^{-R/2}.
\end{aligned}
\]
Therefore, for every $t\in \mathbb{R}$ we have
\[
\begin{aligned}
    \|\sigma_t^{\varphi}(a)-a\|_{\varphi}&=\|(\Delta_{\varphi}^{\ri t}-1)a\om_{\varphi}\|\\
    &\le nR|t|+4ne^{-R/2}.
\end{aligned}\qedhere
\]
\end{proof}

We shall also use the following elementary effective approximation fact.

\begin{lem}\label{lem:effective-sqrt}
Given a positive rational $\varepsilon$, one can effectively find a real
polynomial $p\in\mathbb{Q}[t]$ such that
\[
 \sup_{0\le t\le1}|p(t)-\sqrt t|<\varepsilon.
\]
\end{lem}

\begin{proof}This is a consequence of a well-known polynomial approximation method. 
For $N\in\mathbb N$, consider the Bernstein polynomial for $\sqrt{t}$, given by 
\[
 B_N(t)
 =
 \sum_{k=0}^N
 \sqrt{\frac{k}{N}}
 \binom{N}{k}t^k(1-t)^{N-k},
 \qquad 0\le t\le1.
\]
We first give an explicit estimate for the error.

Fix $t\in[0,1]$, and let $X$ be a binomial random variable with
parameters $(N,t)$. Thus, 
\[\mathbb{P}(X=k)
=
\binom{N}{k}t^k(1-t)^{N-k}, \qquad k=0,\dots,N,\]
and $\mathbb{E}(X)=Nt, \operatorname{Var}(X)=Nt(1-t).$ Then $B_N(t)=\mathbb E\left[\sqrt{\frac XN}\right]$.
Using $|\sqrt{s}-\sqrt{t}|
 \le \sqrt{|s-t|},\,(s,t\ge0)$, 
Cauchy--Schwarz, and Jensen's inequality, we obtain
\[
\begin{aligned}
 |B_N(t)-\sqrt t|
 &\le
 \mathbb E\left[
   \left|\sqrt{\frac XN}-\sqrt t\right|
 \right]                                      
 \le
 \mathbb E\left[
   \sqrt{\left|\frac XN-t\right|}
 \right]                                      \\
 &\le
 \left(
   \mathbb E\left|\frac XN-t\right|
 \right)^{1/2}                                 \le
 \left(
   \mathbb E\left|\frac XN-t\right|^2
 \right)^{1/4}                                \\
 &=
 \left(\frac{t(1-t)}{N}\right)^{1/4}
 \le
 \frac{1}{\sqrt{2}\,N^{1/4}}.
\end{aligned}
\]
Consequently,
\[
 \sup_{t\in [0,1]}|B_N(t)-\sqrt{t}|
 \le
 \frac{1}{\sqrt{2}\,N^{1/4}}.
\]

Given $\varepsilon\in\mathbb Q_{>0}$, choose $N$ effectively so large
that $\frac{1}{\sqrt{2}\,N^{1/4}}<\frac{\varepsilon}{2}$; for example, any integer $N>4\varepsilon^{-4}$ suffices.
For each $0\le k\le N$, choose a rational number $r_k$ such that
$\left|r_k-\sqrt{\frac{k}{N}}\right|
 <\frac{\varepsilon}{2}$. 
Set
\[
 p(t)
 =
 \sum_{k=0}^N
 r_k\binom{N}{k}t^k(1-t)^{N-k}.
\]
Then $p\in\mathbb Q[t]$. Since the Bernstein polynomials are
nonnegative on $[0,1]$ and sum to $1$, we have
\[
 \|p-B_N\|_\infty
 \le
 \max_{0\le k\le N}
 \left|r_k-\sqrt{\frac{k}{N}}\right|
 <\frac{\varepsilon}{2}.
\]
Therefore
\[
 \sup_{0\le t\le1}|p(t)-\sqrt t|
 <
 \varepsilon.
\]
All choices are effective, which proves the lemma.
\end{proof}
The next proposition will play a key role. 
\begin{prop}
\label{prop:selfadjoint-clipping}
Let $a=a^*\in S_n(M,\varphi)$.  Define the normal positive functional $\omega_a$ on $M$ by
\[
 \omega_a(x)=\varphi(axa), \qquad x\in M,
\]
and let
\[
 \chi:=\omega_a-\varphi=\chi_+-\chi_-
\]
be the Jordan decomposition of $\chi\in M_*$.  Set $q:=\chi_+(1)$. 
Then there exists $b=b^*\in S_1(M,\varphi)$ such that
\begin{equation}
 \|a-b\|_\varphi\le(n+1)\sqrt q.
 \label{eq:selfadjoint-clipping-distance}
\end{equation}
Consequently,
\begin{equation}
 \operatorname{dist}_{\|\cdot\|_\varphi^\#}
 \bigl(a,S_1(M,\varphi)\bigr)\ge d
 \quad\Longrightarrow\quad
 q\ge\frac{d^2}{(n+1)^2}.
 \label{eq:q-lower}
\end{equation}
\end{prop}

\begin{proof}
For each $\lambda\in M_*^+$ and $c\in M$, define the linear functional $\lambda_c$ on $M$ by 
\[
   \lambda_c(x)=\lambda(c^*xc),\qquad x\in M.
\]
Set $\mu=\varphi+\chi_+$.
We have $\varphi\le\mu$ and $\omega_a\le\mu$. 
By Sakai's quadratic Radon--Nikodym theorem \cite[Theorem 2]{Sakai65RN} applied to
$\varphi\le\mu$, there exists a positive contraction $h\in M$ such that
\begin{equation}
   \varphi(x)=\mu(hxh),
   \qquad x\in M,
   \label{eq:Sakai-functional}
\end{equation}
or, in other words, $\mu_h=\varphi$. 
Since $\varphi\le\mu$, for every $x\in M_+$ we have
\[
   \varphi_h(x)
   =\varphi(hxh)
   \le\mu(hxh)
   =\varphi(x),
\]
and hence $\varphi_h\le \varphi$. 

Likewise, since $\omega_a\le\mu$, we have $(\omega_a)_h\le \mu_h=\varphi.$ 

Set $b=hah$. Then $b=b^*$.  We first show that $\varphi_b\le \varphi$. 

Indeed, if $x\in M_+$, then $ahxha\ge0$, and therefore, by $\varphi_h\le \varphi$ and $(\omega_a)_h\le \varphi$, 
\[
\begin{aligned}
   \varphi(bxb)
   &=\varphi(hahxhah)=\varphi_h(ahxha)\\
   &\le\varphi(ahxha)=\omega_a(hxh)\\
   &=(\omega_a)_h(x)\\
   &\le\varphi(x).
\end{aligned}
\]
This proves $\varphi_b\le \varphi$. For $x\in M$,
\[
\begin{aligned}
   \|xb\|_\varphi^2
   &=\varphi(bx^*xb)
   =\varphi_b(x^*x)\\
   &\le\varphi(x^*x)=\|x\|_\varphi^2.
\end{aligned}
\]
Hence $\|R_b\|\le1$.  Since $b=b^*$, the same estimate gives
$\|R_{b^*}\|\le1$.

We next show that $\|b\|\le1$.  Suppose otherwise.  Then there is
$\varepsilon>0$ such that
\[
   p=1_{(1+\varepsilon,\infty)}(|b|)
\]
is nonzero.  Since $b=b^*$, $|b|=(b^2)^{\frac{1}{2}}\in W^*(b)$. Therefore, $p\in W^*(b)$, whence $p$ commutes with $b$. Thus, 
\[
   bpb=|b|^2p\ge(1+\varepsilon)^2p.
\]
On the other hand, $\varphi_b\le \varphi$ gives
$\varphi(bpb)\le\varphi(p)$. 
Consequently,
\[
   (1+\varepsilon)^2\varphi(p)
   \le\varphi(bpb)
   \le\varphi(p),
\]
which contradicts the faithfulness of $\varphi$, since $p\ne0$.
Therefore $\|b\|\le1$, and hence
$b\in S_1(M,\varphi)$. 

Next, we estimate $\|a-b\|_{\varphi}$. Since $0\le h\le1$, $(1-h)^2\le1-h^2$ holds.
Using \eqref{eq:Sakai-functional} at $x=1$, we have
$\mu(h^2)=\varphi(1)=1$. 
Also,
\[
   \mu(1)
   =\varphi(1)+\chi_+(1)
   =1+q.
\]
Therefore
\[
\begin{aligned}
   \mu((1-h)^2)
   &\le\mu(1-h^2)=\mu(1)-\mu(h^2)\\
   &=q.
\end{aligned}
\]

Since $\omega_a\le\mu$,
\[
\begin{aligned}
   \|(1-h)a\|_\varphi
   &=\varphi\bigl(a(1-h)^2a\bigr)^{\frac12}=\omega_a((1-h)^2)^{\frac12}\\
   &\le\mu((1-h)^2)^{\frac12}\\
   &\le \sqrt{q}.
\end{aligned}
\]
Similarly, since $\varphi\le\mu$,
\[
\begin{aligned}
   \|1-h\|_\varphi
   &=\varphi((1-h)^2)^{\frac12}\le\mu((1-h)^2)^{\frac12}\\
   &\le \sqrt{q}.
\end{aligned}
\]
From $a-b
   =a-hah
   =(1-h)a+ha(1-h)$, 
and from the preceding estimates it follows that, using $\|h\|\le1$ and $\|a\|\le n$,
\[\begin{aligned}
    \|a-b\|_{\varphi}&\le \|(1-h)a\|_{\varphi}+\|ha\|\|(1-h)\|_{\varphi}\\
    &\le (n+1)\sqrt{q}.
\end{aligned}\]
Since both $a$ and $b$ are self-adjoint, $\|a-b\|_\varphi^\#=\|a-b\|_\varphi$ holds. Therefore,
\[
   \operatorname{dist}_{\|\cdot\|_\varphi^\#}
   \bigl(a,S_1(M,\varphi)\bigr)
   \le(n+1)\sqrt q,
\]
and consequently
\[
   \operatorname{dist}_{\|\cdot\|_\varphi^\#}
   \bigl(a,S_1(M,\varphi)\bigr)\ge d
   \quad\Longrightarrow\quad
   q\ge\frac{d^2}{(n+1)^2}.\qedhere
\]
\end{proof}

\subsection{An effective distance estimate to $S_1$ for self-adjoint elements}
We now turn the lower bound on $q$ into a totally $1$-bounded test element
of controlled $L^2$-norm.

\paragraph{\textbf{Convention on effectivity.}}
Throughout the effective arguments below, the input parameters are
$n\in\mathbb N$ and $d\in\mathbb Q\cap(0,1]$. All real constants
constructed from $(n,d)$ are uniformly computable from these parameters.
Whenever below we choose a rational number strictly below a positive
computable real, or a rational number strictly above a computable real,
such a choice is understood to be made effectively.

\begin{thm}\label{thm:selfadjoint-witness}
Let $n\ge2$ and let $d\in \mathbb{Q}\cap (0,1]$.  From $(n,d)$ one can effectively compute
positive real numbers $\mu_n^{\mathrm{sa}}(d)>0$ and $\eta_n^{\mathrm{sa}}(d)>0$ 
with the following property:  If $a=a^*\in S_n(M,\varphi)$ and
\[
 \operatorname{dist}_{\|\cdot\|_\varphi^\#}
 \bigl(a,S_1(M,\varphi)\bigr)\ge d,
\]
then there exists a self-adjoint $z\in S_1(M,\varphi)$ such that
\[
 \|z\|_\varphi\ge\mu_n^{\mathrm{sa}}(d),\qquad  \text{and}\qquad 
 \|za\|_\varphi
 \ge
 \bigl(1+\eta_n^{\mathrm{sa}}(d)\bigr)\|z\|_\varphi.
\]
\end{thm}

\begin{proof}
Set
\begin{equation}
 q_0=\frac{d^2}{(n+1)^2},
 \qquad
 m_0=\frac{q_0}{n^2-1}.
 \label{eq:q0-m0}
\end{equation}
Let $\chi:=\omega_a-\varphi=\chi_+-\chi_-$ be the Jordan decomposition and $q:=\chi_+(1)$ be as in 
Proposition~\ref{prop:selfadjoint-clipping}.  Then $q\ge q_0$.
Let $e$ be the support projection of $\chi_+$.  Since the supports of the
positive and negative parts in the Jordan decomposition are orthogonal, we have
\[
 q=(\omega_a-\varphi)(e)
 =\varphi(aea)-\varphi(e).
\]
Set $m:=\varphi(e)$. The right $n$-boundedness of $a$ gives $\varphi(aea)\le n^2\varphi(e)$,
and hence 
\begin{equation}
 m\ge\frac q{n^2-1}\ge m_0.
 \label{eq:m-lower}
\end{equation}

We now regularize $e$.  Define
\begin{equation}
 R=2\log\left(\frac{64n^2}{q_0}\right),
 \qquad
 \tau=\frac{\pi q_0^2}{1024n^4R^2},
 \label{eq:R-tau}
\end{equation}
and set $e_\tau=\sigma_{k_\tau}^\varphi(e)$. 
Let $q_\tau:=\varphi(ae_\tau a)-\varphi(e_\tau)$. 
Since $\varphi(e_\tau)=\varphi(e)$, we have
\[
 q_\tau-q
 =\int_\mathbb{R} k_\tau(t)
 \left[
 \varphi\bigl(a\sigma_t^\varphi(e)a\bigr)-\varphi(aea)
 \right]dt.
\]
By the $\sigma^{\varphi}$-invariance of $\varphi$, we have
\[
 \varphi\bigl(a\sigma_t^\varphi(e)a\bigr)
 =\varphi\bigl(\sigma_{-t}^\varphi(a)e\sigma_{-t}^\varphi(a)\bigr).
\]
For self-adjoint elements $c,c'\in M$ with $\|c\|,\|c'\|\le n$, the Cauchy--Schwarz inequality gives
\[
\begin{aligned}
    |\varphi(cec)-\varphi(c'ec')|&\le |\varphi((c-c')ec)|+|\varphi(c'e(c-c'))|\\
    &\le \|(c-c')^*\|_{\varphi}\|ec\|_{\varphi}+\|(ec')^*\|_{\varphi}\|c-c'\|_{\varphi}\\
    &\le2n\|c-c'\|_\varphi. 
\end{aligned}
\]
Together with Lemma~\ref{lem:modular-continuity}, this yields
\begin{equation}
 |q_\tau-q|
 \le4n^2R\sqrt{\frac\tau\pi}+8n^2e^{-R/2}.\label{eq:qt-q}
\end{equation}
Indeed, we use (\ref{eq:modular-continuity}) and (\ref{eq:R-tau}) to obtain 
\[
\begin{aligned}
|q_{\tau}-q|&\le \int_{\mathbb{R}}k_{\tau}(t)|\varphi(\sigma_{-t}^{\varphi}(a)e\sigma_{-t}^{\varphi}(a))-\varphi(aea)|dt\\
&\le \int_{\mathbb{R}}k_{\tau}(t)\cdot 2n\|\sigma_{-t}^{\varphi}(a)-a\|_{\varphi}dt\\
&\le \int_{\mathbb{R}}k_{\tau}(t)\cdot 2n(nR|t|+4ne^{-R/2})dt\\
&=4n^2R\sqrt{\frac{\tau}{\pi}}+8n^2e^{-R/2}\\
&\le \frac{q_0}{4}.
\end{aligned}
\]
Here, we used 
\[
\int_{\mathbb{R}} k_\tau(t)|t|\,dt
=
2\sqrt{\frac{\tau}{\pi}}.
\]
Therefore $q_{\tau}\ge \frac{1}{2}q_0$. 

Let $u=e_\tau^{1/2}$. Then
$\|u\|_\varphi^2=\varphi(e_\tau)=m$, and 
\[
 \|ua\|_\varphi^2
 =\varphi(ae_\tau a)
 =m+q_\tau
 \ge m+\frac{q_0}{2}.
\]
Since $m\le1$, we obtain
\begin{equation}
 \|ua\|_\varphi
 \ge(1+\rho_0)\|u\|_\varphi,
 \qquad
 \rho_0=\sqrt{1+\frac{q_0}{2}}-1.
 \label{eq:u-excess}
\end{equation}
Moreover, \eqref{eq:m-lower} gives
\[
 \|u\|_\varphi\ge s_0,
 \qquad
 s_0=\sqrt{m_0}.
\]

The element $u$ need not be totally bounded, so we replace it by a polynomial
in $e_\tau$.  

Choose $\varepsilon_0\in\mathbb Q_{>0}$ such that
\[
 \varepsilon_0<
 \min\left\{\frac{s_0}{4},
 \frac{\rho_0s_0}{8(n+2)}\right\}.
\]
By Lemma~\ref{lem:effective-sqrt}, choose effectively a
polynomial $
 p(t)=\sum_{j=0}^N c_jt^j\in \mathbb{Q}[t]$ 
such that
$\sup_{0\le t\le1}|p(t)-\sqrt t|<\varepsilon_0$. 
Set $w:=p(e_\tau)$. Then $w=w^*$, $\|w-u\|\le\varepsilon_0$, and 
 $\|w\|_\varphi\ge\frac{s_0}{2}$. The right $n$-boundedness of $a$ gives $\|(w-u)a\|_\varphi\le n\varepsilon_0.$ 
Therefore, 
\[
\begin{aligned}
    \|wa\|_{\varphi}&\ge \|ua\|_{\varphi}-\|(u-w)a\|_{\varphi}\\
    &\ge (1+\rho_0)\|u\|_{\varphi}-n\varepsilon_0.
\end{aligned}
\]
On the other hand, 
\[
\begin{aligned}
    \left(1+\frac{\rho_0}{2}\right)\|w\|_\varphi
 &\le
 \left(1+\frac{\rho_0}{2}\right)
 (\|u\|_\varphi+\varepsilon_0).
\end{aligned}
\]
Note that 
\begin{equation}
    \left(1+\frac{\rho_0}{2}\right)
 (\|u\|_\varphi+\varepsilon_0)\le (1+\rho_0)\|u\|_{\varphi}-n\varepsilon_0.\label{eq:wainequality} 
\end{equation}
Indeed, since $0<\rho_0<1$, 
\[\begin{aligned}
    (1+\rho_0)\|u\|_{\varphi}&-n\varepsilon_0-\left(1+\frac{\rho_0}{2}\right)
 (\|u\|_\varphi+\varepsilon_0)\\
 &=\frac{\rho_0}{2}\|u\|_{\varphi}-n\varepsilon_0-\left (1+\frac{\rho_0}{2}\right ) \varepsilon_0\\
 &\ge \frac{\rho_0s_0}{2}-(n+1+\tfrac{\rho_0}{2})\varepsilon_0\\
 &\ge \frac{\rho_0s_0}{2}-(n+2)\cdot \frac{\rho_0s_0}{8(n+2)}\\
 &=\frac{3\rho_0s_0}{8}>0.
\end{aligned}\]
Therefore, we obtain
\begin{equation}
 \|wa\|_\varphi
 \ge\left(1+\frac{\rho_0}{2}\right)\|w\|_\varphi.
 \label{eq:w-excess}
\end{equation}
By Lemma~\ref{lem:gauss-total-bounded}, $e_\tau\in S_{e^{1/(16\tau)}}(M,\varphi)$ holds.  
Lemma~\ref{lem:tb-algebra} gives
\[
 \|w\|_{\tb}
 \le D,
 \qquad
 D:=\sum_{j=0}^N|c_j|e^{\frac{j}{16\tau}}.
\]
Effectively choose an integer $D_0\ge\max\{1,D\}$ and set $z=D_0^{-1}w$. Then $z=z^*\in S_1(M,\varphi)$,
\[
 \|z\|_\varphi
 \ge\frac{s_0}{2D_0},
\]
and \eqref{eq:w-excess} is unchanged by scaling.  Thus we may take
\begin{equation}
 \mu_n^{\mathrm{sa}}(d)=\frac{s_0}{2D_0},
 \qquad
 \eta_n^{\mathrm{sa}}(d)=\frac{\rho_0}{2}.
 \label{eq:selfadjoint-constants}
\end{equation}
We note that all choices were made effectively in $(n,d)$.
\end{proof}

\subsection{The $2\times2$ matrix amplification}
Set $(N,\psi):=\bigl(M_2(M),\operatorname{tr}_2\otimes\varphi\bigr)$, where $\operatorname{tr}_2$ is the normalized matrix trace. 
Then for $X=\begin{pmatrix}x_{11}&x_{12}\\x_{21}&x_{22}\end{pmatrix}\in N$, the following norm equality holds: 
\begin{equation}
\|X\|_{\psi}^2=\frac12\sum_{i,j=1}^2\|x_{ij}\|_{\varphi}^2.
\end{equation}
For $x\in M,$
set
\[
 \widehat x=
 \begin{pmatrix}
 0&x\\ x^*&0
 \end{pmatrix}.
\]

\begin{lem}\label{lem:matrix-reduction}
For every $K>0$, the equality $\|x\|_{\mathrm{tb}}=\|\widehat{x}\|_{\mathrm{tb}}$ holds. In particular, 
\begin{equation}
 x\in S_K(M,\varphi)
 \quad\Longleftrightarrow\quad
 \widehat x\in S_K(N,\psi).\label{eq:S_Kequiv}
\end{equation}
Moreover, $\|\widehat x\|_\psi=\|x\|_\varphi^\#$ 
and 
\begin{equation}
 \operatorname{dist}_{\|\cdot\|_\psi^\#}
 \bigl(\widehat x,S_1(N,\psi)\bigr)
 =
 \operatorname{dist}_{\|\cdot\|_\varphi^\#}
 \bigl(x,S_1(M,\varphi)\bigr).
 \label{eq:matrix-distance}
\end{equation}
Finally, if $Z=(z_{ij})\in S_1(N,\psi)$, then every matrix entry
$z_{ij}$ belongs to $S_1(M,\varphi)$.
\end{lem}

\begin{proof}
We have 
\[\|\widehat{x}\|=\|\widehat{x}^*\widehat{x}\|^{\frac12}=\left \|\begin{pmatrix}xx^*&0\\0&x^*x\end{pmatrix}\right \|^{\frac12}=\|x\|,\]
and similarly, $\|\widehat{x}\|_{\psi}^{\#}=\|x\|_{\varphi}^{\#}$ holds. 
Let $X=\begin{pmatrix}a&b\\c&d\end{pmatrix}\in N.$ 
Then 
\[\begin{aligned}
    \|X\widehat{x}\|_{\psi}^2&=\left \|\begin{pmatrix}
        bx^*&ax\\dx^*&cx
    \end{pmatrix}\right \|_{\psi}^2\\
    &=\frac12(\|bx^*\|_{\varphi}^2+\|dx^*\|_{\varphi}^2+\|ax\|_{\varphi}^2+\|cx\|_{\varphi}^2)\\
    &\le \|x\|_{\mathrm{tb}}^2\|X\|_{\psi}^2.
\end{aligned}\]
This shows that $\|R_{\widehat{x}}\|=\|R_{\widehat{x}^*}\|\le \|x\|_{\mathrm{tb}}$. Therefore, $\|\widehat{x}\|_{\mathrm{tb}}\le \|x\|_{\mathrm{tb}}$ holds. 

On the other hand, 
\[\begin{aligned}
    \|R_{\hat{x}}\|^2\left \|\begin{pmatrix}a
    &0\\0&0\end{pmatrix}\right \|_{\psi}^2
    &\ge \left \|\begin{pmatrix}
        0&ax\\0&0
    \end{pmatrix}\right \|_{\psi}^2\\
    &=\frac12\|ax\|_{\varphi}^2.\end{aligned}\]
which implies that $\|R_x\|\le \|R_{\widehat{x}}\|$. A similar argument using $\begin{pmatrix}0&b\\0&0\end{pmatrix}$ gives $\|R_{x^*}\|\le \|R_{\widehat{x}}\|$. Therefore, $\|\widehat{x}\|_{\mathrm{tb}}=\|x\|_{\mathrm{tb}}$ holds. 
The equivalence (\ref{eq:S_Kequiv}) then follows. 

We now show (\ref{eq:matrix-distance}). 
Set $\widehat{C}:=\operatorname{dist}_{\|\cdot\|_\psi^\#}
 \bigl(\widehat x,S_1(N,\psi)\bigr),\,C:=
 \operatorname{dist}_{\|\cdot\|_\varphi^\#}
 \bigl(x,S_1(M,\varphi)\bigr)$, and take $\varepsilon>0$. Choose $y\in S_1(M,\varphi)$ such that $\|x-y\|_{\varphi}^{\#}<C+\varepsilon$. 
 Then by the first part, we have \[\|\widehat{x}-\widehat{y}\|_{\psi}^{\#}=\|x-y\|_{\varphi}^{\#}<C+\varepsilon.\]
 This shows that $\widehat{C}\le C+\varepsilon$. Since $\varepsilon$ is arbitrary, $\widehat{C}\le C$ holds. 
 To show $C\le \widehat{C}$, define  
\[
 v=\begin{pmatrix}1&0\\0&-1\end{pmatrix},
 \qquad
 \gamma(A)=vAv^*,\qquad A\in N.
\]
Define the real-linear map $\theta\colon N\to N$ by 
\[
 \theta (A)
 =\frac14\bigl(A-\gamma(A)+A^*-\gamma(A^*)\bigr).
\]
For $A=\begin{pmatrix}a&b\\c&d\end{pmatrix}$ and $X=\begin{pmatrix}x&y\\z&w\end{pmatrix}$, 
\[\gamma(A)=\begin{pmatrix}a&-b\\-c&d\end{pmatrix},\,\qquad R_{\gamma(A)}(X)=\begin{pmatrix}xa-yc &-xb+yd\\za-wc&-zb+wd\end{pmatrix}.\]
In particular, $\|\gamma(A)\|_{\psi}=\|A\|_{\psi}$ holds.
Since $\gamma$ is $\psi$-preserving, it induces a unitary
$U_\gamma$ on $L^2(N,\psi)$ by
\[
U_\gamma(x\omega_\psi)=\gamma(x)\omega_\psi, \qquad x\in N.
\]
For $A\in S_1(N,\psi)$ and $x\in N$, we have $U_\gamma R_AU_\gamma^*(x\omega_\psi)
=x\gamma(A)\omega_\psi$, 
and hence $R_{\gamma(A)}=U_\gamma R_AU_\gamma^*$. 
Therefore
\[
\|R_{\gamma(A)}\|=\|R_A\|,
\qquad
\|R_{\gamma(A)^*}\|=\|R_{A^*}\|.
\]
Together with $\|\gamma(A)\|=\|A\|$, this shows that
$\gamma(S_1(N,\psi))=S_1(N,\psi)$.

Therefore, $\theta$ is contractive for the $\|\cdot\|_\psi^\#$-metric. Moreover, $\gamma(\widehat{x})=-\widehat{x}$, whence $\theta(\widehat{x})=\widehat{x}$,
and $\theta$ maps $S_1(N,\psi)$ into itself, because $S_1(N,\psi)$ is convex and invariant under
$A\mapsto A^*$ (Lemma \ref{lem:tb-algebra}), $A\mapsto\gamma(A)$, and multiplication by $-1$. 
Now take $\varepsilon>0$. Choose $A=\begin{pmatrix}a&b\\c&d\end{pmatrix}\in S_1(N,\psi)$ such that $\|\widehat{x}-A\|_{\psi}^{\#}<\widehat{C}+\varepsilon.$ 
Note that 
\[\theta(A)=\frac{1}{2}\begin{pmatrix}0&b+c^*\\b^*+c&0\end{pmatrix}=\widehat{y},\qquad y:=\frac{b+c^*}{2}.\] 
From (\ref{eq:S_Kequiv}), $\widehat y\in S_1(N,\psi)$ implies
$y\in S_1(M,\varphi)$, and 
\[\|x-y\|_{\varphi}^{\#}=\|\widehat{x}-\widehat{y}\|_{\psi}^{\#}<\widehat{C}+\varepsilon.\]
Therefore, as $\varepsilon$ is arbitrary, we obtain $C\le \widehat{C}$. This proves \eqref{eq:matrix-distance}.

If $Z=(z_{ij})\in S_1(N,\psi)$, then $\|z_{ij}\|\le1$. 
For $A=\begin{pmatrix}a&0\\0&0\end{pmatrix}\in N$, we have 
\[
    \|R_Z(A)\|_{\psi}^2=\left \|\begin{pmatrix}az_{11} & az_{12}\\0& 0\end{pmatrix}\right \|_{\psi}^2
=\frac{1}{2}(\|az_{11}\|_{\varphi}^2+\|az_{12}\|_{\varphi}^2),\]
which implies $\max\{\|R_{z_{11}}\|,\|R_{z_{12}}\|\}\le \|R_Z\|\le 1.$ 

Similarly, using $A=\begin{pmatrix}0 &0\\0 & d\end{pmatrix}$, we obtain 
\[\max\{\|R_{z_{21}}\|,\|R_{z_{22}}\|\}\le \|R_Z\|\le 1.\]
Applying the same argument to $Z^*$ gives $\|R_{z_{ij}^*}\|\le 1$ for $i,j=1,2.$ 
Hence $z_{ij}\in S_1(M,\varphi)$ for $i,j=1,2$.
\end{proof}

Combining the above matrix amplification trick with the self-adjoint case treated in the previous subsection, we have:

\begin{thm}
\label{thm:general-witness}
For every integer $n\ge2$ and every rational $d\in(0,1]$, one can
effectively compute positive real numbers $\mu_n(d)>0$ and $\eta_n(d)>0$ such that the following holds.  If $x\in S_n(M,\varphi)$ and  $\operatorname{dist}_{\|\cdot\|_\varphi^\#}
 \bigl(x,S_1(M,\varphi)\bigr)\ge d$, 

then there exists $z\in S_1(M,\varphi)$ such that
\[
 \|z\|_\varphi\ge\mu_n(d),
\qquad \text{and}\qquad 
 \max\{\|zx\|_\varphi,\|zx^*\|_\varphi\}
 \ge(1+\eta_n(d))\|z\|_\varphi.
\]
\end{thm}

\begin{proof}
We apply Theorem~\ref{thm:selfadjoint-witness} in $(N,\psi)$ to
$X=\widehat x$, noting that, by Lemma~\ref{lem:matrix-reduction}, the distance hypothesis
is preserved.  We obtain $Z=(z_{ij})\in S_1(N,\psi)$ such that
\[
 \|Z\|_\psi\ge\mu_n^{\mathrm{sa}}(d)
\]
and
\[
 \|Z\widehat{x}\|_\psi
 \ge(1+\eta_{\mathrm{sa}})\|Z\|_\psi,
 \qquad
 \eta_{\mathrm{sa}}=\eta_n^{\mathrm{sa}}(d).
\]
Note that every entry of $Z$ belongs to $S_1(M,\varphi)$.

Put
\[
 t=1+\eta,
 \qquad
 \eta=\min\left\{\frac{\eta_{\mathrm{sa}}}{2},
 \frac{n-1}{2}\right\}.
\]
Thus $1<t<n$.  Write
\[
 A^2=\sum_{i,j=1}^2\|z_{ij}\|_\varphi^2
 =2\|Z\|_\psi^2.
\]
The entries of $Z\widehat{x}$ are of the form $z_{ij}x$ or $z_{ij}x^*$, and each of
their $L^2$-norms is at most $n\|z_{ij}\|_\varphi$. 
For $i,j\in \{1,2\}$, set $a_{ij}:=\|z_{ij}\|_\varphi$.
Set
\[
 w_{11}:=z_{11}x,\qquad
 w_{12}:=z_{12}x^*,\qquad
 w_{21}:=z_{21}x,\qquad
 w_{22}:=z_{22}x^*.
\]
Then the four $w_{ij}$ are precisely the four entries of
$Z\widehat x$, up to permutation, and hence
\[
 \sum_{i,j}\|w_{ij}\|_\varphi^2
 =2\|Z\widehat x\|_\psi^2.
\]

Define 
\[G:=\{(i,j)\in \{1,2\}^2: \|w_{ij}\|_{\varphi}\ge t\|z_{ij}\|_{\varphi}\}.\]
By the
right $n$-boundedness of $x$ and $x^*$,
$\|w_{ij}\|_\varphi\le n a_{ij}$ for every $(i,j)$.  For $(i,j)\notin G$,  $\|w_{ij}\|_\varphi<t a_{ij}$ holds.  
Writing
\[
   S_G=\sum_{(i,j)\in G}a_{ij}^2,
\]
we therefore obtain
\[
\begin{aligned}
   2\|Z\widehat{x}\|_\psi^2
   &=\sum_{i,j}\|w_{ij}\|_\varphi^2\\
   &\le
   n^2\sum_{(i,j)\in G}a_{ij}^2
   +t^2\sum_{(i,j)\notin G}a_{ij}^2\\
   &=n^2S_G+t^2(A^2-S_G)\\
   &=t^2A^2+(n^2-t^2)S_G.
\end{aligned}
\]
On the other hand, we have $2\|ZX\|_\psi^2
   \ge(1+\eta_{\rm sa})^2A^2$.  
Combining the two estimates gives
\[
   \bigl((1+\eta_{\rm sa})^2-t^2\bigr)A^2
   \le (n^2-t^2)S_G.
\]
Since $1<t<n$, we have $n^2-t^2>0$, and hence
\[
   S_G
   \ge
   \frac{(1+\eta_{\rm sa})^2-t^2}
        {n^2-t^2}A^2.
\]
Thus, with
\[
   \Gamma
   :=
   \frac{(1+\eta_{\rm sa})^2-(1+\eta)^2}
        {n^2-(1+\eta)^2},
\]
we obtain
\[
   \sum_{(i,j)\in G}\|z_{ij}\|_\varphi^2
   \ge \Gamma A^2.
\]
It follows that for some $(i,j)\in G$,
$\|z_{ij}\|_\varphi^2
   \ge \frac{\Gamma A^2}{4}$. 
Since $A=\sqrt{2}\|Z\|_\psi\ge \sqrt{2}\mu_n^{\mathrm{sa}}(d)$,
we have  
\[
   \|z_{ij}\|_\varphi
   \ge
   \frac{\sqrt{\Gamma}}{2}
   \sqrt{2}\,\mu_n^{\mathrm{sa}}(d)
   =
   \mu_n^{\mathrm{sa}}(d)\sqrt{\frac{\Gamma}{2}}.
\]
For this $(i,j)$, $w_{ij}$ is either of the form $z_{ij}x$ or $z_{ij}x^*$ and has norm
at least $(1+\eta)\|z_{ij}\|_\varphi$.  Therefore we may take
\begin{equation}
 \mu_n(d)=\mu_n^{\mathrm{sa}}(d)\sqrt{\frac\Gamma2},
 \qquad
 \eta_n(d)=\eta.
 \label{eq:general-constants}
\end{equation}
Note that these quantities are effectively computable from $(n,d)$.
\end{proof}

\subsection{Proofs of Theorems~\ref{thm:effective-lower-sort-test} and~\ref{thm:computable-alpha}}\label{subsec:effectivealpha}
\begin{proof}[Proof of Theorem~\ref{thm:effective-lower-sort-test}]
For $n=1$ the assertion is trivial, since the ambient sort is already $S_1$.
Assume $n\ge2$.  Apply Theorem~\ref{thm:general-witness} with the prescribed
$d$.  Effectively choose positive rational lower bounds
\[
 0<\underline\mu<\mu_n(d),
 \qquad
 0<\underline\eta<\eta_n(d).
\]
Set
$k=\left\lceil\frac{1}{\underline{\mu}}\right\rceil$ and $\delta=\frac12\underline{\eta}$. 

Suppose that
\[
 \operatorname{dist}_{\|\cdot\|_\varphi^\#}
 \bigl(x,S_1(M,\varphi)\bigr)\ge d.
\]
Theorem~\ref{thm:general-witness} gives $z\in S_1(M,\varphi)$ with
\[
 \|z\|_\varphi\ge\mu_n(d)>\underline\mu\ge\frac1k
\]
and 
\[
 \max \{\|zx\|_\varphi,\|zx^*\|_{\varphi}\}\ge(1+\eta_n(d))\|z\|_\varphi.
\]
 Hence
\[
 \operatorname{TB}_{k,n}(x)\ge1+\eta_n(d)>1+\delta.
\]
Taking the contrapositive proves the theorem.
\end{proof}

We now show how to use the previous theorem to obtain expressions that will become formalizable in the language presented in Section \ref{SectionLanguage}.

Note that $0\le\Phi_n(x)<1$ 
and that the series converges uniformly and effectively (the tail after the $N$-th
term is at most $2^{-N}$).
\begin{proof}[Proof of Theorem~\ref{thm:computable-alpha}]
For $n=1$, take $\alpha_1=0$.  Let $n\ge2$.
By Theorem~\ref{thm:effective-lower-sort-test}, compute an integer $k_m$ and
a positive rational $\delta_m$ such that
\[
 \operatorname{TB}_{k_m,n}(x)\le1+\delta_m
 \quad\Longrightarrow\quad
 \operatorname{dist}_{\|\cdot\|_\varphi^\#}(x,S_1)<2^{-m}.
\]
Set
$s_m:=\beta_{k_m,n}\delta_m>0$. Choose recursively a strictly decreasing sequence of positive rational numbers
$(\tau_m)_m$ as follows:
\[
 \tau_1=\min\{s_1/2,1/2\},\qquad \tau_{m+1}=\frac{1}{2}\min\{s_{m+1},\tau_m\}.
\]
If $\Phi_n(x)\le\tau_m$, then the $k_m$-th summand in
\eqref{eq:Phi-n} gives
\[
 \beta_{k_m,n}\bigl(\operatorname{TB}_{k_m,n}(x)\dminus 1\bigr)
 \le\Phi_n(x)\le s_m,
\]
so $\operatorname{TB}_{k_m,n}(x)\le1+\delta_m$.

Therefore
\begin{equation}
 \Phi_n(x)\le\tau_m
 \quad\Longrightarrow\quad
 \operatorname{dist}_{\|\cdot\|_\varphi^\#}(x,S_1)<2^{-m}.
 \label{eq:threshold-implication}
\end{equation}

Define $\alpha_n(0)=0$, set
$\alpha_n(\tau_1)=n$, $\alpha_n(\tau_m)=2^{-(m-1)}\quad(m\ge2)$, 
let $\alpha_n$ be  affine on every interval
$[\tau_{m+1},\tau_m]$, and set $\alpha_n(t)=n$ for
$t\in[\tau_1,1]$.  Then $\alpha_n$ is continuous and non-decreasing.  
We show that
$\alpha_n$ is computable. 
For $M\ge2$, define a rational piecewise-linear function
$\alpha_{n,M}:[0,1]\to[0,n]$ as follows.  On $[\tau_M,1]$, let
$\alpha_{n,M}$ agree with $\alpha_n$, and on $[0,\tau_M]$ let
$\alpha_{n,M}$ be the affine function joining the points
$(0,0)$ and $\bigl(\tau_M,2^{-(M-1)}\bigr)$.
Thus
\[
 \alpha_{n,M}(t)
 =
 \frac{2^{-(M-1)}}{\tau_M}t,
 \qquad 0\le t\le\tau_M.
\]
Since all of the numbers $\tau_1,\ldots,\tau_M$ are rational and are computable from $n$ and $M$, the function
$\alpha_{n,M}$ is an effectively computable rational piecewise-linear
function.

Moreover, since $\alpha_n$ is nondecreasing and
$\alpha_n(\tau_M)=2^{-(M-1)}$, 
we have
\[
 0\le\alpha_n(t)\le2^{-(M-1)}
 \qquad(0\le t\le\tau_M).
\]
The same estimate holds for $\alpha_{n,M}(t)$ on $[0,\tau_M]$.
Since $\alpha_{n,M}=\alpha_n$ on $[\tau_M,1]$, it follows that
\[
 \|\alpha_{n,M}-\alpha_n\|_\infty
 \le 2^{-(M-1)}.
\]
Hence $\alpha_n$ is the effective uniform limit of the rational
piecewise-linear functions $(\alpha_{n,M})_{M\ge2}$.  In particular,
$\alpha_n$ is a computable continuous function.
Let $t=\Phi_n(x)$.  If $t\ge\tau_1$, then
\[
 \operatorname{dist}_{\|\cdot\|_\varphi^\#}(x,S_1)
 \le\|x\|_\varphi^\#\le n=\alpha_n(t),
\]
since $0\in S_1$.  If
$t\in[\tau_{m+1},\tau_m]$, then
\eqref{eq:threshold-implication} gives distance less than $2^{-m}$, while
\[
 \alpha_n(t)\ge\alpha_n(\tau_{m+1})=2^{-m}.
\]
Finally, if $t=0$, then \eqref{eq:threshold-implication} holds for every $m$,
so the distance is zero.  This proves \eqref{eq:alpha-bound}.
\end{proof}

\subsection{Proof of Theorem~\ref{thm:SK-liftingnew}}\label{subsec:SKdefinabilitynew}
First, we show the easier inclusion, namely that $S_K(M,\varphi)\supseteq \prod_{\cU}S_K(M_i,\varphi_i)$: 
\begin{lem}\label{represents}
Fix $K>0$.
Suppose that, for each $i\in I$, we have a totally $K$-bounded element $a_i\in S_{{K}}(M_i,\varphi_i)$.  Then we have that $(a_i)_{i\in I}\in \mathcal{M}^\cU$ and the corresponding element $(a_i)_\cU$ of the Ocneanu ultraproduct $M$ is totally {$K$}-bounded.
\end{lem}

\begin{proof}
By scaling, we may and do assume that $K=1$.
We have $\|(a_i)_i\|=\sup_i\|a_i\|\le 1$ and so $(a_i)_i \in \ell^\infty(\bm{M})$. Since $a_i$ is
    totally $1$-bounded, we have 
    \[
    \begin{aligned}
        \|a_im_i\|_{\varphi_i}^{\#\,2}
        &=
        \frac{
            \|a_im_i\|_{\varphi_i}^2
            +
            \|m_i^*a_i^*\|_{\varphi_i}^2
        }{2}                                                    \\
        &\leq
        \frac{
            \|a_i\|^2\|m_i\|_{\varphi_i}^2
            +
            \|a_i^*\|_{\rt}^2\|m_i^*\|_{\varphi_i}^2
        }{2}                                                     \\
        &\leq
        \|m_i\|_{\varphi_i}^{\#\,2}.
    \end{aligned}
    \]
    Similarly, $\|m_ia_i\|_{\varphi_i}^{\#\,2}\le \|m_i\|_{\varphi_i}^{\#\,2}$ holds. 
    Therefore, $(a_im_i)_i,\ (m_ia_i)_i\in\mathcal I_{\mathcal U}$, whence $(a_i)_i\in\mathcal M^{\mathcal U}$ holds. 
    Then the element $a:=(a_i)_{\mathcal U}\in M$ satisfies $\|a\|\leq\lim_{\mathcal U}\|a_i\|\leq1$. 
    For $x=(x_i)_{\mathcal U}\in M$, we have
    \[\|xa\|_{\varphi}=\lim_{i\to \mathcal{U}}\|x_ia_i\|_{\varphi_i}\le \lim_{i\to \mathcal{U}}\|x_i\|_{\varphi_i}=\|x\|_{\varphi},\]
    whence $a$ is right 1-bounded. Applying the same argument to $a^*=(a_i^*)_{\mathcal{U}}$ shows that $a^*$ is right 1-bounded, and thus $a\in S_1(M,\varphi)$. 
\end{proof}
The proof of the converse inclusion $\subseteq$ is much more delicate, and it relies crucially on Theorem~\ref{thm:effective-lower-sort-test}. We need one more lemma.

\begin{lem}\label{lem:finite-TB-transfer}
Let $(M,\varphi)=\prod_{\cU}(M_i,\varphi_i)$ be as above, 
$N\in\mathbb N$, and suppose that 
\[
   b=(b_i)_{\cU}\in S_1(M,\varphi),
   \qquad
   b_i\in S_N(M_i,\varphi_i)
\]
for every $i\in I$.  Then, for every $k\in\mathbb N$,   $\displaystyle \lim_{i\to\cU}\operatorname{TB}_{k,N}(b_i)\leq 1.$
\end{lem}

\begin{proof}
We treat the four terms in the definition of
$\operatorname{TB}_{k,N}$ separately.

We first claim that $\displaystyle \lim_{i\to\cU} B^R_{k,N}(b_i)\leq1.$ 
Suppose this is not the case. Then there are $\epsilon>0$ and a set $J\in\cU$ such that $B^R_{k,N}(b_i)>1+\epsilon$ for every $i\in J$. Then for $i\in J$, choose $y_i\in S_1(M_i,\varphi_i)$ such that 
\[
   \|y_i b_i\|_{\varphi_i}
   >
   (1+\epsilon)
   \max\{\|y_i\|_{\varphi_i},1/k\}.
\]
Set $y_i=0$ for $i\notin J$.  By Lemma~\ref{represents},
$y=(y_i)_{\cU}\in S_1(M,\varphi)$ holds. 
Passing to the ultralimit and using
$b=(b_i)_{\cU}$ gives
\[
\begin{aligned}
   \|yb\|_\varphi
   &\ge
   (1+\epsilon)
   \lim_{i\to\cU}
   \max\{\|y_i\|_{\varphi_i},1/k\}  \\
   &=
   (1+\epsilon)
   \max\{\|y\|_\varphi,1/k\}
   >
   \|y\|_\varphi.
\end{aligned}
\]
This contradicts the fact that $b$ is right $1$-bounded.

The same argument, using that $b^*$ is right $1$-bounded, gives $\displaystyle\lim_{i\to\cU} B^R_{k,N}(b_i^*)\leq1.$ 

For the left multiplication terms, suppose that $\displaystyle
   \lim_{i\to\cU} B^L_{k,N}(b_i)>1.$ 
Arguing as above, we obtain $y=(y_i)_{\cU}\in S_1(M,\varphi)$ such that $\|by\|_\varphi>\|y\|_\varphi.$ 
But $\|b\|\leq1$, since $b\in S_1(M,\varphi)$, and hence
\[
   \|by\|_\varphi\leq\|b\|\|y\|_\varphi\leq\|y\|_\varphi,
\]
a contradiction.  Applying the same argument to $b^*$ gives
$\displaystyle \lim_{i\to\cU} B^L_{k,N}(b_i^*)\leq1.$ 
Taking the maximum of the four estimates proves the lemma.
\end{proof}

\begin{proof}[Proof of Theorem~\ref{thm:SK-liftingnew}]
Let $(M,\varphi):=\prod_{\cU}(M_i,\varphi_i).$ By scaling, it suffices to consider $K=1$.  The inclusion
$\displaystyle \prod_{\cU}S_1(M_i,\varphi_i)
   \subseteq
   S_1(M,\varphi)$ is proved in Lemma~\ref{represents}.  We prove the converse. Fix $x\in S_1(M,\varphi)$. First, assume that $\cU$ is countably incomplete. 
  Choose a representative $a=(a_i)_i\in\mathcal M^{\cU}$ such that $x=(a_i)_{\cU}$. 
Put $A:=\sup_{i\in I}\|a_i\|<\infty$. 
For $\tau>0$, set
\[
   x_\tau:=\sigma_{k_\tau}^{\varphi}(x),
   \qquad
   b_i^{(\tau)}
   :=
   \sigma_{k_\tau}^{\varphi_i}(a_i).
\]
By Lemma~\ref{lem:smoothingpreserveSK}, we have
\[
   x_\tau\in S_1(M,\varphi),\qquad \|x_\tau-x\|_\varphi^\#\longrightarrow0
   \qquad(\tau\downarrow0).
\]
Moreover, by Corollary~\ref{cor:heat-descends}, $x_\tau=(b_i^{(\tau)})_{\cU}$ holds. By Lemma~\ref{lem:gauss-total-bounded}, $b_i^{(\tau)}\in S_{N_{\tau}}(M_i,\varphi_i)$ for each $i\in I$, where $N_\tau:=
   \max\left\{
      1,\,
      \left\lceil A e^{1/(16\tau)}\right\rceil
   \right\}$. 
Thus Lemma~\ref{lem:finite-TB-transfer} yields, for every fixed
$k\in\mathbb N$,
\begin{equation}\label{eq:finite-TB-transfer-smoothed}
   \lim_{i\to\cU}
   \operatorname{TB}_{k,N_\tau}(b_i^{(\tau)})
   \leq1.
\end{equation}

For each $m\geq1$, choose $\tau_m>0$ so small that
$\|x_{\tau_m}-x\|_\varphi^\#
   <\frac{1}{2m}$. Set $N_m:=N_{\tau_m}$.  
Apply Theorem~\ref{thm:effective-lower-sort-test} with
$n=N_m,
   \qquad
   d=\frac1{2m}$.  
We obtain $k_m\in\mathbb N$ and $\delta_m>0$ such that
\[
   \operatorname{TB}_{k_m,N_m}(z)\leq1+\delta_m
   \quad\Longrightarrow\quad
   \operatorname{dist}_{\|\cdot\|_\varphi^\#}
   (z,S_1)<\frac1{2m}
\]
in every $W^*$-probability space. By \eqref{eq:finite-TB-transfer-smoothed}, the set
\[
   I_m
   :=
   \left\{
      i\in I\,\middle|\,
      \operatorname{TB}_{k_m,N_m}
      (b_i^{(\tau_m)})
      \leq1+\delta_m
   \right\}
\]
belongs to $\cU$.  Hence, for each $i\in I_m$, choose
$c_i^{(m)}\in S_1(M_i,\varphi_i)$ such that
$\|b_i^{(\tau_m)}-c_i^{(m)}\|_{\varphi_i}^\#
   <\frac1{2m}$, 
and set $c_i^{(m)}=0$ for $i\notin I_m$. By Lemma~\ref{represents}, $c^{(m)}:=(c_i^{(m)})_{\cU}\in S_1(M,\varphi)$. 
Moreover,
\[\|c^{(m)}-x\|_\varphi^\#\leq
   \|c^{(m)}-x_{\tau_m}\|_\varphi^\#+\|x_{\tau_m}-x\|_\varphi^\#<1/m.\]
   
Thus $\displaystyle \lim_{m\to \infty}\|c^{(m)}-x\|_{\varphi}^{\#}=0$. 
Since $x=(a_i)_{\cU}$, this implies that
\[
   E_m:=
   \left\{
      i\in I\,\middle|\,
      \|c_i^{(m)}-a_i\|_{\varphi_i}^\#
      <\frac{2}{m}
   \right\}
   \in\cU.
\]
 Since $\cU$ is countably incomplete, choose a decreasing sequence
$C_1\supseteq C_2\supseteq\cdots$ in $\cU$ such that 
$\bigcap_m C_m=\emptyset$. 
Set $J_m:=C_m\cap E_1\cap\dots\cap E_m$. Then
$J_1\supseteq J_2\supseteq\dots$ is a decreasing sequence in $\cU$ with $\bigcap_mJ_m=\emptyset$. 
Define
\[
   z_i=
   \begin{cases}
      c_i^{(m)},& i\in J_m\setminus J_{m+1},\\
      0,& i\notin J_1.
   \end{cases}
\]
Then $z_i\in S_1(M_i,\varphi_i)$ for every $i$.  Moreover, the integer
$m$ determined by $i\in J_m\setminus J_{m+1}$ tends to infinity along
$\cU$, and $\|z_i-a_i\|_{\varphi_i}^\#<\frac2m.$ 
Consequently,
\[
   \lim_{i\to\cU}
   \|z_i-a_i\|_{\varphi_i}^\#=0.
\]
By Lemma~\ref{represents}, $(z_i)_i\in\mathcal M^{\cU}$, and therefore $(z_i)_{\cU}
   =
   (a_i)_{\cU}
   =
   x.$ 
Thus $x$ admits a representative with
$z_i\in S_1(M_i,\varphi_i)$ for every $i$, as required.\\
Next, we prove the theorem for countably complete $\cU$.  First, note that in this case, we have $\cM^{\cU}=\ell^{\infty}(\bm{M})$. This is well-known, but for completeness we include the standard argument. 
If $a=(a_i)_i\in\mathcal I_{\cU}$, 
then $\displaystyle \lim_{i\to \cU}\|a_i\|_{\varphi_i}^{\#}=0.$ 
Hence, for every $n\geq1$, we have 
\[
E_n=\left\{i\in I\, \,\middle|\, \ \|a_i\|_{\varphi_i}^{\#}<\frac1n\right\}\in \cU.
\]
Therefore, by countable completeness,
\[
\{i\in I \ : \  \|a_i\|_{\varphi_i}^{\#}=0\}=\bigcap_{n=1}^\infty E_n\in\mathcal U,
\]
whence $a_i=0$ for $\cU$-almost every $i$. 
It follows that
$
\cM^{\cU}=\ell^\infty(\bm{M}),
$
because every bounded sequence multiplies sequences which vanish on a
$\mathcal U$-large set. 

Let $q:\ell^\infty(\bm{M})\to M$ 
be the quotient map.  Since $q$ is a surjective unital $*$-homomorphism between unital C$^*$-algebras, it maps the closed unit ball of $\ell^{\infty}(\bm{M})$ onto the closed unit ball of $M$. Thus, we may choose a representative
$(a_i)_i\in \ell^{\infty}(\bm{M})$ of $x$ satisfying $\sup_{i\in I}\|a_i\|\le 1$. 

We next show that $a_i$ is right $1$-bounded for $\mathcal U$-almost every $i$. Assume by contradiction that this is not the case. Then for some $n\geq1$, the set
\[
A_n=
\left\{
i\in I \ : \  
a_i\text{ is not right }(1+1/n)\text{-bounded}
\right\}
\]
belongs to $\mathcal U$.  For each $i\in A_n$, choose $b_i\in M_i$ such
that
$
\|b_i\omega_{\varphi_i}\|=1$ and $
\|b_i a_i\omega_{\varphi_i}\|>1+\frac1n.$
This is possible by the definition of the right multiplication norm.

Note that the elements $b_i$ need not be uniformly bounded in operator norm.  However,
by countable completeness, they are uniformly bounded on a $\mathcal U$-large subset.  Indeed, the sets
\[
A_{n,m}=\{i\in A_n \ : \ \|b_i\|\leq m\},
\qquad m\geq1,
\]
increase to $A_n$.  Since $A_n\in\mathcal U$, countable completeness
then implies that $A_{n,m_0}\in\mathcal U$
for some $m_0$. Redefine $b_i=0$ for  $i\in I\setminus A_{n,m_0}$.  Then $(b_i)_i\in\ell^\infty(\bm{M})$,
and hence $b:=(b_i)_{\cU}\in M.$ 
Since $x$ is right $1$-bounded in $M$, we have $\|b x\om_\varphi\|\leq \|b\om_\varphi\|.$ However, by construction, \[
\|b\om_\varphi\|
=
\lim_{i\to \mathcal U}\|b_i\omega_{\varphi_i}\|
=
1,
\]
while
\[
\|b x\om_\varphi\|
=
\lim_{i\to \mathcal U}\|b_i a_i\omega_{\varphi_i}\|
\geq 1+\frac1n.
\]
This contradicts $\|b x\om_\varphi\|\leq 1.$
Therefore, for every $n\geq1$,
\[
E_{+,n}=\left\{
i\in I \ : \  
a_i\text{ is right }(1+1/n)\text{-bounded}
\right\}\in\mathcal U.
\]
Then by countable completeness of $\cU$, we obtain a 
set $E_+=\bigcap_{n=1}^{\infty}E_{+,n}\in\mathcal U$ such that
$
\|a_i\|_{\operatorname{right}}\leq 1
$
for every $i\in E_+$.

Applying the same argument to $x^*=(a_i^*)_{\cU}$, and using that $x^*$ is
right $1$-bounded, we obtain a set $E_-\in\mathcal U$ such that
$
\|a_i^*\|_{\operatorname{right}}\leq 1
$
for every $i\in E_-$. Set $$E=E_+\cap E_-\in \cU.$$
Define
\[
x_i:=
\begin{cases}
a_i, & i\in E,\\
0, & i\in I\setminus E.
\end{cases}
\]
Then for every $i\in I$, we have $x_i\in S_1(M_i,\varphi_i)$, and $(x_i)_{\cU}=(a_i)_{\cU}=x.$ This finishes the proof.
\end{proof}
\section{Axiomatizing \wstar-probability spaces}\label{SectionLanguage}  

We are now ready to introduce our language $\cL_{W^*}$ for \wstar-probability spaces:
\begin{enumerate}[label=(\textup{L\arabic*}),ref=L\arabic*]
\item For each $n \in \mathbb{N}$, there is a sort $S_n$ with bound $2n$, whose intended interpretation is the set $S_n(M)$. We let $d_n$ denote the metric symbol on $S_n$, whose intended interpretation is the metric induced by $\|\cdot\|_\vp^\#$.
\item For each $n\in \mathbb N$, there are binary function symbols $+_n$ and $-_n$ with domain $S_{n}^2$ and range $S_{2n}$ and whose modulus of uniform continuity is $\delta(\epsilon)=\epsilon/2$.  The intended interpretation of these symbols are addition and subtraction in the algebra restricted to the sort $S_n$. 
\item For each $n\in \mathbb N$, there is a binary function symbol $\cdot_n$ with domain $S_n^2$ and range $S_{n^2}$ and whose modulus of uniform continuity is $\delta(\epsilon)=\epsilon/(2n)$.  The intended interpretation of these symbols is multiplication in the algebra restricted to the sort $S_n$.
    \item For each $n\in \mathbb{N}$, there is a unary function symbol $*_n$ whose modulus of uniform continuity is $\delta(\epsilon)=\epsilon$.  The intended interpretation of these symbols is for the adjoint restricted to each sort.
    \item For each $n\in\mathbb{N}$, there are two constant symbols $0_n$ and $1_n$ which lie in the sort $S_n$. The intended interpretation of these symbols are the elements $0$ and $1$.
    \item For each $n\in \mathbb N$ and $\lambda \in \mathbb{C}$, there is a unary function symbol $\lambda_n$ whose domain is $S_n$ and range is $S_{mn}$, where $m = \max(\lceil|\lambda|\rceil,1)$, and with modulus of uniform continuity $\delta(\epsilon)=\frac{\epsilon}{|\lambda|}$ when $\lambda \not= 0$ and $\delta(\epsilon)=\epsilon$ otherwise. The intended interpretation of these symbols is scalar multiplication by $\lambda$ restricted to $S_n$. 
    For the effective version of the language, we restrict the scalar
symbols to the computable countable field $\mathbb{Q}(\ri,e)$.  In
particular, all of the scalars $e^{\pm a}$, $a\in\mathbb N$, occurring
in axiom~\ref{axiom:closedunderh_a(logD)} belong to the language.
    \item For each $m,n\in \mathbb{N}$ with $m<n$, we have unary function symbols $\iota_{m,n}$ with domain $S_m$ and range $S_n$ and whose modulus of uniform continuity is $\delta(\epsilon)=\epsilon$. The intended interpretation of these symbols is the inclusion map between the sorts.  We abbreviate $\iota_{1,n}$ by $\iota_n$.
    \item For each $n\in \mathbb N$, we have a unary predicate symbol $\vp_{n}$, whose range is $[-n,n]$ and whose modulus of uniform continuity is $\delta(\epsilon)=\frac{\epsilon}{\sqrt{2}}$ (see Lemma \ref{statecont}).  The intended interpretation of this symbol is the restriction of the state to $S_n$.  Technically speaking, there should really be two such symbols, one for the real and imaginary parts of the states, but we content ourselves to abuse notation here.
    \item For each $n\in \mathbb N$, we have function symbols $\vP_n:S_n\to S_{3n}$, $\vQ_n:S_n\to S_{3n}$, and $\vR_n:S_n\to S_{6n}$.  The moduli of uniform continuity of the first two symbols are $\delta(\epsilon) = \epsilon/\sqrt{2}$ while the latter has modulus $\delta(\epsilon)=\epsilon/(2\sqrt{2})$.  The intended interpretations of these symbols is as in the Rieffel-Van Daele theory from Section \ref{SectionRvD} (see Corollary \ref{PpreservesRb}).
\end{enumerate}
For the effective presentation, we use a fixed effectively presented family of continuous connectives containing the usual
connectives and the functions $\alpha_n$.

We associate to each \wstar-probability space $(M,\vp)$ in $W^*$-Prob an $\cL_{W^*}$-structure $\cal D(M,\vp)$, called its \textbf{dissection}, by interpreting the symbols above with their intended meanings as outlined when introducing the language.  

The goal now is to axiomatize the class of such dissections.  We consider the following $\cal L_{W^*}$-theory $T_{W^*}$.  For the sake of readability, we sometimes omit some of the inclusion mappings between sorts and/or subscripts on the predicate and function symbols.

For fixed $k,n$, the predicates
$B^L_{k,n}$, $B^R_{k,n}$, and $\operatorname{TB}_{k,n}$
are regarded as abbreviations for the corresponding
$\mathcal L_{W^*}$-formulas. In particular, for each $N$, $\Phi_{n,N}$ is an
$\mathcal L_{W^*}$-formula.

\begin{enumerate}[label=(\textup{A\arabic*}),ref=(A\arabic*)]
    \item \label{axiom:staralg}Axioms for a unital $*$-algebra. 
    \item \label{axiom:positivelinear}Axioms saying that $\vp$ is a positive linear functional satisfying $\varphi(1)=1$.   
    \item \label{axiom:algebraoperations}
Axioms saying that the connecting maps preserve the unital complex
$^*$-algebra structure and the state.  More precisely, whenever
$m<n$, the map $\iota_{m,n}:S_m\to S_n$ preserves $0$, $1$, scalar
multiplication, addition, multiplication, adjoints, and $\varphi$, and
the connecting maps are compatible:
\[
\iota_{\ell,n}\circ\iota_{m,\ell}=\iota_{m,n}
\qquad (m<\ell<n).
\]
    \item Axioms saying that the metric is derived from $\|\cdot\|_\vp^\#$.
    \item \label{axiom:S_n}Axioms requiring that the elements of $S_n$ are totally $n$-bounded. More precisely, for each $n,k\in \mathbb N$, we have the axioms
    \[\sup_{x \in S_n}\left(\sup_{y \in S_k}\max \{ \vp((xy)^*xy)\dminus n^2\vp(y^*y), \vp((yx)^*yx)\dminus n^2\vp(y^*y)\}\right)\]
    and similar axioms with $x$ replaced by $x^*$.
    \item \label{axiom:S_1toS_n}Axioms expressing that the inclusion maps are isometries and map elements of $S_1$ to 1-bounded elements of $S_n$. Specifically, for each $n\ge 2$, we add:
    \[
        \sup_{x \in S_1}  \bigl|\|\iota_{n}(x)\|_{\vp}-\|x\|_{\vp}\bigr|
    \]
    together with
    \[
    \sup_{x \in S_1}\sup_{z\in S_n} \bigl|\vp((xz)^*(xz)) -\vp((\iota_n(x)z)^*(( \iota_n(x))z))\bigr|
    \]
    and 
    \[
    \sup_{x \in S_1}\sup_{z\in S_n} \bigl|\vp((zx)^*(zx)) -\vp((z \iota_n(x))^*(z \iota_n(x)))\bigr|
    \]
    and similar axioms with $x$ replaced by $x^*$.
    \item \label{axiom:iota_n}
Axioms expressing that the inclusion $\iota_n$ has as image the set of
all $x\in S_n$ which are totally $1$-bounded.

For $n\ge2$ and $N\ge1$, put
\[
 \Phi_{n,N}(x)
 :=
 \sum_{k=1}^{N}
 \beta_{k,n}
 \bigl(\operatorname{TB}_{k,n}(x)\dminus1\bigr).
\]
We include, for every $n\ge2$ and $N\ge1$, the axiom
\begin{equation}
 \sup_{x\in S_n}
 \left(
   \inf_{y\in S_1}
   d_n\bigl(x,\iota_n(y)\bigr)
   \dminus
   \alpha_n\bigl(\Phi_{n,N}(x)+2^{-N}\bigr)
 \right)
 =0.
 \label{eq:lower-sort-axiom}
\end{equation}
The case $n=1$ is vacuous.

Indeed, by \eqref{eq:Phi-tail},
\[
 \Phi_n(x)
 \le
 \Phi_{n,N}(x)+2^{-N}.
\]
Since $\alpha_n$ is nondecreasing, Theorem~\ref{thm:computable-alpha}
implies that every dissection satisfies these axioms.

Notice also that
\[
 \Phi_{n,N}(x)
 \le
 \frac{n-1}{n}(1-2^{-N}),
\]
and hence
\[
 \Phi_{n,N}(x)+2^{-N}
 \le
 \frac{n-1+2^{-N}}{n}
 <1.
\]
Thus, in the intended structures, the constant extension of $\alpha_n$ past $1$ is never used.
    \item \label{axiom:P_nprojection}Axioms stating that $\vP_n$ is the orthogonal projection on to the set of self-adjoint elements. 
    This is taken care of by axioms expressing that $\vP_n$ is real-linear together with the following axioms:
    $$\sup_{x\in S_n} d_{3n}(\vP_{n}(x),\vP_n(x)^*) \quad \sup_{x \in S_n} d_{9n}(\vP_{3n}(\vP_n(x)), \vP_n(x))$$ and 
    $$\sup_{x \in S_n} d_{6n}\left(\vP_{2n}\left(\frac{x + x^*}{2}\right), \frac{x + x^*}{2}\right )$$
    and 
    $$\sup_{x\in S_n}\sup_{y\in S_1} |\Re(\vp((y+y^*)(x-\vP_n(x))))|.$$
    In the last axiom above, we mean the real part of $\vp$.
    \item \label{axiom:PnQnrelation} Axioms expressing the relation between $\vP_n$ and $\vQ_n$: $$\sup_{x\in S_n}d_{3n}(\vQ_n(x),i\vP_n(-ix)).$$
    \item \label{axiom:PnQnRnrelation}Axiom expressing the relation between $\vP_n$, $\vQ_n$, and $\vR_n$: $$\sup_{x\in S_n}d_{6n}(\vR_n(x),\vP_n(x)+\vQ_n(x)).$$
    \item \label{axiom:closedunderh_a(logD)}
For each $a\in\mathbb N$, axioms that express that the model is closed
under $h_a(\log\Delta)$. By Lemma~\ref{calculation}, this can be taken
care of with the following axioms (for appropriately sorted $\vR$'s):
\[
\sup_{x\in S_n}\inf_{y\in S_n}
d\left(
2\vR(2-\vR)x,
\bigl(e^{-a}(2-\vR)^2+e^a\vR^2\bigr)y
\right).
\]
\setcounter{axiomlist}{\value{enumi}}
\end{enumerate}

\begin{thm}\label{EquivalenceOfModels}
$T_{W^*}$ axiomatizes the class of dissections of \wstar-probability spaces.  Moreover, the class of models of $T_{W^*}$ is categorically equivalent to the category of \wstar-probability spaces.  In fact, this equivalence is an equivalence of \textbf{ultracategories} in the sense that, given a family $(M_i,\varphi_i)_{i\in I}$ of W$^*$-probability spaces and an ultrafilter $\cU$ on $I$, the equivalence of categories maps the model-theoretic ultraproduct $\prod_\cU \mathcal{D}(M_i,\varphi_i)$ of the dissections to the Ocneanu ultraproduct $\prod_\cU (M_i,\varphi_i)$. Finally, the set $T_{W^*}$ is an effective set of axioms.
\end{thm}
\begin{proof}
Based on the calculations in Section \ref{SectionRvD}, it is straightforward to see that, for any \wstar-probability space $(M,\vp)$, the dissection $\cal D(M,\vp)$ is a model of $T_{W^*}$.

Conversely, consider a model $A$ of the theory $T_{W^*}$.  We begin by forming the direct limit $M_0$ of the sorts $S_n(A)$ for $n \in \mathbb N$ via the embeddings $i$.   
Using the interpretations of the algebraic function symbols,
$M_0$ is a unital $^*$-algebra over the chosen scalar field.
In the countable presentation, scalar multiplication extends
uniquely from $\mathbb Q(\ri,e)$ to $\mathbb C$ by completeness
of the sorts: for $a\in S_n(A)$ and a bounded sequence
$(\lambda_j)_j\subseteq\mathbb Q(\ri)$ converging to $\lambda$,
the elements $\lambda_j a$ lie in a common sort and satisfy
\[
 d(\lambda_j a,\lambda_k a)\le n|\lambda_j-\lambda_k|.
\]
The algebra identities extend by continuity, making $M_0$
a complex unital $^*$-algebra.
Furthermore, using the predicate for the state on each sort, one can define a semi-inner product $\langle x,y \rangle := \varphi(y^*x)$ on $M_0$.  We let $H$ denote the Hilbert space separation/completion of $M_0$ with respect to this inner product.  For each $a\in M_0$, the maps $b\mapsto ab:M_0\to M_0$ and $b\mapsto ba:M_0\to M_0$ induce bounded linear operators $\pi(a):H\to H$ and $\pi'(a):H\to H$ respectively satisfying $\|\pi(a)\|,\|\pi'(a)\|\leq n$ if $a\in S_n(A)$. Note that, for $a\in S_n(A)$, we have that $\varphi(a)=\langle \hat{a},\hat{1}\rangle=\langle \pi(a)\hat{1},\hat{1}\rangle$. Note also that $\pi'(M_0)\subseteq \pi(M_0)'$. 
We note that $\pi$ is faithful.  Indeed, if $\pi(a)=0$, then also
$\pi(a^*)=0$, and hence
\[
\varphi(a^*a)=\|\pi(a)\widehat{1}\|^2=0,
\qquad
\varphi(aa^*)=\|\pi(a^*)\widehat{1}\|^2=0.
\]
By the axiom identifying the metric with the
$\|\cdot\|_\varphi^\#$-metric, it follows that
$d(a,0)=0$, whence $a=0$.
Set $M_A:=\pi(M_0)''\subseteq \cal B(H)$.  Note that the vector $\om:=\hat{1}$ is cyclic for $M_A$.  Moreover, $\pi'(M_0)\om=\pi(M_0)\om=\widehat{M_0}$ is dense in $H$ and $\pi'(M_0)\subseteq M_A'$.  Hence $\om$ is cyclic for $M_A'$ and thus separating for $M_A$.  It follows that the vector state $\varphi_A$ on $M_A$ given by $\varphi_A(x)=\langle x\om,\om\rangle$ is faithful. By construction, $\varphi_A(\pi(x))=\varphi(x),\,(x\in M_0)$.

We next claim that $\pi(M_0)\subseteq \cal B(H)$ is a good subalgebra of $M_A$, whence, by Theorem~\ref{characterization}, $\pi(M_0)=(M_A)_{\tb}$.  By definition, $\pi(M_0)$ is a strongly dense $*$-subalgebra of $M_A$. By axiom~\ref{axiom:S_n}, if $a\in S_n(A)$, then the left multiplication by $a$
and the right multiplication by both $a$ and $a^*$ induce bounded
operators on $H$ of norm at most $n$.  Thus
$\pi(a)\in S_n(M_A,\varphi_A).$
Since every element of $M_0$ belongs to some sort, it follows that $\pi(M_0)\subseteq (M_A)_{\tb}.$
Moreover, if $a\in S_n(A)$, then $\pi(a) \in S_n(\pi(M_0))$.  
We first verify that, for every $n\ge2$,
\begin{equation}
 \iota_n(S_1(A))
 =
 \{x\in S_n(A):\pi(x)\in S_1(M_A,\varphi_A)\}.
 \label{eq:S1-image-reconstruction}
\end{equation}
The inclusion $\subseteq$ follows from axiom~\ref{axiom:S_1toS_n}. Conversely, suppose that $x\in S_n(A)$ and
$\pi(x)\in S_1(M_A,\varphi_A)$. Then $\operatorname{TB}_{k,n}(x)\le1$ for every $k\in \mathbb{N}$, and hence $\Phi_{n,N}(x)=0$ for every $N\in \mathbb{N}$. 
Axiom~\ref{axiom:iota_n} therefore gives
\[
 \operatorname{dist}_{d_n}
 \bigl(x,\iota_n(S_1(A))\bigr)
 \le
 \alpha_n(2^{-N})
 \qquad(N\ge1).
\]
Since $\alpha_n$ is continuous and $\alpha_n(0)=0$, letting
$N\to\infty$ yields
\[
 \operatorname{dist}_{d_n}
 \bigl(x,\iota_n(S_1(A))\bigr)=0.
\]
By Axiom~\ref{axiom:S_1toS_n}, $\iota_n$ is an isometry, and since
$S_1(A)$ is complete, its image is closed in $S_n(A)$. Consequently
there exists $y\in S_1(A)$ such that $x=\iota_n(y)$. This proves \eqref{eq:S1-image-reconstruction}. Now suppose that $(\pi(x_m))_{m\in\mathbb N}$ is a
$\|\cdot\|_{\varphi_A}^{\#}$-Cauchy sequence in
$S_1(\pi(M_0))$. 
For each $m$, choose $n_m\geq2$ such that $x_m\in S_{n_m}(A)$.
By \eqref{eq:S1-image-reconstruction}, there is $y_m\in S_1(A)$
such that $x_m=\iota_{1,n_m}(y_m)$. 
Since the inclusion maps are compatible isometries, the metric on the direct limit restricts to the original metric on $S_1(A)$. Hence $(y_m)_m$ is Cauchy in $S_1(A)$. By completeness of $S_1(A)$, it converges to some
$y\in S_1(A)$, and hence $(\pi(x_m))_m$ converges to $\pi(y)$ in
$S_1(\pi(M_0))$.

Now we identify the interpretations of $\vP,\vQ$, and $\vR$ with the corresponding Rieffel--Van Daele operators for 
$(M_A,\varphi_A)$. Fix $x\in S_n(A)$, and let $p\in M_0$ denote 
the element represented by $\vP_n(x)$.  By
Axiom~\ref{axiom:P_nprojection}, $p=p^*$, and, for every
$z\in S_1(A)$,
\[
 \Re\varphi_A\bigl((\pi(z)+\pi(z^*))(\pi(x)-\pi(p))\bigr)=
 \Re\langle (\pi(x)-\pi(p))\omega,
              (\pi(z)+\pi(z^*))\omega\rangle=0.
\]
We claim that the real linear span of $\{(\pi(z)+\pi(z^*))\omega:z\in S_1(A)\}$ is dense in $\overline{(M_A)_{\mathrm{sa}}\omega}$. 
Indeed, if $a=a^*\in M_0$, choose $n\geq2$ such that $a\in S_n(A)$.
Then $\pi(a/n)$ is totally $1$-bounded.  
By
\eqref{eq:S1-image-reconstruction}, there is $z\in S_1(A)$ such that $\iota_n(z)=a/n$. 
Since $\iota_n$ preserves adjoints and is injective, $z=z^*$. 
Consequently, 
\[
 \pi(a)\omega=\frac n2(\pi(z)+\pi(z^*))\omega.
\]
Thus $\pi((M_0)_{\mathrm{sa}})\omega
 \subseteq
 \operatorname{span}_{\mathbb R}
 \{(\pi(z)+\pi(z^*))\omega:z\in S_1(A)\}$. 
The reverse inclusion is immediate. 
Since $\pi(M_0)$ is strongly
dense in $M_A$, its self-adjoint part is strongly dense in $(M_A)_{\mathrm{sa}}$, and hence
$\overline{\pi((M_0)_{\mathrm{sa}})\omega}
 =
 \overline{(M_A)_{\mathrm{sa}}\omega}$. Thus, the claim follows. 

Since $p=p^*$, we have
$\pi(p)\omega\in\overline{(M_A)_{\mathrm{sa}}\omega}$, while the preceding
orthogonality relation shows that
\[
 (\pi(x)-\pi(p))\omega
 \perp_{\mathbb R}
 \overline{(M_A)_{\mathrm{sa}}\omega}.
\]
Thus $\pi(p)\omega=P_{\varphi_A}\pi(x)\omega$.
Consequently, the interpretation of $\vP$ agrees with the
Rieffel--Van Daele projection $P_{\varphi_A}$ on $\pi(M_0)\omega$.

Axiom~\ref{axiom:PnQnrelation} and the identity
$Q(\xi)=iP(-i\xi)$ then show that the interpretation of $\vQ$ agrees with
$Q_{\varphi_A}$, and axiom~\ref{axiom:PnQnRnrelation} shows that
the interpretation of $\vR$ agrees with $R_{\varphi_A}=P_{\varphi_A}+Q_{\varphi_A}$.

Finally, we show that $\pi(M_0)$ is closed under $h_a(\log(\Delta))$.  
For simplicity, set $L_a(r):=e^{-a}(2-r)^2+e^ar^2$ and $T(r):=2r(2-r)$. 
For $0\le r\le2$, one has
\[
    L_a(r)
    =
    e^{-a}(2-r)^2+e^ar^2
    \geq \frac{2}{\cosh a},
\]
and hence both $L_a(R_{\varphi_A})$ and
$L_{-a}(R_{\varphi_A})$ are bounded below. Fix $x\in S_n(A)$.  By axiom~\ref{axiom:closedunderh_a(logD)}, there is a sequence $(y_m)_m$ in $S_n(A)$ such that
$\|L_a(\vR)y_m-T(\vR)x\|_{\varphi}^{\#}\longrightarrow0$. 
By the preceding identification of the interpretation of $\vR$ with
the Rieffel--Van Daele operator $R_{\varphi_A}$, it follows that
\[
 L_a(R_{\varphi_A})\pi(y_m)\omega
 \longrightarrow
 T(R_{\varphi_A})\pi(x)\omega
\]
in $H$.  Since $L_a(R_{\varphi_A})$ is bounded below,
$(\pi(y_m)\omega)_m$ is Cauchy.

On the other hand, the relation 
 $\pi\bigl((\vR(y))^*\bigr)\omega
 =
 (2I-R_{\varphi_A})\pi(y^*)\omega$
implies
\[
 \pi\bigl((L_a(\vR)y)^*\bigr)\omega
 =
 L_{-a}(R_{\varphi_A})\pi(y^*)\omega,
\]
while $T(2-r)=T(r)$ gives
\[
 \pi\bigl((T(\vR)x)^*\bigr)\omega
 =
 T(R_{\varphi_A})\pi(x^*)\omega.
\]

Thus convergence in the $d_\varphi^\#$-metric also gives
\[
 L_{-a}(R_{\varphi_A})\pi(y_m^*)\omega
 \longrightarrow
 T(R_{\varphi_A})\pi(x^*)\omega.
\]
Since $L_{-a}(R_{\varphi_A})$ is bounded below,
$(\pi(y_m^*)\omega)_m$ is Cauchy as well.

Hence $(y_m)_m$ is Cauchy in $d_\varphi^\#$. By completeness of $S_n(A)$, there is $y\in S_n(A)$ such that
$y_m\to y$ in $d_\varphi^\#$. Passing to the limit gives 
\[
    L_a(R_{\varphi_A})\pi(y)\omega=T(R_{\varphi_A})\pi(x)\omega.
\]
By Lemma~\ref{calculation}, we obtain 
\[
    \pi(y)\om=h_a(\log\Delta)\pi(x)\omega.
\]
Thus $\pi(M_0)$ is closed under $h_a(\log\Delta)$. This shows that $\pi(M_0)$ is a good subalgebra of $M_A$, whence $\pi(M_0)=(M_A)_{\mathrm{tb}}$ by Theorem~\ref{characterization}. 
We now recover the named sorts, namely $\pi(S_n(A))=S_n(M_A,\varphi_A)$ for every $n\ge 1$.  We already know that $ \pi(S_n(A))\subseteq S_n(M_A,\varphi_A)$ for every $n$.  
Conversely, let $x\in S_n(M_A,\varphi_A)$. Since
$\pi(M_0)=(M_A)_{\mathrm{tb}}$, choose $a\in M_0$ such that
$\pi(a)=x$, and place $a$ in a sort $S_N(A)$ with $N\ge\max\{n,2\}$.
Then $n^{-1}a\in S_N(A)$ represents the totally $1$-bounded element
$n^{-1}x$. By \eqref{eq:S1-image-reconstruction}, there exists
$b\in S_1(A)$ such that $n^{-1}a=\iota_{1,N}(b)$. 
Scalar multiplication by $n$ on $S_1$ has target $S_n$.
Consequently, by compatibility of scalar multiplication and inclusions,
$\pi(nb)=x$. Thus $\pi(S_n(A))=S_n(M_A,\varphi_A)$. 

At this point, we have associated to $A$ a unital $*$-algebra $M_0$,
a faithful $*$-representation $\pi\colon M_0\to\mathcal B(H)$, and the $W^*$-probability space $(M_A,\varphi_A)$ obtained from the
strong closure of $\pi(M_0)$.  The preceding arguments show, for every
$n\ge1$, that $\pi(S_n(A))=S_n(M_A,\varphi_A)$, 
and that all symbols of $\mathcal L_{W^*}$ have their intended interpretations.  Hence $A\cong\mathcal D(M_A,\varphi_A)$. 
In particular, if we start with a \wstar-probability space $(M,\vp)$ and let $A:=\cal D(M,\vp)$, we have (using that the totally bounded vectors in $M$ are dense in $M$) that $(M_A,\vp_A)=(M,\vp)$ and $\cal D(M_A,\vp_A)=\cal D(M,\vp)$.

We next observe that embeddings between $W^*$-probability spaces
correspond to embeddings between their dissections.  One direction
follows from Lemma~\ref{expectation} and the preceding identification
of the Rieffel--Van Daele operators.

Conversely, let $F\colon \mathcal D(M,\varphi)\to \mathcal D(N,\psi)$ 
be an embedding of $\mathcal L_{W^*}$-structures.  The compatibility with the connecting maps allows us to define a unital, state-preserving $*$-homomorphism
$F_0\colon M_{\mathrm{tb}}\to  N_{\mathrm{tb}}$. 
Moreover, $F_0$ preserves the $\|\cdot\|^\#_\varphi$-metric. We note that $F_0$ preserves the operator norm.  Indeed, for
$a\in M_{\mathrm{tb}}$, faithfulness of $\varphi$ gives
\[
 \|a\|
 =
 \lim_{r\to\infty}
 \varphi((a^*a)^r)^{1/(2r)}.
\]
Since $F_0$ preserves multiplication, adjoints, and the state,
\[
 \psi\bigl((F_0(a)^*F_0(a))^r\bigr)
 =
 \varphi((a^*a)^r)
\]
for every $r$, and hence $\|F_0(a)\|=\|a\|$. 
Since $M_{\mathrm{tb}}$ is strongly dense in $M$ and $F_0$ preserves both the operator norm and the $\|\cdot\|_\varphi^\#$-metric,
Kaplansky's density theorem and completeness of the closed operator-norm balls for the $\|\cdot\|_\psi^\#$-metric imply that
$F_0$ extends uniquely to a unital state-preserving $*$-homomorphism $F\colon M\to N$. 
$F$ is normal because $\psi\circ F=\varphi$ and both $\varphi$ and $\psi$ are faithful normal states. 
For $a\in M_{\mathrm{tb}}$, the preservation of the symbol $\vP$ and
the preceding identification of its interpretation give $F(P_\varphi(a))=P_\psi(F(a))$.  
Consequently,
\[
 \begin{aligned}
 d_\varphi(a,M_{\mathrm{sa}})
 &=
 \|(a-P_\varphi(a))\omega_\varphi\|\\
 &=
 \|(F(a)-P_\psi(F(a)))\omega_\psi\|\\
 &=
 d_\psi(F(a),N_{\mathrm{sa}}).
 \end{aligned}
\]
By the strong density of $M_{\mathrm{tb}}$ and continuity of distance
to a closed subspace in the GNS Hilbert space, the same equality holds
for every $a\in M$.  Lemma~\ref{expectation} now shows that there is
a normal state-preserving conditional expectation from $N$ onto
$F(M)$.  Thus $F$ is an embedding of $W^*$-probability spaces.

We finally verify the assertion concerning ultraproducts.  Set $(M,\varphi):=\prod_{\mathcal U}(M_i,\varphi_i)$. By Theorem~\ref{thm:SK-liftingnew}, for every $n\in\mathbb N$ we have $S_n(M,\varphi)
 =
 \prod_{\mathcal U}S_n(M_i,\varphi_i)$. 
Under these identifications, the algebra operations, the connecting
maps, the constants, and the state agree with their coordinatewise 
interpretations in the model-theoretic ultraproduct.

It remains only to check the Rieffel--Van Daele symbols.  Let
$x=(x_i)_{\mathcal U}\in S_n(M,\varphi)$, where, by
Theorem~\ref{thm:SK-liftingnew}, we may take
$x_i\in S_n(M_i,\varphi_i)$ for every $i$.  Put
\[
 p_i:=P_{\varphi_i}(x_i),\qquad
 p:=(p_i)_{\mathcal U}.
\]
Then $p\in S_{3n}(M,\varphi)$ and $p=p^*$.  If
$y\in S_1(M,\varphi)$, again use
Theorem~\ref{thm:SK-liftingnew} to write
$y=(y_i)_{\mathcal U}$ with $y_i\in S_1(M_i,\varphi_i)$.  Since  $\Re\varphi_i\bigl((y_i+y_i^*)(x_i-p_i)\bigr)=0$ 
for every $i$, passage to the ultralimit yields
\[
 \Re\varphi\bigl((y+y^*)(x-p)\bigr)=0.
\]
As above, the real linear span of
$\{(y+y^*)\omega_\varphi:y\in S_1(M,\varphi)\}$ 
is dense in $\overline{M_{\mathrm{sa}}\omega_\varphi}$.
It follows that $p\omega_\varphi=P_\varphi(x\omega_\varphi)$. 
Thus $P_\varphi(x)
 =
 (P_{\varphi_i}(x_i))_{\mathcal U}$. 
The relations between $P,Q$, and $R$ then give
\[
 Q_\varphi(x)
 =
 (Q_{\varphi_i}(x_i))_{\mathcal U},
 \qquad
 R_\varphi(x)
 =
 (R_{\varphi_i}(x_i))_{\mathcal U}.
\]
Hence every symbol of $\mathcal L_{W^*}$ has the same interpretation
in $\mathcal D\left(\prod_{\mathcal U}(M_i,\varphi_i)\right)$ and $\prod_{\mathcal U}\mathcal D(M_i,\varphi_i)$.  
Therefore
\[
 \mathcal D\left(\prod_{\mathcal U}(M_i,\varphi_i)\right)
 \cong
 \prod_{\mathcal U}\mathcal D(M_i,\varphi_i),
\]
as required.

Finally, we verify the effectivity assertion.  We use the countable
signature with scalars in the computable field $\mathbb{Q}(\ri,e)$.
The scalar operations, continuity moduli, and target-sort computations
are effective.  The predicates $\Phi_{n,N}$ are finite formulas,
uniformly in $(n,N)$, and the functions $\alpha_n$ are uniformly
computable by Theorem~\ref{thm:computable-alpha}.  Moreover, the
coefficients $e^{\pm a}$ occurring in
axiom~\ref{axiom:closedunderh_a(logD)} belong to
$\mathbb{Q}(\ri,e)$ and are effective uniformly in $a$.
All the remaining axiom schemes are plainly effective.  Hence
$T_{W^*}$ is an effective set of axioms.
\end{proof}
\begin{remark}
The theory $T_{W^*}$ consists of $\forall\exists$-axioms.
For axiom~\ref{axiom:iota_n}, write $\Phi_{n,N}$ as a single supremum over $N$ independent test variables. Since $\alpha_n$ is continuous and nondecreasing and occurs on the subtracted side of a truncated difference, these supremum quantifiers become infimum quantifiers. Thus this axiom, like
axiom~\ref{axiom:closedunderh_a(logD)}, has $\forall\exists$
form; the remaining axioms are universal.
Dabrowski's theory is universal. We believe that the increase in quantifier complexity of the axiomatization is justified by having access to full multiplication.
\end{remark}

\section{Expanding the language by the modular automorphism group}\label{SectionExpanding}
In this section, we give a concrete effective expansion of $T_{W^*}$
which captures the modular automorphism group.  Our first goal is to
obtain sort-uniform polynomial approximations to the modular
automorphisms in terms of the Rieffel--Van Daele operator $R$.
These approximations will also show that the modular automorphisms
are $T_{W^*}$-definable.

To obtain an effective presentation, we add symbols at rational
times, on each named sort.  We construct the polynomial
approximations and the corresponding axioms below, and identify
their interpretations in Theorem~\ref{thm:recovermodular}.
Proposition~\ref{prop:real-time-recovery-minus} then recovers all
real times.

Let $(M,\varphi)$ be a $W^*$-probability space, $H_\varphi=L^2(M,\varphi)$ be the
GNS Hilbert space, and $\om_\varphi=\widehat 1$.
Recall that $R_{\varphi}=2(1+\Delta_{\varphi})^{-1}$ is a bounded operator satisfying $0\le R_{\varphi}\le 2$. Recall also that $\Delta_{\varphi}$ is recovered from $R_{\varphi}$ by $\Delta_{\varphi}=R_{\varphi}^{-1}(2-R_{\varphi})$. 
Because $\Delta_{\varphi}$ is unbounded in general, the spectrum $\sigma(R_{\varphi})$ may contain $0$ and $2$. Thus, one might not be able to  approximate $f_t(x)=x^{it}$ or $f_t(2-x)$ uniformly on $\sigma(R_{\varphi})$ by polynomials. We remedy this issue by some spectral cut-off arguments. 

Fix $K\ge 1$ and $x\in S_K(M,\varphi)$. Then 
because $x\in D(\sigma_{\pm \ri/2}^{\varphi})$, $\Delta_{\varphi}^{\pm \frac12}x\omega_{\varphi}=\sigma_{\mp \ri/2}^{ \varphi}(x)\omega_{\varphi}$, and therefore 
\begin{equation}
\|x\|_{\varphi}\le K,\qquad \|\Delta_{\varphi}^{\frac{1}{2}}x\omega_{\varphi}\|\le K,\qquad \|\Delta_{\varphi}^{-\frac{1}{2}}x\omega_{\varphi}\|\le K.\label{eq: SK graph bound}
\end{equation}
We will frequently use the above inequalities. 
\begin{lem}\label{lem:tail-estimate-minus}
Fix $K\ge 1,\,\, L>0$ and $x\in S_K(M,\varphi)$.  Then
\[
  \left\|
    \bigl(1-1_{[e^{-L},e^L]}(\Delta_{\varphi})\bigr)x\omega_{\varphi}
  \right\|
  \le
  \sqrt 2\, Ke^{-L/2}.
\]
The same estimate holds with $x$ replaced by $x^*$.
\end{lem}

\begin{proof} Let $E(\cdot)$ be the spectral measure associated with $\Delta_{\varphi}$. 
Then by (\ref{eq: SK graph bound}), we have
\begin{align*}
    \|1_{(e^L,\infty)}(\Delta_{\varphi})x\omega_{\varphi}\|^2&=\int_{(e^L,\infty)}d\|E(t)x\omega_{\varphi}\|^2\\
    &\le \int_{(e^L,\infty)}te^{-L}d\|E(t)x\omega_{\varphi}\|^2\\
    &\leq e^{-L}\|\Delta_{\varphi}^{\frac{1}{2}}x\omega_{\varphi}\|^2\\
    &\le K^2e^{-L},
\end{align*}
and similarly,
\begin{align*}
    \|1_{(0,e^{-L})}(\Delta_{\varphi})x\omega_{\varphi}\|^2&\le \int_{(0,e^{-L})}t^{-1}e^{-L}d\|E(t)x\omega_{\varphi}\|^2\\
    &\leq e^{-L}\|\Delta_{\varphi}^{-\frac{1}{2}}x\omega_{\varphi}\|^2\\
    &\le K^2e^{-L}.
\end{align*}
Therefore, we obtain (noting that $0$ is not an eigenvalue of $\Delta_{\varphi}$)
\[
\left\|
    \bigl(1-1_{[e^{-L},e^L]}(\Delta_{\varphi})\bigr)x\omega_{\varphi}
  \right\|^2\le 2K^2e^{-L}.
\]
The proof for $x^*$ then follows, because $x^*\in S_K(M,\varphi)$ holds. 
\end{proof}
\begin{defn}
For $t\in\mathbb R$, define $\theta_t:(0,2)\to \mathbb{C}$ by
  \[\theta_t(r)
  :=
  \left(\frac{2-r}{r}\right)^{\ri t}.
\]
Then $\Delta_{\varphi}^{\ri t}=\theta_t(R_{\varphi})
$. Moreover, $\theta_t$ is conjugate-symmetric in the sense that 
\[\theta_t(r)=\overline{\theta_t(2-r)}\]
holds for all $r\in (0,2)$.

\end{defn}
\begin{lem}\label{lem:symmetric-cutoff-minus}
Fix $t\in\mathbb R$ and $L>0$.  Set   $a_L:=\frac{2}{1+e^L}.$ Then the following statements hold.
\begin{enumerate}[label=\textup{(\arabic*)}]
    \item $1_{[e^{-L},e^L]}(\Delta_{\varphi})=1_{[a_L,2-a_L]}(R_{\varphi})$. 
    \item There exists $h_{t,L}\in C([0,2])$ such that $\|h_{t,L}\|_{\infty}\le 1$, 
\[
  h_{t,L}(r)=\theta_t(r)
  \qquad a_L\le r\le 2-a_L
\]
and
\[
  h_{t,L}(r)=\overline{h_{t,L}(2-r)}
  \qquad 0\le r\le 2.
\]
\end{enumerate}
\end{lem}

\begin{proof}
(1) is immediate since $R_{\varphi}=2(1+\Delta_{\varphi})^{-1},\,\,\Delta_{\varphi}=R^{-1}(2-R)$, and 
\[
  r=\frac{2}{1+s}\in [a_L,2-a_L]
  \iff
  s=\frac{2-r}{r}\in [e^{-L},e^L].
\]
For (2), define
\[
h_{t,L}(r):=
\begin{cases}
\theta_t(a_L), & 0\le r\le a_L,\\
\theta_t(r), & a_L\le r\le 2-a_L,\\
\theta_t(2-a_L), & 2-a_L\le r\le 2.
\end{cases}
\]
Then $h_{t,L}$ is continuous, satisfies $\|h_{t,L}\|_\infty=1$,
and is conjugate-symmetric on $[0,2]$.
\end{proof}
\begin{lem}\label{lem:symmetric-poly-minus}
Let $h\in C([0,2])$ be conjugate-symmetric on $[0,2]$. 
Then for every $\eta>0$, there exists a conjugate-symmetric polynomial $p\in \mathbb Q(\ri)[X]$ such that
\[
  \|p-h\|_{\infty}:=\max_{t\in [0,2]}|p(t)-h(t)|<\eta.
\]
\end{lem}

\begin{proof}
By the Weierstrass approximation theorem and the density of $\mathbb Q(\ri)$ in $\mathbb C$, we may choose $q\in\mathbb{Q}(\ri)[X]$ with $\|q-h\|_{\infty}<\eta.$  
Define
\[
  p(r):=\frac12\bigl(q(r)+\overline{q(2-r)}\bigr).
\]
Then $p$ is a conjugate-symmetric polynomial and, because $h$ is also conjugate-symmetric, we have 
\[
  \|p-h\|_{\infty}
  \le
  \|q-h\|_{\infty}
  <\eta.
\]
\end{proof}
For $x\in S_K(M,\varphi)$, write $R x$ for the unique element of $M_{\mathrm{tb}}$
represented by $R_\varphi x\om_\varphi$, that is,
\[
  (Rx)\om_\varphi=R_\varphi x\om_\varphi.
\]
More generally, for a polynomial $p\in\mathbb C[X]$, write $p(R)x$ for the
corresponding $M$-valued polynomial $R$-term, so that
\[
  (p(R)x)\om_\varphi=p(R_\varphi)x\om_\varphi.
\]
For $K\in\mathbb N$ and $p\in\mathbb Q(\ri)[X]$, this notation denotes an actual
$\mathcal L_{W^*}$-term after inserting the appropriate inclusion maps between sorts.  More precisely, if
\[
p(X)=\sum_{j=0}^D c_jX^j,\qquad
m_j:=\max\{1,\lceil |c_j|\rceil\},
\]
set
\[
L_D=m_DK,\qquad
L_j=2\max\{6L_{j+1},m_jK\}\quad(0\le j<D).
\]
Define recursively
\[
t_D(x)=c_Dx,\qquad
t_j(x)=\vR(t_{j+1}(x))+c_jx.
\]
Then $t_0(x)=p(R)x$. 
Thus $x\mapsto p(R)x$ is represented by an
$\mathcal L_{W^*}$-term, after inserting the appropriate inclusion
maps, and its target sort is computable effectively from $p$ and $K$.
We continue to suppress these inclusion maps from the notation.
\begin{lem}
\label{lem:conjugate-symmetric-R-polynomial-adjoint}
Suppose that $p\in \mathbb C[X]$ is conjugate-symmetric on $[0,2]$. 
Then for every $x\in S_K(M,\varphi)$, we have
\[
  (p(R)x)^*\om_\varphi
  =
  p(R)x^*\om_\varphi.
\]
\end{lem}

\begin{proof}
Let $S_\varphi=J_\varphi\Delta_\varphi^{1/2}$ be the Tomita operator.  Since
$ R_\varphi=2(1+\Delta_\varphi)^{-1},
$
we have
\[
  J_\varphi R_\varphi J_\varphi
  =
  2(1+\Delta_\varphi^{-1})^{-1}
  =
  2\Delta_\varphi(1+\Delta_\varphi)^{-1}
  =
  2-R_\varphi.
\]
Moreover $R_\varphi$ commutes with $\Delta_\varphi^{1/2}$.  Hence on $D(S_{\varphi})$, we have 
\[
  S_\varphi R_\varphi
  =
  J_\varphi\Delta_\varphi^{1/2}R_\varphi
  =
  J_\varphi R_\varphi\Delta_\varphi^{1/2}
  =
  (2-R_\varphi)J_\varphi\Delta_\varphi^{1/2}
  =
  (2-R_\varphi)S_\varphi .
\]
Since $x\in S_K(M,\varphi)$, the vector $R_\varphi x\om_\varphi$ is represented
by the bounded element $Rx\in M$.  Therefore
\[
  (Rx)^*\om_\varphi
  =
  S_\varphi(Rx)\om_\varphi
  =
  S_\varphi R_\varphi x\om_\varphi
  =
  (2-R_\varphi)S_\varphi x\om_\varphi
  =
  (2-R_\varphi)x^*\om_\varphi.
\]
Iterating this identity gives
\[
  (R^j x)^*\om_\varphi
  =
  (2-R_\varphi)^j x^*\om_\varphi.
\]

Now write $p(X)=\sum_{j=0}^d a_jX^j$.
Then

\begin{align*}
  (p(R)x)^*\om_\varphi
  &=
  \left(\sum_{j=0}^d a_jR^jx\right)^*\om_\varphi 
  =
  \sum_{j=0}^d \overline{a_j}(R^jx)^*\om_\varphi \\
  &=
  \sum_{j=0}^d \overline{a_j}(2-R_\varphi)^j x^*\om_\varphi 
  =
  \overline p(2-R_\varphi)x^*\om_\varphi .
\end{align*}

Here, $\overline p$ denotes the polynomial obtained by conjugating the
coefficients of $p$.  Since $p$ is conjugate-symmetric on $[0,2]$ we have $\overline{p(2-r)}=p(r)$ for every $0\le r\le 2$. 
Thus, by continuous functional calculus,
\[
  \overline p(2-R_\varphi)=p(R_\varphi).
\]
Consequently,
\[
  (p(R)x)^*\om_\varphi
  =
  p(R_\varphi)x^*\om_\varphi
  =
  (p(R)x^*)\om_\varphi.\qedhere
\]
\end{proof}
\begin{prop}\label{prop:sort-uniform-poly-minus}
Fix $K\ge 1$, $t\in\mathbb R$, and $\varepsilon>0$.  There exists a conjugate-symmetric 
polynomial $p=p_{t,K,\varepsilon}\in\mathbb Q(\ri)[X]$ such that, for every $W^*$-probability space $(M,\varphi)$ and every
$x\in S_K(M,\varphi)$, we have
\[
      \left\|
    \sigma_t^\varphi(x)\omega_{\varphi}-p(R_{\varphi})x\omega_{\varphi}
  \right\|<\varepsilon
\]
and
\[
      \left\|
    \sigma_t^\varphi(x)^*\omega_{\varphi}-(p(R)x)^*\omega_{\varphi}
  \right\|<\varepsilon.
\]
Consequently, viewing $\sigma_t^\varphi(x)$ and $p(R)x$ in a common
sort, we have
\[
\|\sigma_t^\varphi(x)-p(R)x\|_\varphi^\#<\varepsilon.
\]
\end{prop}

\begin{proof}
Choose $L>0$ so large so that $
  2\sqrt 2\, Ke^{-L/2}<\varepsilon/2.
$ Let $h=h_{t,L}$ be the conjugate-symmetric continuous function on $[0,2]$ from 
Lemma~\ref{lem:symmetric-cutoff-minus}.  Choose $\eta>0$ such that 
$\eta K<\varepsilon/2.$
By Lemma~\ref{lem:symmetric-poly-minus}, we may choose a conjugate-symmetric polynomial $p$
such that $
  \|p-h\|_{\infty}<\eta.$
Set $P_L:=1_{[e^{-L},e^L]}(\Delta_{\varphi}).$
On $P_LH_\varphi$, we have
\[
  h(R_{\varphi})=\theta_t(R_{\varphi})=\Delta_{\varphi}^{\ri t}.
\]
Therefore, for $x\in S_K(M,\varphi)$, we have
\[
\begin{aligned}
  \|(\Delta_{\varphi}^{\ri t}-p(R_{\varphi}))x\omega_{\varphi}\|
  &\le
  \|(\Delta_{\varphi}^{\ri t}-h(R_{\varphi}))x\omega_{\varphi}\|
  +
  \|(h(R_{\varphi})-p(R_{\varphi}))x\omega_{\varphi}\| \\
  &\le
  2\|(1-P_L)x\omega_{\varphi}\|
  +
  \eta\|x\omega_{\varphi}\| \\
  &<2\sqrt{2}Ke^{-L/2}+\eta K<
  \varepsilon.
\end{aligned}
\]
Since $\sigma_t^\varphi(x)\omega_{\varphi}=\Delta_{\varphi}^{\ri t}x\omega_{\varphi}$, 
this proves the first estimate.

For the second estimate, apply the same argument to $x^*$.  Note that $\sigma_t^{\varphi}(x)^*\omega_{\varphi}=\Delta_{\varphi}^{\ri t}x^*\omega_{\varphi}$. Since $x^*\in
S_K(M,\varphi)$ and $p$ is conjugate-symmetric on $[0,2]$, by Lemma \ref{lem:conjugate-symmetric-R-polynomial-adjoint} we have 
$(p(R)x)^*\omega_{\varphi}=p(R_{\varphi})x^*\omega_{\varphi}$, and thus the last inequality follows. 
\end{proof}
\begin{cor}
For every $t\in\mathbb R$ and $K\in\mathbb N$,
$\sigma_t^\varphi|_{S_K}$ is a $T_{W^*}$-definable function.
\end{cor}
\begin{proof}
Choose $p_j\in\mathbb Q(\ri)[X]$ as in
Proposition~\ref{prop:sort-uniform-poly-minus} with error $2^{-j}$.
For $x,y\in S_K$, the formulas
\[
 g_j(x,y):=\min\{2K,d(p_j(\vR)x,y)\},
\]
with the two arguments compared in a common target sort, converge
uniformly over all models to
$d_K(\sigma_t^\varphi(x),y)$ by the reverse triangle inequality.
Thus this distance predicate is definable.
\end{proof}
We now expand the language $\mathcal{L}_{W^*}$ by adding,  
for every $q\in\mathbb Q$ and every $K\in \mathbb{N}$,  a function symbol
\[
  \Sigma_{q,K}:S_K\to S_K,
\]
intended to denote $\sigma_q^\varphi|_{S_K}.$ The modulus of uniform continuity of each $\Sigma_{q,K}$ is the identity
modulus.
For each $q\in\mathbb Q$, $K\ge 1$, and $m\in\mathbb N$, fix a 
polynomial $
  p_{q,K,m}\in \mathbb{Q}(\ri)[X]$ which is conjugate-symmetric on $[0,2]$ 
obtained from Proposition~\ref{prop:sort-uniform-poly-minus} with $t=q,\,\, \varepsilon=2^{-m}.$
We denote the resulting expanded language by $\mathcal{L}_{W^*-\operatorname{mod}}$. Now we extend the $\mathcal{L}_{W^*}$-theory $T_{W^*}$ to the $\mathcal{L}_{W^*-\operatorname{mod}}$-theory $T_{W^*-\operatorname{mod}}$ by adding the following axioms:
\begin{equation*}
  \sup_{x\in S_K}
  d\left(
    \Sigma_{q,K}(x),
    p_{q,K,m}(\vR)x
  \right)\dminus 2^{-m}
  \axiomtag
  \label{ax:Sigma-approx}
\end{equation*}

\begin{equation*}
  \sup_{x\in S_K}
  d\bigl(\Sigma_{0,K}(x),x\bigr)
  \axiomtag
  \label{ax:Sigma-zero}
\end{equation*}

\begin{equation*}
  \sup_{x\in S_K}
  d\left(
    \Sigma_{q+r,K}(x),
    \Sigma_{q,K}\bigl(\Sigma_{r,K}(x)\bigr)
  \right)
  \axiomtag
  \label{ax:Sigma-group}
\end{equation*}

\begin{equation*}
  \sup_{x,y\in S_K}
  d\left(
    \Sigma_{q,2K}(x+y),
    \Sigma_{q,K}(x)+\Sigma_{q,K}(y)
  \right)
  \axiomtag
  \label{ax:Sigma-additive}
\end{equation*}

\begin{equation*}
  \sup_{x\in S_K}
  d\left(
    \Sigma_{q,K}(x^*),
    \Sigma_{q,K}(x)^*
  \right)
  \axiomtag
  \label{ax:Sigma-star}
\end{equation*}

\begin{equation*}
  \sup_{x,y\in S_K}
  d\left(
    \Sigma_{q,K^2}(xy),
    \Sigma_{q,K}(x)\Sigma_{q,K}(y)
  \right)
  \axiomtag
  \label{ax:Sigma-multiplicative}
\end{equation*}

\begin{equation*}
  \sup_{x\in S_K}
  \left|
    \varphi\bigl(\Sigma_{q,K}(x)\bigr)-\varphi(x)
  \right|
  \axiomtag
  \label{ax:Sigma-state}
\end{equation*}

\begin{thm}\label{thm:recovermodular}
Let $T_{W^*-\operatorname{mod}}$ be the extension of $T_{W^*}$ obtained
by adding the axioms \ref{ax:Sigma-approx}-\ref{ax:Sigma-state}.  Then any $W^*$-probability space $(M,\varphi)$ can be expanded to be a $T_{W^*-\operatorname{mod}}$ structure by the interpretation
\[
\Sigma_{q,K}^{(M,\varphi)}:=\sigma_q^\varphi|_{S_K(M,\varphi)}.
\]
Conversely, let $\mathcal A\models T_{W^*-\operatorname{mod}}$ and let
$(M,\varphi)$ be the $W^*$-probability space reconstructed from its
$T_{W^*}$-reduct.  Then the interpretation of $\Sigma_{q,K}$ in
$\mathcal A$ agrees with $\sigma_q^\varphi|_{S_K(M,\varphi)}$.  In other words, the category of models of $T_{W^*-\md}$ consists of the dissections of \wstar-probability spaces with the symbols $\Sigma_{q,K}$ interpreted as
$\sigma_q^\varphi|_{S_K(M,\varphi)}$ for $q\in\mathbb Q$ and $K\in\mathbb N$.
\end{thm}

\begin{proof}
For the forward direction, axiom~\ref{ax:Sigma-approx} follows from
Proposition~\ref{prop:sort-uniform-poly-minus}, while
the remaining
axioms follow from the group law, the fact that each
$\sigma_q^\varphi$ is a $^*$-automorphism, and the invariance
$\varphi\circ\sigma_q^\varphi=\varphi$.
For the converse, suppose $(M,\varphi)$ is a model of $T_{W^*-\operatorname{mod}}$. Fix $q\in\mathbb Q$, $K\ge 1$, and $x\in S_K(M,\varphi)$.  By axiom \ref{ax:Sigma-approx} and Proposition~\ref{prop:sort-uniform-poly-minus}, we have 
\[
  \Sigma_{q,K}^{(M,\varphi)}(x)\omega_{\varphi}
  =
  \lim_{m\to\infty}
  p_{q,K,m}(R_{\varphi})x\omega_{\varphi}=\sigma_q^{\varphi}(x)\omega_{\varphi}.
\]
Since $\om_\varphi$ is separating for $M$, it follows that $\Sigma_{q,K}^{(M,\varphi)}(x)=\sigma_q^\varphi(x).$
\end{proof}

\begin{remark}
We note that the set of axioms \ref{ax:Sigma-approx}-\ref{ax:Sigma-state} is computable.  To see this, it suffices to note that the map $(q,K,m)\mapsto p_{q,K,m}$ is computable. 
Fix $q\in\mathbb Q$, $K\geq1$, and $m\in\mathbb N$.  In the proof of
Proposition~\ref{prop:sort-uniform-poly-minus}, one may choose $L$
effectively so that $2\sqrt2\,K e^{-L/2}<2^{-m-1}$, and then choose a positive rational $\eta$ effectively so that $\eta K<2^{-m-1}$. 
With the constant extension chosen in
Lemma~\ref{lem:symmetric-cutoff-minus}, the function $h_{q,L}$ is
computable and Lipschitz with the computable bound
\[
C_{q,L}:=\frac{2|q|}{a_L(2-a_L)}.
\]
Indeed, on $[a_L,2-a_L]$,
\[
|\theta_q'(r)|
=
\frac{2|q|}{r(2-r)}
\leq C_{q,L},
\]
while $h_{q,L}$ is constant on the two complementary intervals.

For $D\geq1$, let
\[
B_Dh(r):=
\sum_{j=0}^D h_{q,L}(2j/D)\binom Dj
       (r/2)^j(1-r/2)^{D-j}.
\]
Then
\[
\|B_Dh-h_{q,L}\|_\infty
\leq \frac{C_{q,L}}{\sqrt D}.
\]
Thus $D$ can be chosen effectively so that this norm is less than
$\eta/2$.  For each $j$, choose
$c_j\in\mathbb Q(\ri)$ effectively such that
\[
|c_j-h_{q,L}(2j/D)|<\eta/2,
\]
and set
\[
s_D(r):=
\sum_{j=0}^D c_j\binom Dj
       (r/2)^j(1-r/2)^{D-j}.
\]
Since the Bernstein basis functions form a positive partition of
unity,
\[
\|s_D-B_Dh\|_\infty<\eta/2,
\]
and hence $\|s_D-h_{q,L}\|_\infty<\eta$.
Finally, set
\[
p(r):=\frac12\bigl(s_D(r)+\overline{s_D(2-r)}\bigr).
\]
Then $p\in\mathbb Q(\ri)[X]$ is conjugate-symmetric and $\|p-h_{q,L}\|_\infty<\eta$.
All choices above are effective uniformly in $(q,K,m)$.
Consequently, $(q,K,m)\mapsto p_{q,K,m}$ is computable.
\end{remark}
In order to keep the axioms effective, we only added symbols meant to capture the modular automorphism group at rational times.  We conclude this section by showing that the rational
symbols $\Sigma_{q,K}$ determine the entire modular automorphism group $\sigma_t^\varphi|_{S_K}$
for every real time $t$.  By density of the totally bounded elements, this suffices to conclude that we have successfully captured W$^*$-probability spaces expanded by their modular automorphism group.

\begin{prop}\label{prop:real-time-recovery-minus}
For each $K\ge 1$ and each W$^*$-probability space $(M,\varphi)$, the map
\[
  \mathbb R\ni t\mapsto \sigma_t^\varphi|_{S_K(M,\varphi)}
\]
is uniformly continuous for the $d_\varphi^\#$-metric and the modulus of uniform continuity is independent of $(M,\varphi)$. 
\end{prop}
\begin{proof}
Let $\varepsilon>0$. 
Choose $\delta>0$ and $L>0$ such that $LK\delta+2\sqrt{2}Ke^{-L/2}<\varepsilon.$
Take $x\in S_K(M,\varphi)$. 
Fix $s,t\in\mathbb R$ such that $|s-t|<\delta$. Set
$P_L:=1_{[e^{-L},e^L]}(\Delta_{\varphi}).$
On $P_LH_\varphi$, we have
\[
  |\lambda^{\ri t}-\lambda^{\ri s}|
  =
  |e^{\ri t\log\lambda}-e^{\ri s\log\lambda}|
  \le
  L|t-s|.
\]
Therefore, by Lemma \ref{lem:tail-estimate-minus}, we have
\[
\begin{aligned}
  \|(\Delta_{\varphi}^{\ri t}-\Delta_{\varphi}^{\ri s})x\omega_{\varphi}\|
  &\le
  L|t-s|\|x\omega_{\varphi}\|
  +
  2\|(1-P_L)x\omega_{\varphi}\| \\
  &\le
  LK|t-s|
  +
  2\sqrt2\,K e^{-L/2}<\varepsilon.
\end{aligned}
\]
The same estimate holds with $x$ replaced by $x^*$. 
This gives uniform continuity on $S_K(M,\varphi)$ in the $d_\varphi^\#$-metric.  It is clear from the proof that the modulus of uniform continuity is independent of the choice of $(M,\varphi)$. 
\end{proof}

\section{Axiomatizable and local classes}\label{SectionLocal}

In this section, we determine whether or not natural classes of \wstar-probability spaces are axiomatizable or local (defined below).  A similar endeavour for the class of tracial von Neumann algebras was undertaken in \cite{FHS3}.

We begin with a series of negative results; these results were also mentioned by Dabrowski in \cite[Section 2.3]{Dab} in his language, but we offer a few details here for the convenience of the reader.

The following is \cite[Proposition 6.3]{AH}:
\begin{fact} Let $\cU$ be a non-principal ultrafilter on $\mathbb{N}$.
There is a family $\vp_n$ of faithful normal states on the hyperfinite II$_1$ factor $\mathcal{R}$ such that $\prod_{\cU}(\mathcal{R},\vp_n)$ is \emph{not} semifinite.
\end{fact}

In what follows, given a property $P$ of von Neumann algebras, we abuse terminology and speak of the class of algebras satisfying $P$ when referring to the class of pairs \wstar-probability spaces $(M,\vp)$ where $M$ satisfies property $P$.  We also say that the class of algebras satisfying $P$ is axiomatizable if the class of dissections of pairs $(M,\vp)$ with $M$ satisfying $P$ is an axiomatizable class.

\begin{cor}
The following classes are not elementary classes:  finite algebras, semifinite algebras, $\mathrm{II}_1$ factors.
\end{cor}

\begin{prop}\label{prop:IIinftynotaxiomatizable}
The class of $\mathrm{II}_\infty$ factors is not axiomatizable.
\end{prop}

\begin{proof}
Let $M$ be a $\mathrm{II}_1$ factor with unique faithful normal tracial state
$\tau$, and consider the $\mathrm{II}_{\infty}$-factor $N:=M\bar\otimes \cal B(\ell^2)$. Let $(e_i)_{i\geq1}$ be the
standard orthonormal basis of $\ell^2$.  For each $n\in \mathbb{N}$, choose numbers
$\lambda_{n,i}>0$ such that $\sum_{i=1}^\infty \lambda_{n,i}=1$ and $\lambda_{n,1}\to 1$.  
Define a faithful normal state $\psi_n$ on $\mathcal{B}(\ell^2)$ by
\[
\psi_n(T)=
\sum_{i=1}^\infty
\lambda_{n,i}\langle Te_i,e_i\rangle,
\]
and set $\varphi_n:=\tau\otimes \psi_n$. 
Fix a non-principal ultrafilter $\mathcal U$ on $\mathbb N$ and put
$(\mathcal N,\varphi):=
\prod_{\mathcal U}(N,\varphi_n)$. Let $p:=1\otimes e_{11}\in N$. Note that $e_{11}\in \mathcal{B}(\ell^2)_{\psi_n}$, so $p\in N_{\varphi_n}$ for every $n$. Therefore, the class of the constant sequence $p$ is in $\mathcal{N}_{\varphi}$. Since $\|1-p\|_{\varphi_n}^{\#}
=
\sqrt{\varphi_n(1-p)}
\to 0$, we have $1-p=0$ in $\mathcal{N}$. Therefore, $p=1_{\mathcal N}$. 
Consequently, every element of $\mathcal N$ is equal to its $p$-corner: $
\mathcal N=p\mathcal Np$. 
Define $\Theta:\ell^\infty(\mathbb N,M)\longrightarrow \mathcal N$
by $\Theta((x_n)_n)
=
(x_n\otimes e_{11})_{\mathcal U}$. For every bounded sequence $(x_n)_n$, we have
\[
\|x_n\otimes e_{11}\|_{\varphi_n}^{\#}
=
\lambda_{n,1}^{1/2}\|x_n\|_{2,\tau}.
\]
Since $\lim_{n\to\mathcal U}\lambda_{n,1}=1$, 
it follows that
\[
\lim_{n\to\mathcal U}
\|x_n\otimes e_{11}\|_{\varphi_n}^{\#}=0
\quad\Longleftrightarrow\quad
\lim_{n\to\mathcal U}\|x_n\|_{2,\tau}=0.
\]
Thus $\Theta$ induces an injective $*$-homomorphism
$\widetilde\Theta:(M,\tau)^{\mathcal U}\longrightarrow\mathcal N$. If $x=(a_n)_{\mathcal U}\in\mathcal N$, then $
x=pxp=(pa_np)_{\mathcal U}$. For every $n\in \mathbb{N}$ there is a unique $x_n\in M$ such that $pa_np=x_n\otimes e_{11}$, 
and
\[
\|x_n\|
=
\|pa_np\|
\leq \|a_n\|,
\]
so $(x_n)_n$ is uniformly bounded.  Hence $x=\widetilde\Theta((x_n)_{\mathcal U})$. Therefore
$\mathcal N\cong (M,\tau)^{\mathcal U}$ is a $\mathrm{II}_1$ factor.  
Since an axiomatizable class is closed under ultraproducts, the class of
$\mathrm{II}_\infty$ factors is not axiomatizable.
\end{proof}

\begin{remark}
In \cite[Theorem 2.10]{Dab}, it is shown that the class of \wstar-probability spaces of the form $(N\otimes \cal B(H),\vp)$, where $N$ is a tracial von Neumann algebra and $\vp$ is a \emph{geometric state} on $N\otimes \cal B(H)$ is elementary in a natural expansion of Dabrowski's language by constants for matrix units.  Here, a geometric state on $N\otimes \cal B(H)$ means that the restriction of the state to $\cal B(H)$ has a particularly simple form.  It is also shown that, in the same expansion, the subclasses of type I$_\infty$ factors and type II$_\infty$ factors (again equipped with geometric states) are each axiomatizable.  The interested reader can readily adapt this axiomatization to our setting.
\end{remark}

As pointed out at the end of Section 3.1 of \cite{AH}, given a von Neumann algebra $M$ and two faithful normal states $\vp_1$ and $\vp_2$, we have $\prod_{\cU}(M,\vp_1)\cong \prod_{\cU}(M,\vp_2)$, whence there is a well-defined notion of the \textbf{ultrapower} of $M$.  In this case, we denote the algebra simply by $M^{\cU}$.  The following is \cite[Theorem 6.18]{AH}:

\begin{fact}\label{notfactor}
Suppose that $M$ is a $\sigma$-finite factor of type $\mathrm{III}_0$ and $\cU$ is a non-principal ultrafilter on $\mathbb{N}$.  Then $M^\cU$ is \emph{not} a factor.
\end{fact}

Recall that a class of structures is called \textbf{local} if it is closed under elementary equivalence; equivalently, the class is closed under isomorphism, ultraroot, and ultrapower.

\begin{cor}
The following classes are not local:  factors, type $\mathrm{III}_0$ factors.
\end{cor}

However:

\begin{prop}\label{prop:localclass}
The following classes are local:  finite algebras, semi-finite algebras, type $\mathrm{I}_n$ factors (for a fixed $n\in \mathbb{N}\cup \{\infty\}$), type $\mathrm{II}_1$ factors, type $\mathrm{II}_\infty$ factors.
\end{prop}

\begin{proof} Fix an ultrafilter $\cU$ on a set $I$. Since closed bounded subsets of a finite-dimensional von Neumann algebra are compact, the class of type I$_n$ factors for fixed $n\in \mathbb {N}$ is local. It is also well-known that the class of II$_1$ factors is local. 
If $(M,\varphi)$ is such that $M$ is finite, then $M^{\cU}:=(M,\varphi)^{\cU}$ is isomorphic to the tracial ultrapower $(M,\tau)^{\cU}$. Conversely, if $M^{\cU}$ is finite, then so is its von Neumann subalgebra $M$. Thus, the class  of finite algebras is local. Next, we consider the type I$_{\infty}$ factor case and the semifinite case. If $(M,\varphi)$ is such that $M$ is semifinite, then it is of the form $M=M_{1}\oplus M_2$, where $M_1$ is finite and $M_2$ is properly infinite semifinite (either summand may be $\{0\}$) and if $M_2\neq \{0\}$, it is of the form $M_2\cong pM_2p\overline{\otimes}\mathcal{B}(\ell^2)$ for a finite projection $0\neq p\in M_2$ such that $z_{M_2}(p)=1_{M_2}$. It is a folklore fact (see \cite[Lemma 2.8]{MT} for a proof; we note that although this result was stated under the separability of $M_*$, the proof actually works for arbitrary $\sigma$-finite algebras and with respect to an arbitrary ultrafilter on any index set) that $(pM_2p\overline{\otimes}\mathcal{B}(\ell^2))^{\cU}\cong (pM_2p)^{\cU}\overline{\otimes}\mathcal{B}(\ell^2)$. Consequently, $M^{\cU}\cong M_1^{\cU}\oplus M_2^{\cU}$ is semifinite. In particular, $\mathcal{B}(\ell^2)^{\cU}\cong \mathcal{B}(\ell^2)$, and if $M$ is a type II$_{\infty}$ factor, then $M_1=\{0\}$ and $pM_2p$ is a type II$_1$ factor, thus $M^{\cU}$ is a type II$_{\infty}$ factor. Conversely, assume $M^{\cU}$ is semifinite. Then so is $M$ by \cite[Theorem 3]{Tomiyama59}, being the range of a faithful normal conditional expectation $\operatorname{E}_{\cU}\colon M^{\cU}\to M$ given by the WOT-limit along $\cU$. Therefore, the class of semifinite algebras is local. If $M^{\cU}$ is a type I$_{\infty}$ factor or a type II$_{\infty}$ factor, then $\mathcal{Z}(M)\subset \mathcal{Z}(M^{\cU})=\mathbb{C}$ and $M$ is a factor, which is semifinite by the preceding argument. Thus $M$ is either type I$_n\, (n\in \mathbb{N}\cup \{\infty\})$, type II$_1$ or type II$_{\infty}$. We have shown that the class of factors of these types are closed under ultrapowers. Thus, it follows that the class of type I$_{\infty}$ (resp. type II$_{\infty}$) factors is local. This finishes the proof. 
\end{proof}

Finally, we show the following positive result.

\begin{prop}\label{prop:axiomatizableIIIlambda}
Suppose that $S\subseteq (0,1]$ is closed (as a subset of $\mathbb{R}$).  Then the set of type $\mathrm{III}_\lambda$ factors where $\lambda\in S$ is an axiomatizable class.
\end{prop}
We use the following result. The proof of \cite[Corollary 4.8 (3)]{AH} works for arbitrary ultrafilter $\cU$. 
\begin{fact}\label{fact:modular-spectrum-ultrapower}
Let $(M,\psi)$ be a $W^*$-probability space and let $\cU$ be an
arbitrary ultrafilter. Then $\sigma(\Delta_{\psi^\cU})=\sigma(\Delta_\psi)$. 
\end{fact}
\begin{proof}[Proof of Proposition~\ref{prop:axiomatizableIIIlambda}]
Let $(M_i,\varphi_i)_{i\in I}$ be W$^*$-probability spaces where each $M_i$ is a type $\mathrm{III}_{\lambda_i}$ factor with $\lambda_i\in S$. Set $(M,\varphi):=\prod_{\cU}(M_i,\varphi_i)$ and $\mathcal{N}=\prod^{\mathrm{GR}}_{\cU}M_i$. Then by Fact~\ref{fact:state space diameter} and Lemma~\ref{lem:diameter-ultraproduct}, we have
\[d(\mathcal{N})=\lim_{i\to \cU}d(M_i)=2\frac{1-\lambda^{\frac{1}{2}}}{1+\lambda^{\frac{1}{2}}},\qquad \lambda:=\lim_{i\to \cU}\lambda_i\in S.\] 
In particular, $d(\mathcal{N})<2$ and therefore $\mathcal{N}$ is a factor. Since each $M_i$ is properly infinite, we can find isometries $u_i,v_i$ such that $u_iu_i^*+v_iv_i^*=1$. These isometries then define isometries $(u_i)_{\cU}, (v_i)_{\cU}$ in $\mathcal{N}$ such that $uu^*+vv^*=1$. Therefore, $\mathcal{N}$ is properly infinite and thus it is not of finite type I. Therefore, by Fact~\ref{fact:state space diameter} again, $\mathcal{N}$ and hence its $\sigma$-finite corner $M$ is a type $\mathrm{III}_{\lambda}$ factor with $\lambda\in S$.
It remains to see that this class is closed under ultraroots.  Thus, suppose that $(M,\vp)$ is such that $M^\cU$ is a type III$_\lambda$ factor with $\lambda\in S$.  It is easy to see that $M$ must then be a factor.  $M$ cannot be semifinite for then $M^\cU$ would also be semifinite.  $M$ cannot be III$_0$; indeed, for any faithful normal state $\psi$ on $M$, we have $\sigma(\Delta_{\psi^{\cU}}
)=\sigma(\Delta_{\psi})$ by Fact~\ref{fact:modular-spectrum-ultrapower} and thus we have the following inclusion of the Connes' $S$-invariants ($S_{\mathrm{fn}}(M)$ denotes the set of all  faithful normal states on $M$):
\[S(M^{\cU})=\bigcap_{\chi\in S_{\mathrm{fn}}(M^{\cU})}\sigma(\Delta_{\chi})\subseteq \bigcap_{\psi\in S_{\mathrm{fn}}(M)}\sigma(\Delta_{\psi^{\cU}})= \bigcap_{\psi\in S_{\mathrm{fn}}(M)}\sigma(\Delta_{\psi})=S(M)=\{0,1\},\]
which contradicts the condition $S(M^{\cU})=\lambda^{\mathbb{Z}}\cup \{0\}$ if $0<\lambda<1$ and $S(M^{\cU})=[0,\infty)$ if $\lambda=1$. Thus, $M$ must be type III$_\eta$ for $\eta\in (0,1]$; by the previous paragraph, it must be that $\eta=\lambda$.
\end{proof}
\begin{remarks}

\

\begin{enumerate}
    \item Proposition~\ref{prop:axiomatizableIIIlambda} in the case that $S$ is a
singleton was stated by Dabrowski in his language in
\cite[Theorem 2.10(3)]{Dab}.
    \item In \cite[Section 4]{GoldbringHoudayer}, it was shown that the class of III$_1$ factors is $\forall\exists$-axiomatizable, whereas, 
    for $\lambda\in (0,1)$, the class of III$_\lambda$ factors is not $\forall\exists$-axiomatizable 
    nor $\exists\forall$-axiomatizable, but is both $\forall\exists\forall$-axiomatizable and 
    $\exists\forall\exists$-axiomatizable.  In all cases, the quantifier-complexity of the 
    axiomatizations is deduced from semantic criteria; it would be very interesting 
    to give concrete (hopefully even effective) axiomatizations of these classes.
    \item In \cite[Theorem 2.10(ii)]{Dab}, Dabrowski showed that, for a fixed $\lambda\in (0,1)$, the class of \wstar-probability spaces of the form $(M,\vp)$, where $M$ is a type III$_\lambda$ factor and $\vp$ is a \textbf{periodic} state, is axiomatizable in his language for \wstar-probability spaces.  The interested reader can readily adapt Dabrowski's axiomatization to our setting.
    \end{enumerate}
\end{remarks}
\appendix 
\section{Alternative proof of Theorem~\ref{thm:SK-liftingnew} via $H^{\infty}$-lifting theorem}\label{ocneanutb}
We fix a family $(\bm{M},\bm{\varphi})=(M_i,\varphi_i)_{i\in I}$ of $W^*$-probability spaces, and a
non-principal ultrafilter $\mathcal{U}$ on $I$ as in $\S$\ref{subsec:SKdefinabilitynew}. 
In this appendix, we give an alternative proof of the definability of $S_K\,(K>0)$, Theorem~\ref{thm:SK-liftingnew} when $\cU$ is countably incomplete. The proof is completely different from the one in $\S$\ref{subsec:SKdefinabilitynew}. Although this proof is longer, it clarifies the role of analyticity of the modular orbit, and it also provides an $H^{\infty}$-lifting theorem for C$^*$-quotients (Theorem ~\ref{thm:almost-isometric-Cstar-Hinfty-lift}) and for Ocneanu multiplier quotients (Theorem~\ref{thm:ocneanu-Hinfty-lifting}) as byproducts, which we believe to be of independent interest.   
\subsection{Outline of the proof}
Here we explain the idea of an alternative proof. The obstruction to proving Theorem~\ref{thm:SK-liftingnew} by kernel
smoothing alone was explained after the statement of that theorem. 

Fix $K>0$ and $x=(x_i)_{\mathcal U}\in S_K(M,\varphi)$.
In order to relate the estimate of $\|x_i\|_{\mathrm{tb}}$ to that of
$\|x\|_{\mathrm{tb}}$, it is crucial to analyze not only the real time orbit
$\mathbb{R}\ni t\mapsto\sigma^{\varphi}_t(x)$ but its analytic extension
$z\mapsto\sigma^{\varphi}_z(x)$. Indeed, as mentioned in Lemma \ref{lem:totally-bounded-char} above, Connes characterized the totally bounded elements as those $x$ 
that are in the domain of the analytic generators
$\sigma^{\varphi}_{\pm i/2}$, in which case
$$
  \|x\|_{\mathrm{tb}} = \max\bigl\{\|x\|,\|\sigma^{\varphi}_{i/2}(x)\|,
  \|\sigma^{\varphi}_{-i/2}(x)\|\bigr\}.
$$
From this viewpoint, we formulate the $S_K$-definability problem as a holomorphic isometric
lifting problem of the orbit $z\mapsto\sigma^{\varphi}_z(x)$ to an orbit in the Ocneanu
multiplier
$$
  z\mapsto (\sigma^{\varphi_i}_z(x_i))_i \in \mathcal{M}^{\mathcal{U}}
  := \mathcal{M}^{\mathcal{U}}(M_i,\varphi_i).
$$

Below we give the main ideas of the structure of the proof of the definability of the $K$-bounded elements in
a somewhat detailed form.

\medskip

\textbf{Step 1: Metric $H^{\infty}$-lifting for \cstar-quotients via Schur's algorithm.} As a preliminary, we show (Theorem ~\ref{thm:almost-isometric-Cstar-Hinfty-lift}) that given a surjective $*$-homomorphism $q\colon A\to B$ of unital $C^*$-algebras, an element $f$ in the Banach algebra $H^{\infty}(\Omega,B)$ of bounded holomorphic functions from a simply connected proper domain $\Omega\subsetneq \mathbb{C}$ to $B$ and $\varepsilon>0$, there is $F\in H^{\infty}(\Omega,A)$ such that $\|F\|_{\infty}\le \|f\|_{\infty}+\varepsilon$ and 
\[q(F(z))=f(z),\qquad z\in \Omega.\]

We note that the crucial point is to prove the theorem for the open disk $\mathbb{D}=\{z\in \mathbb{C}\mid |z|<1\}$; the simply connected proper domain $\Omega\subsetneq \mathbb{C}$ case follows by Riemann's mapping theorem.
Also, the key ingredient in the proof  of this theorem has been well-known in the operator theory literature as the operator-valued Schur's algorithm, which is used to recover an $H^{\infty}$-function from its Schur parameters, and what we do here is an adaptation of Schur's algorithm to the \cstar-algebraic setting while keeping track of the norm estimate. See $\S$\ref{sec Schurinterpolation} for historical background on the operator-valued Schur's algorithm. For the later application, we will choose $\Omega$ to be a horizontal open strip $$\mathbb{S}_{\rho}=\{z\in \mathbb{C} \ : \  |\operatorname{Im}z|<\rho\}.$$ In this case, an explicit biholomorphic mapping between $\mathbb{D}$ and $\mathbb{S}_{\rho}$ exists (see Remark \ref{em:explicitbiholomorphic}). 

\medskip
\textbf{Step 2: Isometric $H^{\infty}$-lifting for the Ocneanu multiplier $\cM^{\cU}$.} Using the $H^{\infty}$-lifting theorem from Step~1 together with a somewhat
long diagonal argument (in the case of countably incomplete ultrafilter; we do not need this
part for the countably complete case as the situation is much simpler), we then show
(Theorem~\ref{thm:ocneanu-Hinfty-lifting}) that if $(M,\varphi)=\prod_{\mathcal{U}}(M_i,\varphi_i)$
is the Ocneanu ultraproduct, then for each $f\in H^{\infty}(\mathbb{D},M)$, there is
$F\in H^{\infty}(\mathbb{D},\mathcal{M}^{\mathcal{U}})$ such that
$\|F\|_{\infty}=\|f\|_{\infty}$ and $q\circ F = f$, where $q\colon\mathcal{M}^{\mathcal{U}}\to M$
is the canonical quotient map. Strictly speaking, we do not need this exact isometric lifting, but we present the result as it may be of independent interest. 

\medskip
\textbf{Step 3: First Gaussian regularization.} Here, we start the proof of the definability of $S_K$. Given
$x\in S_K(M,\varphi)$, we do a Gaussian smoothing to approximate $x$
by $\sigma^{\varphi}$-analytic element $x_s=\sigma^{\varphi}_{k_s}(x)$. 
This allows us to define the analytic
extension $z\mapsto\sigma^{\varphi}_z(x_s)$ to a slightly larger open strip
$\mathbb{S}_{1/2+s}$ ($s>0$) than the closed strip $\overline{\mathbb{S}_{1/2}}$. Choosing
$s>0$ sufficiently small, we ensure that $\|F_s\|_{\infty}\le e^{s/4}K$ is sufficiently
close to $K$. Now consider the analytic map $F_s:\mathbb{S}_{1/2+s}\to M$ given by
$$
  F_s(z) = \sigma^{\varphi}_z(x_s).
$$
We use Step~2 to find an isometric lift
$$
  G_s=(G_{s,i})_i\in H^{\infty}(\mathbb{S}_{1/2+s},\mathcal{M}^{\mathcal{U}}),
  \quad q\circ G_s = F_s.
$$

\medskip
\textbf{Step 4: The second Gaussian regularization and Gaussian averaging.} Now, $F_s$ satisfies the $\sigma^{\varphi}$-equivariance
$$
  \sigma^{\varphi}_t(F_s(z)) = F_s(z+t),\quad t\in\mathbb{R},\; z\in\mathbb{S}_{1/2+s}.
$$
Unfortunately, there is no reason to believe that the lifted elements $G_i$ are
$\sigma^{\varphi_i}$-equivariant as well. Now consider the Gaussian weight 
$$
  \gamma_a(u) = \Bigl(\frac{a}{\pi}\Bigr)^{1/2}e^{-au^2},\quad u\in\mathbb{R}.
$$
We may average the orbit $u\mapsto\sigma^{\varphi_i}_{-u}(G_i(u+z))$ with integrable
Gaussian weight to obtain a new function
$$
  \tilde{G}_i(z) = \int_{\mathbb{R}}\gamma_a(u+z)\sigma^{\varphi_i}_{-u}(G_i(u+z))\,du,
$$
which is expected to be $\sigma^{\varphi_i}$-equivariant:
\begin{align*}
  \tilde{G}_i(t+z)
  &= \int_{\mathbb{R}}\gamma_a(u+z+t)\sigma^{\varphi_i}_{-u}(G_i(u+z+t))\,du\\
  &= \int_{\mathbb{R}}\gamma_a(v+z)\sigma^{\varphi_i}_{-(v-t)}(G_i(v+z))\,dv\\
  &= \sigma^{\varphi_i}_t(\tilde{G}_i(z)).
\end{align*}
However, the map
$$
  u\mapsto \bigl(\sigma^{\varphi_i}_{-u}(G_i(u+z))\bigr)_i
$$
need not be norm-continuous as a map into the Ocneanu multiplier algebra $\mathcal{M}^{\mathcal{U}}$,
because the modular groups $\sigma^{\varphi_i}$ need not be uniformly norm-continuous in $i$.
This makes it impossible to guarantee that the resulting map be norm-holomorphic in the open
strip. To overcome this issue, we insert (with a sufficiently small parameter $\tau>0$) a
preliminary (the second) Gaussian smoothing $\sigma^{\varphi_i}_{k_\tau}$ to obtain this missing
norm-continuity. After this Gaussian smoothing, one can rewrite
$$
  \sigma^{\varphi_i}_{-u}\!\Bigl(\sigma^{\varphi_i}_{k_\tau}(G_i(u+z))\Bigr)
  = \int_{\mathbb{R}} k_\tau(v+u)\sigma^{\varphi_i}_v(G_i(u+z))\,dv.
$$
Consequently, for $u,u'\in \mathbb R$, 
\begin{align*}
  \Bigl\|\bigl(\sigma^{\varphi_i}_{-u}\!\bigl(\sigma^{\varphi_i}_{k_\tau}(&G_i(u+z))\bigr)\bigr)_i
  - \bigl(\sigma^{\varphi_i}_{-u'}\!\bigl(\sigma^{\varphi_i}_{k_\tau}(G_i(u'+z))\bigr)\bigr)_i
  \Bigr\|_{\mathcal{M}}\\
  &\le \|G\|_{\infty}\|k_\tau(\cdot+u)-k_\tau(\cdot+u')\|_{L^1(\mathbb{R})}
   + \|G(u+z)-G(u'+z)\|_{\mathcal{M}}.
\end{align*}
The right-hand side tends to $0$ as $u'\to u$. 
Thus the smoothed covariance integrand is norm-continuous and uniformly bounded; after multiplication by $\gamma_a(u+z)$, it is Bochner integrable in $\mathcal M^{\mathcal U}$.
The resulting
Gaussian average is then a well-defined element of $\mathcal{M}^{\cU}$, and gives exact
coordinatewise equivariance.

\medskip
\textbf{Step 5:  The final  diagonalization.} In Step~4, because of the preliminary Gaussian smoothing, the resulting
covariant map is an exact lift of $\sigma^{\varphi}_{k_\tau}(F)$; however, it is only an
approximate lift of $F$. The error term can be arbitrarily small by choosing smaller and
smaller $s$ and $\tau$.

By choosing smaller and smaller parameters $s$ and $\tau$, we do once again a diagonal
argument along $\mathcal{U}$ to obtain a desired lifting $(x_i)_i\in\mathcal{M}^{\mathcal{U}}$
of $x$ with $x_i\in S_K(M_i,\varphi_i)$, and the proof is complete.

Therefore, although we do Gaussian smoothing several times during the proof,
none of these smoothing steps is used directly to obtain the desired
coordinatewise total $K$-bound; as was explained in the above, there is a trade-off. We must choose only one of
$x_i\sim\sigma^{\varphi_i}_{k_\tau}(x_i)$ by choosing $\tau\sim 0$, or the control of
$\|\sigma^{\varphi_i}_{k_\tau}(x_i)\|_{\mathrm{tb}}$ by choosing $\tau\to\infty$, but not both. Instead, the smoothing/averaging are used (i) to enlarge the domain of holomorphy (first
Gaussian smoothing); (ii) to ensure the norm-continuity of the averaging process (second
preliminary Gaussian smoothing); and then (iii) to restore the modular equivariance that we
lost when we take an isometric lifting of the modular orbit $z\mapsto F_s(z)=\sigma^{\varphi}_z(x_s)$
(the third Gaussian average). During the diagonalization process, we need to take care of
the modular-theoretic properties of $G_i$ so that we define an element in the Ocneanu
multiplier $\mathcal{M}^{\mathcal{U}}$.

In the subsequent discussions, we will use standard facts about Banach-space-valued holomorphic functions;
see, for example, \cite[Appendix~A]{ABHN}.  We will also freely use the
fact that the usual scalar arguments apply after scalarization by
functionals in the dual space.  In particular, for a  Banach space $X$, the $X$-valued versions of Schwarz's lemma, the identity theorem, and the three-lines theorem follow from their scalar counterparts by applying them to
$x^*\circ f$, $x^*\in X^*$, and then using the Hahn--Banach theorem. The
corresponding Cauchy and Morera type theorems, as well as holomorphy under
Bochner integrations, can be obtained in the same way.
\subsection{Operator-valued Schur parametrization and metric $H^{\infty}$ lifting for $C^*$-quotients}\label{sec Schurinterpolation}
Let $\mathbb D=\{z\in\mathbb C: |z|<1\}$ be the open unit disk in $\mathbb{C}$.  For a unital \cstar-algebra $A$, we denote its open unit ball by $\operatorname{Ball}^{\circ}(A)$. More generally, for each nonempty domain (i.e., an open connected subset) $\Omega\subset \mathbb{C}$, we write $H^\infty(\Omega,A)$ for the Banach algebra of bounded norm-holomorphic functions $\Omega\to A$, with norm
\[
\norm{F}_\infty=\sup_{z\in\Omega}\norm{F(z)}.
\]
We also use the notation 
\[
\mathscr S(A)
=
\left\{
f\in H^\infty(\mathbb D,A):\|f\|_\infty\leq1
\right\},\qquad \mathscr{S}_0(A)=\{f\in \mathscr{S}(A):\|f\|_{\infty}<1\}.
\]
$\mathscr{S}(A)$ (resp. $\mathscr{S}_0(
A)$) is the space of \textbf{$A$-valued Schur class} (resp. \textbf{strict Schur class}) functions. 

As a first step towards the definability result, we prove a holomorphic almost-isometric lifting result for \cstar-quotients. 
\begin{thm}
\label{thm:almost-isometric-Cstar-Hinfty-lift}
Let $A$ and $B$ be unital \cstar-algebras and let $q:A\to B$ be a surjective unital $*$-homomorphism. Let $\Omega\subsetneq \mathbb{C}$ be a simply connected proper domain. Then for every $f\in H^\infty(\Omega,B)$ and every $\varepsilon>0$, there exists $F\in H^\infty(\Omega,A)$ 
such that $q\circ F=f$ and $\|F\|_\infty\leq \|f\|_\infty+\epsilon$.
Consequently, the induced map $q_*:
H^\infty(\Omega,A)
\longrightarrow
H^\infty(\Omega,B)$ is a metric quotient map.
\end{thm}
The proof uses a $C^*$-algebraic version of Schur's algorithm \cite{Schur1917,Schur1918}, which we now establish. The fractional-linear transformations used in the construction are closely related to several classical works in operator theory, including Potapov's multiplicative factorization theory
for $J$-contractive matrix functions \cite{Potapov1955R}, Harris's M\"obius transformations in Banach $^*$-algebras \cite{Harris1}, and Shmul'yan's fractional-linear transformations of operator balls
\cite{Shmul80fractionallinear}.  For systematic treatments of
operator-valued Schur algorithms and Schur parametrizations, see
Constantinescu \cite{Const},
Bakonyi--Constantinescu \cite{BC}, and
Arlinski{\u{\i}} \cite{Arlinski}.
We give the argument here in a self-contained form in order to keep track of the strict norm estimates needed below.

Recall that for $\gamma\in \operatorname{Ball}^{\circ}(A)$, the \textbf{defect operators} are defined by 
\[
D_\gamma=(1-\gamma^*\gamma)^{1/2},
\qquad
D_{\gamma^*}=(1-\gamma\gamma^*)^{1/2}.
\]
The operators $D_\gamma$ and $D_{\gamma^*}$ are invertible. We then define maps $\Psi_{\gamma}$, $\Theta_{\gamma}:\operatorname{Ball}(A)\to A$ by  
\begin{align}
\Psi_\gamma(w)
&=
D_{\gamma^*}^{-1}(w-\gamma)
(1-\gamma^*w)^{-1}D_\gamma,
\label{eq Psiform}\\
\Theta_\gamma(v)
&=
\gamma+
D_{\gamma^*}v(1+\gamma^*v)^{-1}D_\gamma
.\label{eq Thetaform}
\end{align}
We start by establishing some well-known facts about operator-valued fractional linear transformations. We include a self-contained proof in
order to keep track of the strict norm estimate needed below. 

\begin{lem}\label{lem PsiTheta}
Let $A$ be a unital $C^*$-algebra and fix
$\gamma\in\operatorname{Ball}^{\circ}(A)$. 
Then:
\begin{enumerate}[label=\textup{(\arabic*)}]
    \item The following defect identities hold. 
    \begin{align}
1-\Psi_\gamma(w)^*\Psi_\gamma(w)
&=
D_\gamma(1-w^*\gamma)^{-1}
(1-w^*w)
(1-\gamma^*w)^{-1}D_\gamma.\label{eq:defect-Psi}\\
1-\Theta_\gamma(v)^*\Theta_\gamma(v)
&=
D_\gamma(1+v^*\gamma)^{-1}
(1-v^*v)
(1+\gamma^*v)^{-1}D_\gamma.\label{eq:defectTheta}
\end{align} 
In particular, both $\Psi_{\gamma}$ and $\Theta_{\gamma}$ map $\operatorname{Ball}(A)$ to itself. 
Moreover, for each $0<r<1$, there exists a constant $0<c=c_{a,r}<1$ depending only on $r$ and $a=\|\gamma\|<1$ such that, if $\|w\|\le r$ (resp. $\|v\|\le r$), then 
$\|\Psi_{\gamma}(w)\|\le c$ (resp. $\|\Theta_{\gamma}(v)\|\le c$) holds. 
\item The maps $\Psi_{\gamma}$ and $\Theta_{\gamma}$ are mutually inverse:
\[
\Psi_\gamma\circ\Theta_\gamma
=
\operatorname{id}_{\operatorname{Ball}(A)},
\qquad
\Theta_\gamma\circ\Psi_\gamma
=
\operatorname{id}_{\operatorname{Ball}(A)}.
\]
\item
For $v,w\in\operatorname{Ball}(A)$,
\[
\Theta_\gamma(v)-\Theta_\gamma(w)
=
D_{\gamma^*}
(1+v\gamma^*)^{-1}
(v-w)
(1+\gamma^*w)^{-1}
D_\gamma.
\]
\end{enumerate}
\end{lem}
\begin{proof}
(1) First, we show (\ref{eq:defect-Psi}), which is a non-commutative analogue of
\[
1-\left|\frac{w-\gamma}{1-\overline{\gamma}w}\right|^2
=
\frac{(1-|\gamma|^2)(1-|w|^2)}{|1-\overline{\gamma}w|^2}.
\]
Set $P=D_{\gamma}^{-2}$ and 
$Q=D_{\gamma^*}^{-2}.$  For $w\in\operatorname{Ball}(A)$, we have
\[
\Psi_\gamma(w)^*\Psi_\gamma(w)
=
D_{\gamma}(1-w^*\gamma)^{-1}
(w^*-\gamma^*)Q(w-\gamma)
(1-\gamma^*w)^{-1}D_{\gamma}.
\]
On the other hand,
\[
1
=
D_{\gamma}(1-w^*\gamma)^{-1}
\bigl[(1-w^*\gamma)P(1-\gamma^*w)\bigr]
(1-\gamma^*w)^{-1}D_{\gamma}.
\]
Therefore
\begin{equation}
\begin{aligned}
1-\Psi_\gamma(w)^*\Psi_\gamma(w)
&=
D_{\gamma}(1-w^*\gamma)^{-1}
\Bigl[
(1-w^*\gamma)P(1-\gamma^*w) \\
&\qquad\qquad
-(w^*-\gamma^*)Q(w-\gamma)
\Bigr]
(1-\gamma^*w)^{-1}D_{\gamma}.
\end{aligned}\label{eq Psigamma expansion}
\end{equation}
Expanding the middle bracket in the right hand side of (\ref{eq Psigamma expansion}) gives
\[
\begin{aligned}
& (1-w^*\gamma)P(1-\gamma^*w)
-(w^*-\gamma^*)Q(w-\gamma) \\
&=
P-P\gamma^*w-w^*\gamma P+w^*\gamma P\gamma^*w
-w^*Qw+w^*Q\gamma+\gamma^*Qw-\gamma^*Q\gamma .
\end{aligned}
\]
We then use the elementary identities
\[
P\gamma^*=\gamma^*Q,
\qquad
\gamma P=Q\gamma,\qquad 
P=1+\gamma^*Q\gamma,
\qquad
Q=1+\gamma P\gamma^*.
\]
The first two identities imply cancellation of the mixed terms:
\[
-P\gamma^*w+\gamma^*Qw=0,
\qquad
-w^*\gamma P+w^*Q\gamma=0.
\]
Thus, the middle bracket equals
\[
P-\gamma^*Q\gamma+w^*(\gamma P\gamma^*-Q)w.
\]
Using $P-\gamma^*Q\gamma=1$ and $\gamma P\gamma^*-Q=-1,$ we obtain
\[
(1-w^*\gamma)P(1-\gamma^*w)
-(w^*-\gamma^*)Q(w-\gamma)
=
1-w^*w.
\]
Consequently, 
\[
1-\Psi_\gamma(w)^*\Psi_\gamma(w)
=
D_{\gamma}(1-w^*\gamma)^{-1}
(1-w^*w)
(1-\gamma^*w)^{-1}D_{\gamma}.
\]
Set $S_w=(1-\gamma^*w)^{-1}D_{\gamma}$. Now assume that $\|w\|\le r$ and set $a:=\|\gamma\|<1$. Since $w^*w\le r^2$, we obtain 
\[
1-\Psi_\gamma(w)^*\Psi_\gamma(w)
=
S_w^*(1-w^*w)S_w\ge (1-r^2)S_w^*S_w.
\]
Moreover, $
S_w^{-1}=D_\gamma^{-1}(1-\gamma^*w),$
and hence $\|S_w^{-1}\|
\leq
(1-a^2)^{-1/2}(1+ar).$
Therefore
\[
S_w^*S_w
\geq
\frac{1-a^2}{(1+ar)^2}1.
\]
Consequently, 
\[
1-\Psi_\gamma(w)^*\Psi_\gamma(w)
\geq
\rho_{a,r}\,1,
\qquad
\rho_{a,r}:=
(1-r^2)\frac{1-a^2}{(1+ar)^2}>0.
\]
Thus $\Psi_\gamma(w)^*\Psi_\gamma(w)\leq (1-\rho_{a,r})1,$
and \[
\|\Psi_\gamma(w)\|\leq c_{a,r},
\qquad
c_{a,r}:=\sqrt{1-\rho_{a,r}}<1.
\]
This shows the claims for $\Psi_{\gamma}$. 
The defect identity (\ref{eq:defectTheta}) for $\Theta_\gamma$ can be shown similarly. Here, we give an alternative argument. We first note that $\Theta_\gamma(v)=\Psi_{-\gamma}(v)$.  
Indeed, using $D_{\gamma^*}\gamma=\gamma D_\gamma$,
we obtain
\[
\begin{aligned}
D_{\gamma^*}\Theta_\gamma(v)
&=
D_{\gamma^*}\gamma
+
D_{\gamma^*}^2v(1+\gamma^*v)^{-1}D_\gamma\\
&=
\gamma D_\gamma
+
(1-\gamma\gamma^*)v(1+\gamma^*v)^{-1}D_\gamma\\
&=
(v+\gamma)(1+\gamma^*v)^{-1}D_\gamma.
\end{aligned}
\]
Thus
\[
\Theta_\gamma(v)
=
D_{\gamma^*}^{-1}(v+\gamma)
(1+\gamma^*v)^{-1}D_\gamma
=
\Psi_{-\gamma}(v).
\]

Applying \eqref{eq:defect-Psi} with $-\gamma$ in place of $\gamma$
therefore gives
\[
1-\Theta_\gamma(v)^*\Theta_\gamma(v)
=
D_\gamma(1+v^*\gamma)^{-1}
(1-v^*v)
(1+\gamma^*v)^{-1}D_\gamma,
\]
which is \eqref{eq:defectTheta}.  The corresponding strict norm
estimate $\|v\|\le r\implies \|\Theta_{\gamma}(v)\|\le c_{a,r}$ follows from the estimate already proved for $\Psi_\gamma$, again with $-\gamma$ in place of $\gamma$, since
$\|-\gamma\|=\|\gamma\|$.
We also observe that the defect identities themselves imply that both $\Psi_{\gamma}$ and $\Theta_{\gamma}$ map $\operatorname{Ball}(A)$ to itself.\\
(2) Suppose that $w=\Theta_\gamma(v),\,\,v\in \operatorname{Ball}(A).$
Since $\gamma^*D_{\gamma^*}=D_\gamma\gamma^*,$ we have
\[
\begin{aligned}
1-\gamma^*w
&=
1-\gamma^*\gamma
-\gamma^*D_{\gamma^*}v(1+\gamma^*v)^{-1}D_\gamma  \\
&=
D_\gamma^2
-
D_\gamma\gamma^*v(1+\gamma^*v)^{-1}D_\gamma  \\
&=
D_\gamma
\left[
1-\gamma^*v(1+\gamma^*v)^{-1}
\right]
D_\gamma .
\end{aligned}
\]
Putting $X=\gamma^*v$, we have $1-X(1+X)^{-1}=(1+X)^{-1}.$ 
Therefore $
1-\gamma^*w
=
D_\gamma(1+\gamma^*v)^{-1}D_\gamma.$

Consequently,
\[
\begin{aligned}
\Psi_\gamma(w)
&=
D_{\gamma^*}^{-1}
D_{\gamma^*}v(1+\gamma^*v)^{-1}D_\gamma
\bigl(D_\gamma(1+\gamma^*v)^{-1}D_\gamma\bigr)^{-1}
D_\gamma\\
&=v.
\end{aligned}
\]
The identity $\Theta_\gamma(\Psi_\gamma(w))=w$ follows similarly, using $$1+\gamma^*\Psi_\gamma(w)
=
D_\gamma(1-\gamma^*w)^{-1}D_\gamma.$$\\
(3) Using $v(1+\gamma^*v)^{-1}
=
(1+v\gamma^*)^{-1}v,$ we obtain
\[
\begin{aligned}
&v(1+\gamma^*v)^{-1}
-w(1+\gamma^*w)^{-1}
=
(1+v\gamma^*)^{-1}
(v-w)
(1+\gamma^*w)^{-1},
\end{aligned}
\]
which proves (3).
\end{proof}

\begin{lem}\label{lem:schur-first step}
Let $A$ be a unital $C^*$-algebra and take $f\in \mathscr{S}_0(
A)$.  Put $\gamma=f(0)$.  Then the \textbf{first step Schur function} $f_1\colon \mathbb D\setminus \{0\}\to A$ of $f$ given by  
\[
 f_1(z)
 =
 z^{-1}D_{\gamma^*}^{-1}(f(z)-\gamma)(1-\gamma^*f(z))^{-1}D_\gamma
\]
has a removable singularity at $0$ and thus extends to a function $f_1\in H^\infty(\mathbb{D},A)$, which further satisfies $\|{f_1}\|_\infty<1$, that is $f_1\in \mathscr{S}_0(
A)$. 
Conversely, if $g\in \mathscr{S}_0(
A)$ with $\|{g}\|_\infty<1$ and $\gamma\in \operatorname{Ball}^{\circ}(A)$, then the map $\Phi_{\gamma}(g):\mathbb{D}\to A$ given by 
\[
\Phi_\gamma(g)(z)
=
\gamma+D_{\gamma^*}\,z g(z)(1+\gamma^*z g(z))^{-1}D_\gamma
\]
belongs to $\mathscr{S}_0(
A)$.  Moreover, $f$ can be recovered from $f_1$ by the formula 
\[\Phi_{\gamma}(f_1)=f.\]
\end{lem}

\begin{proof}[Proof of Lemma \ref{lem:schur-first step}]
If $f=0$, choose any $0<r<1$. If $f\neq 0$, set $r:=\|f\|_{\infty}<1$. Then consider the map $g:\mathbb{D}\to A$ given by $g(z)=\Psi_\gamma(f(z)).$ 
Then $g\in H^\infty(\mathbb D,A)$; by Lemma \ref{lem PsiTheta} (1), there exists a constant $0<c=c_{a,r}<1$ depending only on $0<r<1$ and $a=\|\gamma\|<1$ such that 
$\|g\|_\infty\leq c<1.$ Note that since $g$ is holomorphic and $g(0)=0$, its Taylor series at $z=0$ has no constant term, and thus it is of the form $g(z)=zh(z)$ for some holomorphic function $h\colon \mathbb{D}\to A$. In particular, the function $f_1:\mathbb{D}\setminus\{0\}\to A$ given by $f_1(z):=z^{-1}g(z)(=h(z))$ has a removable singularity at $0$. By the Banach-space-valued Schwarz lemma, we have $\|g(z)\|\leq c|z|$ for all $z\in\mathbb D.$  (Indeed, this follows by applying the usual scalar-valued Schwarz lemma to $\mathbb D\ni 
z\mapsto c^{-1}\varphi(g(z))\in \mathbb{C}$
for each $\varphi \in A^*$ with $\|\varphi\|=1$, and then taking the supremum over such $\varphi$.)  Therefore, 
\[
\|f_1(z)\|
=
\left\|z^{-1}g(z)\right\|
\leq c
\]
for $z\neq0$ 
 and the same estimate holds at $0$ by continuity.  Hence $f_1\in \mathscr{S}_0(
A)$.
Conversely, take $g\in \mathscr{S}_0(
A)$
and $\gamma\in \operatorname{Ball}^{\circ}(A)$.  Set $b:=\|\gamma\|<1$.  
Consider the map $\Phi_\gamma(g):\mathbb{D}\to A$ defined above, that is, $\Phi_\gamma (g)(z)=\Theta_\gamma(zg(z))$. 
Set $\rho:=\|g\|_\infty<1$. If $g=0$, then $\Phi_{\gamma}(g)(z)=\gamma$ and $\Phi_{\gamma}(g)\in \mathscr{S}_0(A)$ holds. Assume $g\neq 0$, so that $0<\rho<1$.  
Then by Lemma \ref{lem PsiTheta} (1), there exists a constant $0<c'=c_{\rho,b}<1$ such that for every $v\in A$ with $\|v\|\leq \rho$, $\|\Theta_{\gamma}(v)\|\le c'$ holds. Therefore,  we have $\Phi_{\gamma}(g)\in \mathscr{S}_0(
A)$ since $\|\Phi_{\gamma}(g)\|_{\infty}\le c'<1$. 

Fix now $z\in \mathbb{D}\setminus \{0\}$. Since $
zf_1(z)=\Psi_\gamma(f(z)),
$
we have, by Lemma \ref{lem PsiTheta} (2), that 
\[
\Phi_{\gamma} (f_1)(z)
=
\Theta_\gamma(zf_1(z))
=
\Theta_\gamma(\Psi_\gamma(f(z)))
=
f(z).
\]
By continuity, the same equality holds at $z=0$. 
\end{proof}
\begin{prop}
\label{prop:schur-mobius-order}
Let $A$ be a unital $C^*$-algebra and let
$\gamma\in\operatorname{Ball}^{\circ}(A)$.  
Then the following statements hold.
\begin{enumerate}[label=\textup{(\arabic*)}]
\item
For $g\in \mathscr{S}(A)$, define $\Phi_\gamma(g):\mathbb{D}\to A$ by
\[
\Phi_\gamma(g)(z)
=
\Theta_\gamma(zg(z)).
\]
Then $
\Phi_\gamma(g)\in \mathscr{S}(A).$
Moreover, $\|\Phi_\gamma(g)\|_\infty<1$ if $\|g\|_{\infty}<1$. 
If $g,h\in \mathscr{S}(A)$ satisfy $g(z)-h(z)=O(z^m)$ at $z=0$, then
\[
\Phi_\gamma(g)(z)-\Phi_\gamma(h)(z)
=
O(z^{m+1}).
\]

\item
Let $q:A\to B$ be a surjective unital $*$-homomorphism.  If
$\widetilde\gamma\in\operatorname{Ball}^{\circ}(A)$,
$\gamma=q(\widetilde\gamma)$, and
$\widetilde g\in \mathscr{S}(A)$, then
\[
q\circ \Phi_{\widetilde\gamma}(\widetilde g)
=
\Phi_\gamma(q\circ\widetilde g).
\]
\end{enumerate}
\end{prop}

\begin{proof}
The second statement of (1) was already established in
Lemma~\ref{lem:schur-first step}, and the first statement is proven analogously. 
To prove the last statement in (1), note that, if $g-h=O(z^m)$, then $zg(z)-zh(z)=O(z^{m+1}),$
and from the formula (\ref{eq Thetaform}), it follows that
\[
\Phi_\gamma(g)-\Phi_\gamma(h)=O(z^{m+1}).
\]
(2) The functional calculus and inversion commute with $q$:
\[
q(D_{\widetilde\gamma})=D_\gamma,
\qquad
q(D_{\widetilde\gamma^*})=D_{\gamma^*}.
\]
Applying $q$ to the defining formula for
$\Phi_{\widetilde\gamma}$ proves the claim.
\end{proof}

\begin{defn}
Let $A$ be a unital $C^*$-algebra and let $f\in\mathscr S(A)$.
A sequence $(\gamma_n)_{n\geq0}\subseteq \operatorname{Ball}^{\circ}(A)$ 
is called a \textbf{Schur parameter sequence} for $f$ if there exist functions $f_n\in\mathscr S(A)$, $n\geq0$, such that
\[
f_0=f,\qquad f_n(0)=\gamma_n,\qquad \text{and}\qquad f_n=\Phi_{\gamma_n}(f_{n+1})
.
\]
If moreover $f_n\in\mathscr S_0(A)$ for every $n\geq 0$, 
then $(\gamma_n)_{n\geq0}$ is called a
\textbf{strict Schur parameter sequence} for $f$.
\end{defn}
Below, we refer to $\gamma_n$ simply as the (strict) \textbf{Schur parameters} of $f$. If $f_n\in\mathscr S_0(A)$ and
$\gamma_n=f_n(0)$, then Lemma~\ref{lem:schur-first step} shows that the usual Schur transform $f_{n+1}$ belongs to
$\mathscr S_0(A)$, and Proposition~\ref{prop:schur-mobius-order}
gives $f_n=\Phi_{\gamma_n}(f_{n+1})$. 
Thus, the usual Schur algorithm for a strict Schur-class function
produces a strict Schur parameter sequence in the sense above.

The next result is a $C^*$-algebraic Schur parametrization.
\begin{prop}
\label{prop:full-Cstar-Schur-parametrization}
Let \(A\) be a unital \(C^*\)-algebra and take
\(f\in\mathscr S_0(A)\).
Recursively define
\[
f_0=f,
\qquad
\gamma_n=f_n(0),
\]
and
\[
f_{n+1}(z)
=
z^{-1}
D_{\gamma_n^*}^{-1}
\bigl(f_n(z)-\gamma_n\bigr)
\bigl(1-\gamma_n^*f_n(z)\bigr)^{-1}
D_{\gamma_n},
\qquad z\neq0,
\]
with the removable singularity at \(z=0\).
Then the following statements hold:
\begin{enumerate}[label=\textup{(\arabic*)}]
\item
For every $n\geq 0$, $f_n\in\mathscr S_0(A)$, 
and hence $\|\gamma_n\|<1$.  In particular,
$(\gamma_n)_{n\geq0}$ is a strict Schur parameter sequence for $f$. 
\item
$f_n=\Phi_{\gamma_n}(f_{n+1}) 
$ for every $n\geq0$.

\item
For $N\geq0$, define the \textbf{$N$-th Schur truncation} recursively by
\[
f_{N+1}^{[N]}=0,
\qquad
f_n^{[N]}=\Phi_{\gamma_n}(f_{n+1}^{[N]})
\quad(n=N,\dots,0),
\]
and set $f^{[N]}:=f_0^{[N]}.$ 
Then $\|f^{[N]}\|_\infty\leq1$ 
and both $f^{[N]}$ and $f$ have the same Taylor coefficients up to 
degree $N$.  In particular, for every $0<r<1$,
\[
\sup_{|z|\leq r}\|f(z)-f^{[N]}(z)\|
\leq
\frac{2r^{N+1}}{1-r}.
\]
Consequently, $f^{[N]}\longrightarrow f$ uniformly on compact subsets of $\mathbb D$. 
\item

Conversely, let $(\gamma_n)_{n\geq0}
\subseteq\operatorname{Ball}^{\circ}(A)$ 
be an arbitrary sequence.  Then the canonical Schur truncations
defined above converge uniformly on compact subsets of $\mathbb D$ to a unique function $f\in\mathscr S(A)$ for which
$(\gamma_n)_{n\geq0}$ is a Schur parameter sequence.
\end{enumerate}
\end{prop}
\begin{proof}
(1) is immediate from Lemma \ref{lem:schur-first step}. (2) follows from the last part of Lemma \ref{lem:schur-first step}.\\
(3) 
Fix $N$.  Iterating (2), we have
\[
f
=
\Phi_{\gamma_0}
\bigl(
\Phi_{\gamma_1}
(
\cdots
\Phi_{\gamma_N}(f_{N+1})
\cdots
)
\bigr),
\]
whereas
\[
f^{[N]}
=
\Phi_{\gamma_0}
\bigl(
\Phi_{\gamma_1}
(
\cdots
\Phi_{\gamma_N}(0)
\cdots
)
\bigr).
\]
At level $N+1$, the two tails $f_{N+1}$ and $0$ differ by $O(1)$.
By Proposition~\ref{prop:schur-mobius-order}, applying
$\Phi_{\gamma_N}$ raises the order of vanishing by one, so
\[
\Phi_{\gamma_N}(f_{N+1})-\Phi_{\gamma_N}(0)=O(z).
\]
Applying the same argument successively at the levels
$N-1,N-2,\ldots,0$, we obtain
\[
f-f^{[N]}=O(z^{N+1}).
\]
Thus the Taylor coefficients of $f$ and $f^{[N]}$ agree up to degree $N$. Write the Taylor series expansions 
\[
f(z)=\sum_{k=0}^\infty a_kz^k,
\qquad
f^{[N]}(z)=\sum_{k=0}^\infty a_k^{[N]}z^k.
\]
Since both functions belong to $\mathscr S(A)$, Cauchy's estimates give
\[
\|a_k\|\leq1,
\qquad
\|a_k^{[N]}\|\leq1.
\]
Since $a_k=a_k^{[N]}$ for $k\leq N$, if $|z|\leq r<1$, we have
\[
\begin{aligned}
\|f(z)-f^{[N]}(z)\|
&\leq
\sum_{k=N+1}^\infty
\bigl(\|a_k\|+\|a_k^{[N]}\|\bigr)r^k\\
&\leq
2\sum_{k=N+1}^\infty r^k
=
\frac{2r^{N+1}}{1-r}.
\end{aligned}
\]
This proves (3).

Now let $(\gamma_n)_{n\ge 0}$ be an arbitrary sequence in $\operatorname{Ball}^{\circ}(A)$.
For $N\geq0$, let $f^{[N]}$ be the $N$-th canonical Schur truncation with parameters $(\gamma_n)_{n\ge 0}$. 
Let
$M>N$.  For $0\leq k\leq M+1$, write $f_k^{[M]}$ for the intermediate
Schur truncations determined by
\[
f_{M+1}^{[M]}=0,
\qquad
f_k^{[M]}=\Phi_{\gamma_k}(f_{k+1}^{[M]}),
\quad k=M,\ldots,0.
\]
Similarly, $f_k^{[N]}$ is defined by
\[
f_{N+1}^{[N]}=0,
\qquad
f_k^{[N]}=\Phi_{\gamma_k}(f_{k+1}^{[N]}),
\quad k=N,\ldots,0.
\]

At level $N+1$, the two tails $f_{N+1}^{[M]}$ and $f_{N+1}^{[N]}=0$
are both bounded Schur-class functions, and hence their difference is $O(1)$.
Applying Proposition \ref{prop:schur-mobius-order} (1) to
\[
f_N^{[M]}=\Phi_{\gamma_N}(f_{N+1}^{[M]}),
\qquad
f_N^{[N]}=\Phi_{\gamma_N}(f_{N+1}^{[N]}),
\]
we obtain
\[
f_N^{[M]}(z)-f_N^{[N]}(z)=O(z).
\]
Applying the same argument successively at the levels
$N-1,N-2,\ldots,0$, we get
\[
f_k^{[M]}(z)-f_k^{[N]}(z)=O(z^{N-k+1}),
\qquad k=0,1,\ldots,N.
\]
In particular,
\[
f^{[M]}(z)-f^{[N]}(z)
=
f_0^{[M]}(z)-f_0^{[N]}(z)
=
O(z^{N+1}).
\]

Then the Cauchy estimate gives
\[
\sup_{|z|\leq r}
\|f^{[M]}(z)-f^{[N]}(z)\|
\leq
\frac{2r^{N+1}}{1-r}.
\]
Thus $(f^{[N]})_{N\geq 0}$ converges uniformly on compact subsets to some $f\in\mathscr S(A).$
On the other hand, for $N\geq n$, let $f_n^{[N]}$ be the Schur truncation
constructed from $
\gamma_n,\gamma_{n+1},\dots,\gamma_N.
$ 
The same argument gives locally uniform limits 
$
f_n^{[N]}\longrightarrow f_n\in\mathscr S(A).
$ as $N\to \infty$. 
Note that on every disk
$|z|\leq r<1$, we have
\[
\left\|
(1+\gamma_n^*zf_{n+1}^{[N]}(z))^{-1}
\right\|
\leq
\frac{1}{1-r\|\gamma_n\|}.
\] 
Thus, we may take the locally uniform limit $N\to \infty$ in the identity proved in (3), namely
\[
f_n^{[N]}
=
\Phi_{\gamma_n}(f_{n+1}^{[N]}),
\]
to obtain $f_n=\Phi_{\gamma_n}(f_{n+1}).$
Therefore, $f_n(0)=\gamma_n,
$ 
and the Schur transform of $f_n$ is $f_{n+1}$.  Hence the Schur parameters
of $f=f_0$ are exactly $(\gamma_n)_{n\ge 0}$.

Finally, if $g\in\mathscr S(A)$ has the same Schur parameters, then for
every $N$, $g$ and $f^{[N]}$ have identical Taylor coefficients
up to degree $N$.  Thus, all Taylor coefficients of $g$ and $f$
coincide, and $g=f$ holds. This finishes the proof. 
\end{proof}
We are now ready to prove the main result of this subsection:

\begin{proof}[Proof of Theorem \ref{thm:almost-isometric-Cstar-Hinfty-lift}]
First, we consider the case $\Omega=\mathbb{D}$. 
Put $R:=\|f\|_\infty+\varepsilon$ and $g:=R^{-1}f.$ Then $\|g\|_\infty<1.$ 
Let $(\gamma_n)_{n\geq0}
\subseteq\operatorname{Ball}^{\circ}(B)
$ be the Schur parameters of $g$, as in
Proposition~\ref{prop:full-Cstar-Schur-parametrization}. For each $n$, choose $\widetilde\gamma_n\in \operatorname{Ball}^{\circ}(A)$ 
such that $q(\widetilde\gamma_n)=\gamma_n$. 

Let $\widetilde g^{[N]}$ be the canonical Schur truncation associated
with parameters $\widetilde\gamma_0,\dots,\widetilde\gamma_N.
$ 
By Proposition~\ref{prop:full-Cstar-Schur-parametrization}, there exists $\widetilde{g}\in \mathscr{S}(
A)$
such that $\widetilde g^{[N]}
\longrightarrow
\widetilde g$ uniformly on compact subsets of $\mathbb D$. 

Let $g^{[N]}$ be the canonical Schur truncation associated with 
$\gamma_0,\dots,\gamma_N$. By Proposition~\ref{prop:schur-mobius-order}, we have 
$
q\circ \widetilde g^{[N]}
=
g^{[N]}$.
Passing to the locally uniform limit gives $
q\circ \widetilde g=g.$

Define $F:=R\widetilde g.$ 
Then $q\circ F=Rg=f
$ 
and
\[
\|F\|_\infty
\leq R
=
\|f\|_\infty+\varepsilon.
\]
Now suppose $\Omega\subsetneq \mathbb{C}$ is a simply connected proper domain. 
By Riemann's mapping theorem (see \cite[Section 6.1, Theorem 1]{Ahlfors} for example), there exists a biholomorphic isomorphism $\theta:\mathbb D\to\Omega$. Given $f\in H^{\infty}(\Omega,B)$ and $\varepsilon>0$, we apply the above result to $f\circ \theta\in H^{\infty}(\mathbb D,B)$ to find $\tilde{F}\in H^{\infty}(\mathbb{D},A)$ such that $q\circ \tilde{F}=f\circ \theta$ and $\|\tilde{F}\|_{\infty}\le \|f\circ \theta\|_{\infty}+\varepsilon=\|f\|_{\infty}+\varepsilon$. Then $F=\tilde{F}\circ \theta^{-1}\in H^{\infty}(\Omega,A)$ is as desired. 
\end{proof}

\begin{remark}\label{em:explicitbiholomorphic}
In the second part of the proof of Theorem \ref{thm:almost-isometric-Cstar-Hinfty-lift}, for a general simply connected proper domain $\Omega$, we needed Riemann's mapping theorem to find a biholomorphic map $\Phi\colon \Omega\to \mathbb{D}$.  However, in our actual application, we only need to find a biholomorphic mapping between $\mathbb{S}_{\rho}$ and  $\mathbb D$, in which case an explicit biholomorphic mapping can be found.  Indeed, for $\rho>0$, the map 
\[
  \Phi_{\rho}(z)
  =
  \tanh\left(\frac{\pi z}{4\rho}\right)
  =
  \frac{\exp\left(\frac{\pi z}{2\rho}\right)-1}
       {\exp\left(\frac{\pi z}{2\rho}\right)+1}
\]
is a biholomorphic isomorphism   $\Phi_{\rho}:\mathbb{S}_{\rho}\to\mathbb{D}.
$
Its inverse is
\[
  \Phi_{\rho}^{-1}(\zeta)
  =
  \frac{2\rho}{\pi}
  \log\left(\frac{1+\zeta}{1-\zeta}\right),
\]
where the logarithm is taken with the principal branch on the right half-plane.
\end{remark}

\subsection{Isometric $H^{\infty}$-lifting for the Ocneanu ultraproduct and its approximate equivariant version}\label{sec metric Hinfty lifting}
Recall that for a sequence $(\bm{M},\bm{\varphi})=(M_i,\varphi_i)_{i\in I}$ of W$^*$-probability spaces and $\cU$ a non-principal ultrafilter on $I$, we define 
\[(M,\varphi)=\prod_{\cU}(M_i,\varphi_i)=\cM^{\cU}/\cI_{\cU}\]
and $q\colon \cM^{\cU}\to M$ is the quotient map. We will keep this hypothesis throughout $\S$\ref{sec metric Hinfty lifting} and $\S$\ref{sec definability of SK}. 

\subsubsection{Isometric $H^{\infty}$-lifting for the Ocneanu multiplier}

In this section, we show how the theorem of the previous subsection can be improved in the case of the quotient map involved in the definition of the Ocneanu ultraproduct to obtain a lift with the same norm.  Although not entirely necessary for the main result of this section, we believe it is of independent interest.  The reader only interested in the main result of this section can safely skip the next theorem and see how the result of the previous subsection suffices.

\begin{thm}
\label{thm:ocneanu-Hinfty-lifting}
Assume that $\cU$ is a countably incomplete ultrafilter on a set $I$. 
Let $\Omega\subsetneq\mathbb C$ be a simply connected
domain. Then for every $f\in H^\infty(\Omega,M)$,  
there exists $F\in H^\infty(\Omega,\mathcal M^{\cU})$ such that  $q\circ F=f$ and $\|F\|_\infty=\|f\|_\infty$.
\end{thm}
\begin{proof}
We may assume $f\neq 0$. Set $K:=\|f\|_{\infty}>0$. 
Choose a countable dense set $
\{\zeta_1,\zeta_2,\ldots\}\subseteq\Omega.$  
By Theorem \ref{thm:almost-isometric-Cstar-Hinfty-lift}  
applied to the quotient map $q\colon \cM^{\cU}\to M$, 
for each $n\in \mathbb{N}$, there exists $
U^{(n)}\in H^\infty(\Omega,\cM^{\cU})$ 
 such that 
\[
q\circ U^{(n)}=f,\qquad \|U^{(n)}\|_{\infty}\le K+\frac{1}{n}.\]

Set $\lambda_n:=\frac{K}{K+1/n}$ and $H^{(n)}:=\lambda_n\,U^{(n)}.$  Then
$\|H^{(n)}\|_\infty\le K$, 
and $q\circ H^{(n)}=\lambda_nf.$
For each $j\geq1$, choose an element
$b^{(j)}=(b_i^{(j)})_i\in\cM^{\cU}$
such that $q(b^{(j)})=f(\zeta_j).$\\
For $n\geq j$, define
\[
e^{(n,j)}
:=
H^{(n)}(\zeta_j)-\lambda_n b^{(j)}.
\]
Also, define $e^{(n,j)}=0$ for $n<j$. 
Then $q(e^{(n,j)})=0,$ 
and hence $e^{(n,j)}\in\cI_{\cU}.$

We now do a diagonalization argument along $\cU$.  Since $e^{(n,j)}\in\cI_{\cU}$, for $n\geq j$ and
$\ell\geq1$, the set
\[
A_{n,j,\ell}
=
\left\{
i\in I \ \middle| \ 
\|e_i^{(n,j)}\|_{\varphi_i}^{\#}<\frac1\ell
\right\}
\]
belongs to $\mathcal U$. Since $\mathcal U$ is countably incomplete, we may choose a decreasing sequence $C_1\supseteq C_2\supseteq\cdots$ of elements of $\cU$ such that $\displaystyle \bigcap_{m=1}^{\infty}C_m=\emptyset.$
For each $N\in \mathbb{N}$, define 
\[
B_N
=
C_N
\cap
\bigcap_{1\leq n\leq N}
\bigcap_{1\leq j\leq n}
\bigcap_{1\leq \ell\leq N}
A_{n,j,\ell}\in \cU.
\]
Then $B_{N+1}\subseteq B_N$, and $\bigcap_{N=1}^\infty B_N=\emptyset .$
Define a function $\nu:I\to\mathbb N$ by 
$
\nu(i)=N$ if $i\in B_N\setminus B_{N+1}$ 
and $\nu(i)=1$ if $i\notin B_1$. Then $
\displaystyle \lim_{i\to \cU}\nu(i)=\infty.
$ 
For every fixed $j$, 
$\|e_i^{(\nu(i),j)}\|\le K+\|b^{(j)}\|,$ so that $(e_i^{(\nu(i),j)})_i$ is uniformly bounded. Moreover, we have  
$
\bigl(e_i^{(\nu(i),j)}\bigr)_i\in\mathcal \cI_{\cU}.
$
Indeed, fix $j,\ell\geq1$ and put $r=\max\{j,\ell\}$.
For $i\in B_r$, set $N=\nu(i)$. Then $N\geq r$ and
$i\in B_N\subseteq A_{N,j,\ell}$. Hence
\[
 \|e_i^{(\nu(i),j)}\|_{\varphi_i}^{\#}<1/\ell
 \qquad(i\in B_r).
\]
Since $B_r\in\mathcal U$, the required ultralimit is zero.
Fix $z\in \Omega$.  Write
$
H^{(n)}(z)=(H_i^{(n)}(z))_i\in\mathcal M^{\cU}.
$ 
Define $F_i(z):=H_i^{(\nu(i))}(z)$
and set $F(z):=(F_i(z))_i$. Since $\sup_n\|H^{(n)}\|_\infty\leq K$, we have
\begin{equation}
\sup_{i\in I}\sup_{z\in\Omega}\|F_i(z)\|\leq K.\label{eq F is bounded}
\end{equation}
Thus $F(z)\in \ell^\infty(\bm{M})$ for every $z\in\Omega$ and $\sup_{z\in\Omega}\|F(z)\|_{\ell^\infty}\leq K.$ 

We claim that $F:\Omega\to\ell^\infty(\bm{M})$ is norm-continuous.  Fix
$\zeta_0\in\Omega$.  Choose $r>0$ such that
$\overline{D(\zeta_0,2r)}\subseteq\Omega.
$ For every $\zeta\in D(\zeta_0,r)$, the circle $|\eta-\zeta|=r$ 
is contained in $\Omega$.  By the Banach-space valued Cauchy formula,
\[
F_i'(\zeta)
=
\frac{1}{2\pi \mathrm{i}}
\int_{|\eta-\zeta|=r}
\frac{F_i(\eta)}{(\eta-\zeta)^2}\,d\eta .
\]
By (\ref{eq F is bounded}), we obtain 
$
\|F_i'(\zeta)\|\leq \frac{K}{r}
$
for every $i\in I$ and $\zeta\in D(\zeta_0,r)$.

If $z,w\in D(\zeta_0,r)$, then the line segment from $w$ to $z$
is contained in $D(\zeta_0,r)$.  Hence
\[
F_i(z)-F_i(w)
=
\int_0^1F_i'(w+t(z-w))(z-w)\,dt,
\]
and therefore
\[
\|F(z)-F(w)\|
=
\sup_{i\in I}\|F_i(z)-F_i(w)\|
\leq
\frac{K}{r}|z-w|.
\]
Thus $F$ is Lipschitz on $D(\zeta_0,r)$.  Since $\zeta_0\in\Omega$ was
arbitrary, $F$ is norm-continuous on $\Omega$.

We next prove that $F$ is norm-holomorphic as an
$\ell^\infty(\bm{M})$-valued function.  Fix $z_0\in\Omega$, and choose
$r>0$ such that
\[
\overline{D(z_0,r)}\subseteq\Omega .
\]
For every $z\in D(z_0,r)$, the function $\xi\mapsto \frac{F(\xi)}{\xi-z}$ 
is continuous as an $\ell^\infty(\bm{M})$-valued function on
$\partial D(z_0,r)$.  Thus the integral
\[
\frac{1}{2\pi\mathrm{i}}
\int_{\partial D(z_0,r)}
\frac{F(\xi)}{\xi-z}\,d\xi
\]
is a Bochner integral in $\ell^\infty(\bm{M})$.

For each $i$, the function
$
F_i(\zeta)=H_i^{(\nu(i))}(\zeta)
$
is holomorphic on $\Omega$.  Therefore, the Banach-space valued Cauchy formula in $M_i$ gives
\[
F_i(z)
=
\frac{1}{2\pi\mathrm{i}}
\int_{\partial D(z_0,r)}
\frac{F_i(\xi)}{\xi-z}\,d\xi .
\]
Since the coordinate evaluation map
$
\pi_i:\ell^\infty(\bm{M})\to M_i
$
is bounded linear, it commutes with Bochner integration.  Hence the
$i$-th coordinate of
\[
\frac{1}{2\pi\mathrm{i}}
\int_{\partial D(z_0,r)}
\frac{F(\xi)}{\xi-z}\,d\xi
\]
is exactly $F_i(z)$.  Since this holds for every $i$, we have the
identity in $\ell^\infty(\bm{M})$:
\[
F(z)
=
\frac{1}{2\pi\mathrm{i}}
\int_{\partial D(z_0,r)}
\frac{F(\xi)}{\xi-z}\,d\xi .
\]
Thus $F$ satisfies the local Cauchy integral formula as an
$\ell^\infty(\bm{M})$-valued function, and hence $F$ is norm-holomorphic.

We next show that $F(z)\in\mathcal M^{\cU}
$ for every $z\in\Omega$.  First, fix $j\in \mathbb{N}$.  On a $\mathcal U$-large set,
$\nu(i)\geq j$, and hence
\[
F_i(\zeta_j)-b_i^{(j)}
=
e_i^{(\nu(i),j)}
+
(\lambda_{\nu(i)}-1)b_i^{(j)}.
\]
The first term belongs to $\cI_{\cU}$.  The
second term also belongs to $\cI_{\cU}$ since $\lim_{i\to \cU}
\lambda_{\nu(i)}=1$
and $b^{(j)}\in\ell^\infty(\bm{M})$.  Therefore $F(\zeta_j)-b^{(j)}\in\cI_{\cU}.$
Since $b^{(j)}\in\mathcal M^{\cU}$ and $\mathcal I\subseteq\mathcal M^{\cU}$,
we get $F(\zeta_j)\in\mathcal M^{\cU}.$ The set $\mathcal M^{\cU}$ is norm-closed in $\ell^\infty(\bm{M})$.  Since
$\{\zeta_j \ : \  j\geq1\}$ is dense in $\Omega$ and since $F$ is
norm-continuous as an $\ell^\infty(\bm{M})$-valued function, it follows that
$F(z)\in\mathcal M^{\cU}$
for every $z\in\Omega$.  Hence 
$F\in H^\infty(\Omega,\mathcal M^{\cU})$ and $\|F\|_\infty\leq K.$  

Finally, for every $j\in \mathbb{N}$, $F(\zeta_j)-b^{(j)}\in\mathcal I_{\cU}$ and \[q(F(\zeta_j))=q(b^{(j)})=f(\zeta_j).\]
The two $M$-valued holomorphic functions $q\circ F$ 
and $f$ thus agree on the dense set $\{\zeta_j \ : \  j\geq1\}$. By the Banach-space valued identity
theorem, $
q\circ F=f$. Since $q$ is contractive, $q\circ F=f$, and $\|f\|_{\infty}=K$, we have $\|F\|_{\infty}=K$.
This proves the theorem.
\end{proof}
\subsubsection{First Gaussian Regularization}
\begin{lem}\label{lem:norm-continuity-integrand}
Let $\rho>1/2,\,\tau>0$, and $G\in H^{\infty}(\mathbb{S}_{\rho},\cM^{\cU})$.  For $z\in \mathbb{S}_{\rho}$, write 
\[
G(z)=(G_i(z))_i.
\]
For $i\in I$ and $z\in \mathbb{S}_{\rho}$, define $G_i^\tau(z)=\sigma_{k_{\tau}}^{\varphi_i}(G_i(z))$. 
For fixed $z$ with $|{\operatorname{Im}z}\,|\leq1/2$ and $u\in \mathbb{R}$, set
\[
E_i^z(u)=\sigma_{-u}^{\varphi_i}(G_i^\tau(u+z)).
\]
Then $E^z(u):=(E_i^z(u))_i\in\cM^{\cU}$ 
for every $u\in \mathbb{R}$, and $u\mapsto E^z(u)$ is norm-continuous as an $\cM^{\cU}$-valued function which further satisfies 
\begin{equation}\|E^z(u)\|\le \|G\|_{\infty}\label{eq: Ez}
\end{equation}
for all $u\in \mathbb{R}$.
\end{lem}

\begin{proof}
For each $u$, $G^\tau(u+z):=(G^{\tau}_i(u+z))_{i\in I}\in\cM^{\cU}$ by Corollary \ref{cor:heat-descends}. Then because $\cM^{\cU}$ is invariant under $(\sigma_{-u}^{\varphi_i})_i$ \cite[Theorem 4.1]{AH}, we obtain $E^z(u)\in\cM^{\cU}$.

Rewrite
\[
\begin{aligned}
E_i^z(u)
&=
\sigma_{-u}^{\varphi_i}
\left(
\int_\mathbb{R} k_\tau(t)\sigma_t^{\varphi_i}(G_i(u+z))\,dt
\right)\\
&=
\int_\mathbb{R} k_\tau(t)\sigma_{t-u}^{\varphi_i}(G_i(u+z))\,dt\\
&=
\int_\mathbb{R} k_\tau(v+u)\sigma_v^{\varphi_i}(G_i(u+z))\,dv.
\end{aligned}
\]
Let $u,u'\in \mathbb{R}$.  Since each real modular automorphism is isometric,
\[
\begin{aligned}
\|{E_i^z(u)-E_i^z(u')}\|
&\leq
\|{G}\|_\infty
\|k_\tau(\cdot+u)-k_\tau(\cdot+u')\|_{L^1(\mathbb{R})}\\
&\quad+
\|G_i(u+z)-G_i(u'+z)\|.
\end{aligned}
\]
Taking the supremum over $i$,
\[
\|E^z(u)-E^z(u')\|
\leq
\|G\|_\infty
\|k_\tau(\cdot+u)-k_\tau(\cdot+u')\|_{1}
+\|G(u+z)-G(u'+z)\|.
\]
The first term tends to $0$ as $u'\to u$ because translations are continuous in $L^1$, and the second tends to $0$ because $G$ is norm-continuous.  The last inequality ($\ref{eq: Ez}$) is immediate. 
\end{proof}
\subsubsection{Gaussian averaging and equivariant lifting}
Next, we prove that if the bounded holomorphic map $F\colon \mathbb{S}_{\rho}\to M$ is $\sigma^{\varphi}$-equivariant, then we may find an equivariant holomorphic lifting of the Gaussian-regularized version of $F$. 
\begin{prop}
\label{prop:approx-equiv-lift}
Assume that $\mathcal U$ is countably incomplete. Fix $\rho>\frac12,\,K>0$ and take $F\in H^{\infty}(\mathbb{S}_{\rho},M)$
 such that $\|F\|_{\infty}\le K$.  Further suppose that $F$ is $\sigma^{\varphi}$-equivariant in the sense that, for all $z\in \mathbb{S}_{\rho}$ and $t\in \mathbb{R}$, we have
\[
F(z+t)=\sigma_t^\varphi(F(z)).
\]
Then for every $\varepsilon>0$ and every $\tau>0$, there exist functions $H_i:\overline{\mathbb{S}_{\frac12}}\to M_i$

which are holomorphic in the interior and continuous on the closed strip,
such that:
\begin{enumerate}
\item for each $i\in I$, $H_i$ is $\sigma_t^{\varphi_i}$-equivariant, that is, for all $z\in \overline{\mathbb{S}_{\frac{1}{2}}}$ and $t\in \mathbb{R}$, we have
\[
H_i(z+t)=\sigma_t^{\varphi_i}(H_i(z));
\]
\item
\[
\sup_{i\in I}\sup_{z\in \overline{\mathbb{S}_{\frac{1}{2}}}}\|H_i(z)\|
\leq K+\varepsilon;
\]

\item for all $z\in \overline{\mathbb{S}_{\frac12}}$, we have
\[
(H_i(z))_{\cU}=\sigma_{k_{\tau}}^{\varphi}(F(z)).
\]
\end{enumerate}
\end{prop}

\begin{proof}
Choose $a>0$ so small that $e^{a/4}K<K+\varepsilon.$ Put
\[
\gamma_a(w)=\left(\frac a\pi\right)^{1/2}e^{-aw^2}.
\]

By Theorem \ref{thm:ocneanu-Hinfty-lifting},
there exists $G\in H^\infty(\mathbb S_\rho,\cM^{\cU})$ with $\|G\|_\infty\leq K$ such that $q\circ G=F$. 
For each $w\in \mathbb{S}_{\rho}$, write $G(w)=(G_i(w))_i$. 
For each $i\in I$, $z\in \overline{\mathbb{S}_{\frac12}}$, and $u\in \mathbb{R}$, define
\[
E_i^z(u)
=
\sigma_{-u}^{\varphi_i}\circ \sigma_{k_{\tau}}^{\varphi_i}\bigl(G_i(u+z)\bigr).
\]
By Lemma \ref{lem:norm-continuity-integrand}, we obtain the norm-continuous map $E^z\colon \mathbb{R}\to \cM^{\cU}$ defined by $
E^z(u)=(E_i^z(u))_i$. Moreover, for each $u\in \mathbb{R}$, we have
\[
\|E^z(u)\|\leq \|G\|_\infty\leq K.
\]
Since $\mathbb{R}\ni u\mapsto |\gamma_a(u+z)|\in \mathbb{R}$ is integrable, the function $$
\mathbb{R}\ni u\mapsto \gamma_a(u+z)E^z(u)\in \cM^{\cU}$$
is Bochner integrable.  Define $H\colon \overline{\mathbb{S}_{\frac12}}\to \cM^{\cU}$ by 
\[
H(z)
=
\int_{\mathbb R}\gamma_a(u+z)E^z(u)\,du
\in \cM^{\cU}.
\]
Write $H(z)=(H_i(z))_i$ for each $z\in \overline{\mathbb{S}_{\frac12}}.$ 
Since the coordinate evaluation maps $\pi_i\colon \ell^{\infty}(\bm{M})\to M_i$ are bounded linear maps, they commute with
Bochner integration, and hence, for all $i\in I$ and $z\in \overline{\mathbb{S}_{\frac12}}$, we have
\[
H_i(z)
=
\int_{\mathbb R}
\gamma_a(u+z)
\sigma_{-u}^{\varphi_i}\circ \sigma_{k_{\tau}}^{\varphi_i}\bigl(G_i(u+z)\bigr)\,du.
\]

Fix $i\in I$. For $t\in\mathbb R,\,z\in \overline{\mathbb{S}_{\frac12}}$, changing variables
$v=u+t$ gives
\[
\begin{aligned}
H_i(z+t)
&=
\int_{\mathbb R}
\gamma_a(u+z+t)
\sigma_{-u}^{\varphi_i}
\bigl(\sigma_{k_{\tau}}^{\varphi_i}( G_i(u+z+t))\bigr)\,du  \\
&=
\int_{\mathbb R}
\gamma_a(v+z)
\sigma_{-(v-t)}^{\varphi_i}
\bigl(\sigma_{k_{\tau}}^{\varphi_i}(G_i(v+z))\bigr)\,dv  \\
&=
\sigma_t^{\varphi_i}
\left(
\int_{\mathbb R}
\gamma_a(v+z)
\sigma_{-v}^{\varphi_i}
\bigl(\sigma_{k_{\tau}}^{\varphi_i}(G_i(v+z))\bigr)\,dv
\right)  \\
&=
\sigma_t^{\varphi_i}(H_i(z)).
\end{aligned}
\]
This shows (1). Next, for $z\in \overline{\mathbb{S}_{\frac12}}$,
\[
\begin{aligned}
\|H_i(z)\|
&\leq
\int_{\mathbb R}|\gamma_a(u+z)|
\|\sigma_{k_{\tau}}^{\varphi_i}(G_i(u+z))\|\,du  \\
&\leq
K\int_{\mathbb R}|\gamma_a(u+z)|\,du.
\end{aligned}
\]
Writing $z=x+\mathrm{i}y$, one has
\[
\int_{\mathbb R}|\gamma_a(u+z)|\,du
=
e^{ay^2}\le e^{a/4}.
\]
This proves (2):
\[
\sup_{i\in I}\sup_{|\operatorname{Im}z|\leq1/2}\|H_i(z)\|
\leq e^{a/4}K<K+\varepsilon.
\]
Since $q$ is bounded linear and $H(z)$ is a
Bochner integral in $\mathcal M^{\cU}$, we have 
\[
q(H(z))
=
\int_{\mathbb R}\gamma_a(u+z)q(E^z(u))\,du.
\]
Let $u\in \mathbb{R}$ and $z\in \overline{\mathbb{S}_{\frac12}}$. 
By Corollary \ref{cor:heat-descends} and \cite[Lemma 4.14]{AH}, 
\[
q(E^z(u))
=
\sigma_{-u}^{\varphi}
\bigl(\sigma_{k_{\tau}}^{\varphi}( q(G(u+z)))\bigr)=\sigma_{-u}^{\varphi}
\bigl(\sigma_{k_{\tau}}^{\varphi}( F(u+z))\bigr).
\]
Using the $\sigma^{\varphi}$-equivariance of $F$, we obtain
\[
\sigma_{-u}^{\varphi}
\bigl(\sigma_{k_\tau}^{\varphi}(F(u+z))\bigr)
=
\sigma_{-u}^{\varphi}
\bigl(\sigma_{k_\tau}^{\varphi}(\sigma_u^\varphi(F(z)))\bigr)
=
\sigma_{k_\tau}^{\varphi}(F(z)).
\]
Note also that the identity
$
\int_{\mathbb R}\gamma_a(u+z)\,du=1
$
follows by shifting the contour. 
Therefore, 
\begin{align*}
(H_i(z))_{\cU}&=q(H(z))
=
\left(\int_{\mathbb R}\gamma_a(u+z)\,du\right)\sigma_{k_\tau}^{\varphi}(F(z))\\
&=\sigma_{k_{\tau}}^{\varphi}(F(z)).
\end{align*}

Finally, we prove that each $H_i$ is holomorphic in the open strip and
continuous on the closed strip. Fix $i\in I$. 
For $u\in\mathbb R$ and $z\in \overline{\mathbb S_{\frac12}}$, put
\[
\Phi_i(u,z)
=
\gamma_a(u+z)\,
\sigma_{-u}^{\varphi_i}
\Bigl(\sigma_{k_\tau}^{\varphi_i}(G_i(u+z))\Bigr).
\]
Thus
\[
H_i(z)=\int_{\mathbb R}\Phi_i(u,z)\,du.
\]
Since $\rho>1/2$, we have $
u+z\in \mathbb S_\rho.$ 
Moreover, $G_i$ is norm-holomorphic on $\mathbb S_\rho$. Hence, for each
fixed $u\in\mathbb R$, the map $
z\mapsto \Phi_i(u,z)$ is norm-holomorphic on $\mathbb S_{1/2}$ and norm-continuous on $\overline{\mathbb S_{1/2}}$.

We record the local domination needed below. Let $A>0$. If
$|\operatorname{Re}z|\leq A$ and $|\operatorname{Im}z|\leq 1/2,$
then, writing $z=x+\mathrm i y$, we have
\[
|\gamma_a(u+z)|
=
\left(\frac a\pi\right)^{1/2}
e^{-a((u+x)^2-y^2)}
\leq
\left(\frac a\pi\right)^{1/2}
e^{a/4}e^{-a(u+x)^2}.
\]
Since
$
(u+x)^2\geq \frac{u^2}{2}-x^2
\geq \frac{u^2}{2}-A^2,
$
it follows that
\[
|\gamma_a(u+z)|
\leq
\left(\frac a\pi\right)^{1/2}
e^{a/4+aA^2}e^{-a u^2/2}.
\]
Consequently,
\[
\|\Phi_i(u,z)\|
\leq
\left(\frac a\pi\right)^{1/2}
e^{a/4+aA^2}K e^{-a u^2/2}.
\]
The right-hand side belongs to $L^1(\mathbb R)$.

We first prove continuity on the closed strip. Let
$
z_n\to z_0$ in  $\overline{\mathbb S_{1/2}}.$
Choose $A>0$ such that
$
|\operatorname{Re}z_n|\leq A$ and $|\operatorname{Re}z_0|\leq A
$ 
for all sufficiently large $n$. For each fixed $u\in\mathbb R$, the
norm-continuity of $G_i$, the contractivity of $\sigma_{k_{\tau}}^{\varphi_i}$, and the
isometry of $\sigma_{-u}^{\varphi_i}$ give $\Phi_i(u,z_n)\longrightarrow \Phi_i(u,z_0)$, where the convergence is norm convergence.
By the dominated convergence theorem for the Bochner integration, we have 
\[
\|H_i(z_n)-H_i(z_0)\|
\leq
\int_{\mathbb R}
\|\Phi_i(u,z_n)-\Phi_i(u,z_0)\|\,du
\longrightarrow 0.
\]
Hence $H_i$ is norm-continuous on $\overline{\mathbb S_{1/2}}$.

Next, we prove the holomorphy in the open strip. Let
$D\subseteq \mathbb S_{1/2}$ be a relatively compact open set. Choose
$A>0$ such that $
|\operatorname{Re}z|\leq A$ holds for all $z\in D$.
The preceding $L^1$-majorant is then uniform for $z\in D$. Hence, for
every triangle $\Gamma$ whose closure is contained in $D$, Fubini's theorem
and the norm-holomorphy of $z\mapsto \Phi_i(u,z)$ give
\[
\int_{\partial\Gamma}H_i(z)\,dz
=
\int_{\partial\Gamma}
\left(\int_{\mathbb R}\Phi_i(u,z)\,du\right)dz
=
\int_{\mathbb R}
\left(\int_{\partial\Gamma}\Phi_i(u,z)\,dz\right)du
=
0.
\]
Since $H_i$ is norm-continuous on $D$, the Banach-space valued Morera theorem
implies that $H_i$ is norm-holomorphic on $D$. Since $D\subseteq
\mathbb S_{1/2}$ was arbitrary, $H_i$ is norm-holomorphic on
$\mathbb S_{1/2}$.
\end{proof}
\subsection{Second proof of Theorem~\ref{thm:SK-liftingnew}}\label{sec definability of SK}
We are now ready to give an alternative proof of Theorem~\ref{thm:SK-liftingnew}.  

\begin{proof}[Second Proof of Theorem~\ref{thm:SK-liftingnew} when $\cU$ is countably incomplete]
Let $K>0$ and $x\in S_K(M,\varphi)$. 
For $s>0$, define
$
x_{s}:=\sigma_{k_{s}}^{\varphi}(x).
$
Then by $\|k_s\|_1=1$ and Lemma \ref{lem:smoothingpreserveSK}, each $x_{s}\in S_K(M,\varphi)$
and because $(k_{s})_{s>0}$ is an approximate identity, we have \[\lim_{s\to +0}\|x_{s}-x\|_{\varphi}^{\#}=0.\]
For each $s>0$, define
$
F_{s}(z):=\sigma_z^\varphi(x_{s}).
$ 
Then $F_{s}$ is bounded and norm-holomorphic on the open strip $\mathbb S_{\frac12+s}$ 
and is $\sigma^{\varphi}$-equivariant:
\[
F_{s}(z+t)=\sigma_t^\varphi(F_{s}(z)),\qquad z\in \mathbb S_{\frac12+s},\,t\in \mathbb{R}.
\]

We claim that
$
\|F_s\|_\infty\leq e^{s/4}K.
$
Indeed, on the upper boundary $\operatorname{Im}z=1/2+s$, we use (\ref{eq heat regularize}) in Lemma \ref{lem:gauss-total-bounded}, and the fact that $\sigma_{k_s}^{\varphi}$ and $\sigma_{\ri/2}^{\varphi}$ commute to obtain 
\begin{align*}
\|\sigma_z^{\varphi}(x_{s})\|&=\|
\sigma_{\mathrm{i}(1/2+s)}^\varphi(x_s)\|
=\|\sigma_{\ri s}^{\varphi}(\sigma_{\ri /2}^{\varphi}(x_{s}))\|\\
&\le e^{s/4}\|\sigma_{\mathrm{i}/2}^\varphi(x)\|\\
& 
\leq e^{s/4}K.
\end{align*}
The lower boundary case is analogous, using $\|\sigma_{-\mathrm{i}/2}^\varphi(x)\|\le K$ instead. 
The Banach-space valued three-lines theorem gives
$
\|F_{s}\|_\infty\leq e^{s/4}K
$ on $\mathbb S_{1/2+s}$.

For each $m\geq1$, choose $s_m>0$ and $\tau_m>0$ so small so that
\[
\|\sigma_{k_{\tau_m}}^{\varphi}(\sigma_{k_{s_m}}^{\varphi}(x))-x\|_\varphi^\#<\frac1m,\qquad 
e^{s_m/4}K<K+\frac{1}{2m}.
\]
Apply Proposition~\ref{prop:approx-equiv-lift} to $F_{s_m}\in H^{\infty}(\mathbb S_{1/2+s_m},M)$ 
with $K$ replaced by $e^{s_m/4}K$ and with the parameters $\tau=\tau_m$ and $\varepsilon=1/(2m)$.
We obtain functions $H_i^{(m)}\colon \overline{\mathbb{S}_{\frac12}}\to M_i$ 
which are continuous, holomorphic in the interior, 
such that
\[
H_i^{(m)}(z+t)=\sigma_t^{\varphi_i}(H_i^{(m)}(z)),
\]
\[
\sup_{i\in I}\sup_{z\in \overline{\mathbb{S}_{\frac{1}{2}}}}
\|H_i^{(m)}(z)\|
\leq K+\frac1m,
\]
and $(H_i^{(m)}(z))_{\cU}
=
\sigma_{k_{\tau_m}}^{\varphi}(F_{s_m}(z))$ 
for every $z\in \overline{\mathbb{S}_{\frac{1}{2}}}$.

Set $y_i^{(m)}:=H_i^{(m)}(0).$ 
By the covariance of $H_i^{(m)}$, for all $t\in \mathbb{R}$, we have
\[
H_i^{(m)}(t)
=
\sigma_t^{\varphi_i}(H_i^{(m)}(0))
=
\sigma_t^{\varphi_i}(y_i^{(m)}).
\]
Moreover, $H_i^{(m)}$ is bounded and continuous on the closed strip
$\overline{\mathbb S_{1/2}}$ and
holomorphic in the interior. Hence $H_i^{(m)}$ is an analytic 
extension of $t\mapsto \sigma_t^{\varphi_i}(y_i^{(m)})$.  Therefore $y_i^{(m)}
\in
D(\sigma_{\mathrm i/2}^{\varphi_i})
\cap
D(\sigma_{-\mathrm i/2}^{\varphi_i})$, 
and its boundary values are
\[
\sigma_{\mathrm i/2}^{\varphi_i}(y_i^{(m)})
=
H_i^{(m)}(\mathrm i/2),
\qquad
\sigma_{-\mathrm i/2}^{\varphi_i}(y_i^{(m)})
=
H_i^{(m)}(-\mathrm i/2).
\]
The norm estimate on $H_i^{(m)}$ gives
\[
\|y_i^{(m)}\|\leq K+\frac1m,
\]
\[
\|\sigma_{\mathrm{i}/2}^{\varphi_i}(y_i^{(m)})\|\leq K+\frac1m,
\qquad
\|\sigma_{-\mathrm{i}/2}^{\varphi_i}(y_i^{(m)})\|\leq K+\frac1m.
\]
By Lemma \ref{lem:totally-bounded-char}, we have 
$
y_i^{(m)}\in S_{K+1/m}(M_i,\varphi_i)$ for $i\in I.
$\\
At $z=0$, the quotient equality gives
\[
(y_i^{(m)})_{\cU}
=
\sigma_{k_{\tau_m}}^{\varphi}(F_{s_m}(0))
=
\sigma_{k_{\tau_m}}^{\varphi}(\sigma_{k_{s_m}}^{\varphi}(x)).
\]
Hence $\|(y_i^{(m)})_{\cU}-x\|_\varphi^\#<\frac1m$. 
Choose any $(a_i)_{i}\in \cM^{\cU}$ such that $x=(a_i)_{\cU}$. 
For each $m\in \mathbb{N}$, define
\[
I_m=
\left\{
i\in I \ \middle| \ 
\|y_i^{(m)}-a_i\|_{\varphi_i}^{\#}<\frac{2}{m}
\right\}\in \cU.
\]
Since $\mathcal U$ is countably incomplete, choose a decreasing sequence $C_1\supseteq C_2\supseteq\cdots$ of elements in $\cU$
 such that
$\bigcap_{m=1}^\infty C_m=\emptyset.$
For $m\in \mathbb{N}$, set 
\[J_m=C_m\cap I_1\cap\cdots\cap I_m.\] 
Then  
$J_1\supseteq J_2\supseteq\cdots$, $J_m\subseteq I_m$ for each $m\in \mathbb{N}$, and $\bigcap_m J_m=\emptyset.$ 
For each $i\in I$, define 
\[
z_i=\begin{cases} \frac{K}{K+1/m}y_i^{(m)} & i\in J_m\setminus J_{m+1},\\
\qquad 0 & i\notin J_1.\end{cases}.
\]
Then $z_i\in S_K(M_i,\varphi_i)$ for every $i$, and thus by Lemma~\ref{represents}, $(z_i)_i\in \cM^{\cU}$ holds. We show that $(z_i)_{\cU}=x$. 
Along $\mathcal U$, the integer $m$ determined by $i\in J_m\setminus J_{m+1}$ 
tends to infinity. For such $i$,
\[
\begin{aligned}
\|z_i-a_i\|_{\varphi_i}^{\#}
&\leq
\|y_i^{(m)}-a_i\|_{\varphi_i}^{\#}
+
\left|1-\frac{K}{K+1/m}\right|
\|y_i^{(m)}\|_{\varphi_i}^{\#}  \\
&\leq
\frac{2}{m}
+
\frac{1/m}{K+1/m}\left(K+\frac1m\right)  =
\frac{3}{m}.
\end{aligned}
\]
Therefore
$
\lim_{i\to \mathcal U}\|z_i-a_i\|_{\varphi_i}^{\#}=0,
$ and $(z_i)_{\cU}=(a_i)_{\cU}=x.
$
This proves the theorem when $\cU$ is a countably incomplete ultrafilter. 
\end{proof}
\bibliographystyle{siam}
\bibliography{reference}
\end{document}